\DeclareSymbolFont{AMSb}{U}{msb}{m}{n}
\documentclass[reqno, a4paper, noamsfonts, 11pt]{amsart}
\usepackage[utf8]{inputenc}
\usepackage[T1]{fontenc}
\usepackage[english]{babel}
\usepackage{mathtools}
\usepackage[usenames,dvipsnames]{xcolor}
\usepackage{graphicx}
\usepackage{enumitem}
\usepackage{hyperref}

\usepackage{bm}
\usepackage{dsfont}
\usepackage{lmodern}
\usepackage{microtype}
\usepackage{algorithm}
\usepackage{algpseudocode}
\definecolor{darkblue}{rgb}{0,0,0.6}
\hypersetup{
    pdftitle={Poincar\'e and logarithmic Sobolev constants for metastable Markov chains via capacitary inequalities},    
    pdfauthor={André Schlichting and Martin Slowik},
    pdfsubject={},
    pdfkeywords={
      capacitary inequality, harmonic functions, logarithmic Sobolev constant, 
      mean hitting time, metastability, Poincar\'e constant, random field 
      Curie-Weiss model},
    colorlinks=true,
    linkcolor=blue,
    citecolor=blue,
    filecolor=blue,
    urlcolor=blue
}

\usepackage[text={33pc,605pt},centering]{geometry}    
\usepackage[charter,expert]{mathdesign}
\DeclarePairedDelimiter{\abs}{\lvert}{\rvert}
\DeclarePairedDelimiter{\norm}{\lVert}{\rVert}
\DeclarePairedDelimiter{\bra}{(}{)}
\DeclarePairedDelimiter{\pra}{[}{]}
\DeclarePairedDelimiter{\set}{\{}{\}}
\DeclarePairedDelimiter{\skp}{\langle}{\rangle}
\newcommand{\dx}{\textup{d}}
\newcommand{\CI}{{\textup{CI}}}

\newcommand{\PI}{{\textup{PI}}}
\newcommand{\eps}{\varepsilon}

\newcommand{\magmu}{\ensuremath{\boldsymbol{\mu}\mspace{3mu}}}
\newcommand{\me}{\ensuremath{\mathrm{e}}}
\newcommand{\one}{\mathds{1}}
\newcommand{\dsOne}{\mathds{1}}

\newcommand{\N}{\mathds{N}}

\newcommand{\R}{\mathds{R}}

\newcommand{\Z}{\mathds{Z}}
\newcommand{\cA}{\ensuremath{\mathcal A}}
\newcommand{\cB}{\ensuremath{\mathcal B}}

\newcommand{\cE}{\ensuremath{\mathcal E}}
\newcommand{\cF}{\ensuremath{\mathcal F}}
\newcommand{\cG}{\ensuremath{\mathcal G}}

\newcommand{\cM}{\ensuremath{\mathcal M}}
\newcommand{\cN}{\ensuremath{\mathcal N}}

\newcommand{\cP}{\ensuremath{\mathcal P}}

\newcommand{\cR}{\ensuremath{\mathcal R}}
\newcommand{\cS}{\ensuremath{\mathcal S}}

\newcommand{\cU}{\ensuremath{\mathcal U}}

\newcommand{\bfA}{\ensuremath{\bm A}}
\newcommand{\bfa}{\ensuremath{\bm a}}
\newcommand{\bfB}{\ensuremath{\bm B}}

\newcommand{\bfm}{\ensuremath{\bm m}}
\newcommand{\bfM}{\ensuremath{\bm M}}
\newcommand{\bfr}{\ensuremath{\bm r}}

\newcommand{\bfx}{\ensuremath{\bm x}}
\newcommand{\bfy}{\ensuremath{\bm y}}
\newcommand{\bfz}{\ensuremath{\bm z}}
\newcommand{\order}{\ensuremath{\boldsymbol{\rho}}}
\newcommand{\EX}[1][E]{\ensuremath {\mathds{#1}}}
\newcommand{\valley}[1][V]{\ensuremath {\mathcal{#1}}}

\DeclareMathOperator{\capacity}{cap}
\DeclareMathOperator{\bcapacity}{\mathbf{cap}}

\DeclareMathOperator{\nBer}{nBer}
\DeclareMathOperator{\Cheeger}{Cheeger}
\DeclareMathOperator{\supp}{supp}

\DeclareMathOperator{\mean}{\mathds{E}}
\DeclareMathOperator{\Mean}{\mathrm{E}}
\DeclareMathOperator{\prob}{\mathds{P}\mspace{-1mu}}
\DeclareMathOperator{\var}{Var}
\DeclareMathOperator{\Ent}{Ent}
\DeclareMathOperator{\LSI}{LSI}

\DeclareMathOperator{\BL}{BL}

\newcommand{\ldef}{\ensuremath{\mathrel{\mathop:}=}}
\newcommand{\rdef}{\ensuremath{=\mathrel{\mathop:}}}

\newcommand{\overbar}[1]{\mkern 1.5mu\overline{\mkern-1.5mu#1\mkern-1.5mu}\mkern 1.5mu}

\newtheorem{thm}{Theorem}[section]
\newtheorem{cor}[thm]{Corollary}
\newtheorem{lem}[thm]{Lemma}
\newtheorem{prop}[thm]{Proposition}

\newtheorem{assume}[thm]{Assumption}
\theoremstyle{definition}
\newtheorem{defn}[thm]{Definition}

\newtheorem{ex}[thm]{Example}

\theoremstyle{remark}
\newtheorem{rem}[thm]{Remark}

\numberwithin{equation}{section}
\begin{document}

\title[Poincar\'e and log-Sobolev constants for metastable Markov chains]{Poincar\'e and logarithmic Sobolev constants for metastable Markov chains via capacitary inequalities}

\author{André Schlichting}
\address{Institut f\"ur Angewandte Mathematik, Universit\"at Bonn}
\curraddr{Endenicher Allee 60, 53115 Bonn, Germany}
\email{schlichting@iam.uni-bonn.de}

\author{Martin Slowik}
\address{Institut f\"ur Mathematik, Technische Universit\"{a}t Berlin}
\curraddr{Strasse des 17. Juni 136, 10623 Berlin}
\email{slowik@math.tu-berlin.de}



\date{\today}

\begin{abstract}
  We investigate the metastable behavior of reversible Markov chains on possibly countable infinite state spaces.  Based on a new definition of metastable Markov processes, we compute precisely the mean transition time between metastable sets.  Under additional size and regularity properties of metastable sets, we establish asymptotic sharp estimates on the Poincar\'e and logarithmic Sobolev constant.  The main ingredient in the proof is a capacitary inequality along the lines of V.~Maz'ya that relates regularity properties of harmonic functions and capacities.  We exemplify the usefulness of this new definition in the context of the random field Curie--Weiss model, where metastability and the additional regularity assumptions are verifiable. 
\end{abstract}

\maketitle

\tableofcontents

\section{Introduction}
%
%
%
%
Metastability is a \emph{dynamical} phenomenon that is characterized by the existence of multiple, well-separated time scales.  Depending on the time scales under consideration, the state space can be decomposed into several disjoint subsets (\emph{metastable partition}) with the property that typical transition times between different subsets are long compared to characteristic mixing times within each subset.

For a rigorous mathematical analysis of metastable Markov processes, various methods have been invented.  The \emph{pathwise approach} \cite{CGOV84, OV05} based on large deviation methods in path space \cite{FW98} has been proven to be robust and somewhat universally applicable.  While it yields detailed information, for example, on the typical exit path, its precision to predict quantities of interest like the mean transition time is, however, limited to logarithmic equivalence.  For reversible systems, the \emph{potential theoretic approach} \cite{BEGK01, BEGK04, BdH15} has been developed to establish sharp estimates on the mean transition time and the low-lying eigenvalues and to prove that the transition times are asymptotically exponential distributed.  A crucial ingredient of this concept is to express probabilistic quantities of interest in terms of capacities and to use variational principles to compute the latter.  For metastable Markov processes in which the expected transition times for a large number of subsets is of the same order, the \emph{martingale approach} \cite{BL15} has recently been developed to identify the limiting process on the time scale of the expected transition times as a Markov process via the solution of a martingale problem.

In the context of Markov processes, there is also a spectral signature of metastability.  Since the transition probabilities between different subsets of the metastable partition are extremely small, an irreducible Markov process exhibiting a metastable behavior can be seen as a perturbation of the reducible version of it in which transitions between different subsets of the metastable partition are forbidden.  For the reducible version, the theorem of Perron--Frobenius implies that the eigenvalue zero of the associated generator is degenerate with multiplicity given by the number of elements in the metastable partition.  In particular, the corresponding eigenfunctions are given as indicator functions on these subsets.  Provided that the perturbation is sufficiently small, the generator of the original process reveals typically a cluster of small eigenvalues that is separated by a gap from the rest of the spectrum.

The main objective of the present work is to extend the potential theoretic approach to derive sharp asymptotics for the spectral gap and the logarithmic Sobolev constants of metastable Markov chains on countable infinite state spaces.
\smallskip

So far sharp estimates of low-lying eigenvalues have been derived in the following settings:
\begin{enumerate}[label=(\roman{enumi})]
\item For a class of reversible Markov processes on discrete state spaces that are strongly recurrent, in the sense that within each set of the metastable partition there is at least one single point that the process visits with overwhelming probability before leaving the corresponding set of the metastable\linebreak partition.  Based on the potential theoretic approach, sharp estimates on the low-lying eigenvalues and the associated eigenfunctions have been obtained under some additional non-degeneracy conditions in \cite{BEGK02}.  Typical examples of strongly recurrent Markov chains are finite-state Markov processes with exponential small transition probabilities \cite{BD16} and models from statistical mechanics under either Glauber or Kawasaki dynamics in finite volume at very low temperature \cite{BdH15, BM02}.
\item For reversible diffusion processes in a potential landscape in $\R^d$ subject to small noise sharp estimates on the low-lying eigenvalues have been obtained in \cite{BGK05, Su95}.  The proof relies on potential theory and a priori regularity estimates of solutions to certain boundary value problems.  Based on hypo-elliptic techniques and a microlocal analysis of the corresponding Witten-complex a complete asymptotic expansion of the lowest eigenvalues was shown in \cite{HKN04}.  Recently, based on methods of optimal transport, an alternative approach to derive a sharp characterization of the Poincar\'e (inverse of the spectral gap) and the logarithmic Sobolev constants has been developed in \cite{MS14, Sch12}.
\end{enumerate}
A common starting point for rigorous mathematical investigations in the settings described above is the identification of a set of \emph{metastable points} that serves as representatives of the sets in the metastable partition.  For strongly recurrent Markov chains the set of metastable points, $\cM$, is chosen in such a way that, for each $m \in \cM$, the probability to escape from $m$ to the remaining metastable points $\cM \setminus \{m\}$ is small compared with the probability to reach $\cM$ starting at some arbitrary point in the state space before returning to it, cf.\ \cite[Definition~8.2]{BdH15}.  In the context of reversible diffusion processes, metastable points are easy to identify and correspond to local minima of the potential landscape.  Since in dimensions $d > 1$ diffusion processes do not hit individual points $x \in \R^d$ in finite time, each metastable point, $m \in \cM$, has to be enlarged (cf.\ \cite[Definition~8.1]{Bo06}), for example, by replacing each $m \in \cM$ by a small ball $B_{\varepsilon}(m)$.  The radius $\varepsilon > 0$ of such balls should be chosen large enough to ensure that it is sufficiently likely for the process to hit $B_{\varepsilon}(m)$, but simultaneously small enough to control typical oscillations of harmonic functions within these balls.

Once the set of metastable points is identified, the low-lying eigenvalues are characterized in terms of mean exit times for generic situations.  Namely, each low-lying eigenvalue is equal to the inverse of the mean exit time from the corresponding metastable point up to negligible error terms.
\medskip

\noindent
\textbf{Starting ideas.} One would expect that the strategy of enlargements of metastable points that has been successfully used in the diffusion setting, should also apply to stochastic spin systems at finite temperature or in growing volumes.  However, proving general regularity estimates for solutions of elliptic equations is challenging on high dimensional discrete spaces, and so far, highly model dependent.

The present work provides a mathematical definition of metastability for Markov chains on possibly \emph{countable infinite} state spaces (see Definition~\ref{def:meta_sets}), where the metastable points that represent the sets in the metastable partition are replaced by metastable sets.  An advantage of this definition is that one can immediately deduce sharp estimates on the mean exit time to ``deeper'' metastable sets without using additional regularity estimates of harmonic functions; cf.\ Theorem~\ref{thm:mean:hitting:times}.  Moreover, sharp estimates on the smallest non-zero eigenvalues of the generator follow under the natural assumption of good mixing properties within metastable sets and some rough estimates on the regularity of the harmonic function at the boundary of metastable sets.  The primary tool in the proof is the \emph{capacitary inequality}; see Theorem~\ref{thm:CapInequ}.

A critical observation leading to the present definition of metastability is the following:  It is well known that classical Poincar\'e--Sobolev inequalities on $\Z^d$ for functions with compact support, say on a ball $B_r(x) \subset \Z^d$ with radius $r > 0$ and center~$x$, follow from the isoperimetric properties of the underlying Euclidean space by means of the so-called co-area formula.  The isoperimetric inequality states that
\begin{align*}
  |A|^{(d-1)/d} \;\leq\; C_{\mathrm{iso}}\, |\partial A|,
  \qquad \forall\, A \subset B_r(x),
\end{align*}
where $|A|$ and $\partial A$ denotes the cardinality and the boundary of the set $A$. The latter is defined as the set of all points $x \in A$ for which there exists a $y \not\in A$ such that $\{x,y\}$ is an element of the edge set of $\Z^d$.  For a positive recurrent Markov chain with state space $\cS$ and invariant distribution $\mu$ functional inequalities can also be established provided that the isoperimetric inequality is replaced by a measure-capacity inequality; cf.\ Proposition~\ref{prop:OrliczBirnbaum}.  For $B \subset \cS$ and $\Psi\!: \R_+ \to \R_+$ being a convex function, the measure-capacity inequality is given by
\begin{align*}
  \mu[A]\, \Psi^{-1}\bra[\big]{1/\mu[A]} 
  \;\leq\; 
  C_{\Psi}\, \capacity(A, B^c),
  \qquad \forall\, A \subset B.
\end{align*}
Inspired by the form of the measure-capacity inequality, we propose a definition of metastability for Markov chains that also encodes local isoperimetric properties by considering for any subset $A$ outside of the union of the metastable sets its escape probability to the union of the metastable sets.

To demonstrate the usefulness of our approach, we prove sharp estimates on the spectral gap and the logarithmic Sobolev constants for the \emph{random field Curie--Weiss model} at finite temperature and with a continuous bounded distribution of the random field.  To prove rough regularity estimates of harmonic functions, we use a coupling construction initially invented in \cite{BBI12}.

In the present work, we decided to focus only on discrete-time Markov chains to keep the presentation as brief as possible.  However, our methods also apply to Markov chains in continuous time with apparent modifications.
\smallskip

The remainder of this paper is organized as follows.  In the next subsection, we describe the setting to which our methods apply.  In Subsections~\ref{subsec:main_results} and \ref{subsec:main_results:RFCW} we state our main results.  In Section~\ref{sec:FI} we first prove the capacitary inequality for reversible Markov processes.  In particular, we show how this universal estimate allows us to derive so-called Orlicz--Birnbaum estimates from which estimates on the Poincar\'e and logarithmic Sobolev constants can be easily deduced.  Then, we prove our main results in Section~\ref{sec:MetaMC}.  Finally, in Section~\ref{sec:RFCW} we apply the previously developed methods to the random field Curie--Weiss model.

\subsection{Setting}
%
%
%
%
Consider an \emph{irreducible} and \emph{positive recurrent} Markov process $X = (X(t) : t \in \N_0)$ in discrete-time on a countable state space $\cS$ with transition probabilities denoted by $(p(x,y) : x,y \in \cS )$.  For any measurable and bounded function $f\!: \cS \to \R$, define the corresponding (discrete) generator by
\begin{align}\label{eq:generator}
  \big(L f\big)(x)
  \;\ldef\;
  \sum_{y \in \cS}\, p(x,y)\, \big(f(y) - f(x)\big).
\end{align}
Throughout, we assume that the Markov chain is \emph{reversible} with respect to a unique invariant distribution $\mu$.  That is, the transitions probabilities satisfy the \emph{detailed balance condition}
\begin{align}\label{eq:detailedBalance}
  \mu(x)\, p(x,y)
  \;=\;
  \mu(y)\, p(y,x)
  \qquad \text{for all } x, y \in \cS.
\end{align}
We denote by $\prob_{\!\nu}$ the law of the Markov process given that it starts with initial distribution $\nu$.  If the initial distribution is concentrated on a single point $x \in \cS$, we simply write $\prob_x$.  For any $A \subset \cS$, let $\tau_A$ be the first hitting time of the set $A$ \emph{after} time zero, that is
\begin{align*}
  \tau_A
  &\;\ldef\;
  \inf \big\{t > 0 \;:\; X(t) \in A\, \big\}.
\end{align*}
Hence, for $X(0)\in A$, $\tau_A$ is the \emph{first return time} to $A$ and, for $X(0)\not\in A$, $\tau_A$ is the \emph{first hitting time} of $A$.  In case the set $A$ is a singleton $\{x\}$ we write $\tau_x$ instead of $\tau_{\{x\}}$.
\smallskip

We are interested in Markov chains that exhibit a metastable behavior.  For this purpose we introduce the notion of metastable sets.
\begin{defn}[Metastable sets]\label{def:meta_sets}
  For fixed $\varrho > 0$ and $K \in \mathbb{N}$ let $\cM = \set*{M_1, \ldots, M_K}$ be a set of subsets of $\cS$ such that $M_i \cap M_j = \emptyset$ for all $i \ne j$.  A Markov chain $(X(t) : t \geq 0)$ is called $\varrho$-metastable with respect to a \emph{set of metastable sets}~$\cM$, if
  \begin{align}\label{intro:eq:cap:meta_sets}
    \abs{\cM} \ \frac{\max_{M \in \cM}
      \prob_{\mu_M}\!\big[
        \tau_{{\scriptscriptstyle \bigcup_{i=1}^K} M_i \setminus M} < \tau_M
      \big]}
    {
      \min_{A \,\subset \cS \setminus {\scriptscriptstyle \bigcup_{i=1}^K} M_i}
      \prob_{\mu_A}\!\big[\tau_{{\scriptscriptstyle \bigcup_{i=1}^K} M_i} < \tau_A\big]
    }
    \;\leq\;
    \varrho
    \;\ll\;
    1,
  \end{align}
  where $\mu_A(x) = \mu[x \mid A]$, $x \in A \ne \emptyset$ denotes the conditional probability on the set $A$ and $\abs{\cM}$ denotes the cardinality $K$ of $\cM$.
\end{defn}
\begin{rem}
  \begin{enumerate}[label=(\roman{enumi})]
  \item  The definition above is a generalization of the one given in \cite[Definition~8.2 and Remark~8.3]{BdH15} in terms of metastable points.  As it was already pointed out in \cite{Bo04}, the hitting probability of single configurations in high dimensional discrete state spaces or continuous state spaces are either zero or are much smaller than the ones of a small neighborhood around them. Hence, it is necessary to come up with a definition that involves metastable sets.  However, the choice of the sets $\set{M_1,\dots, M_K}$ are typically model dependent.  For instance, in the random field Curie--Weiss model with continuous distribution of the random field each metastable set is defined as the preimage with respect to the mesoscopic magnetization of a local minima of the mesoscopic free energy (see Section~\ref{subsec:main_results:RFCW} for details).  Let us stress the fact that in this model it suffices to only take in account the sufficiently deep minima in order to verifying~\eqref{intro:eq:cap:meta_sets}.

  Further, notice that Definition~\ref{def:meta_sets} does not depend explicitly on the cardinality of the state space $\cS$.  As a consequence, the constant $\varrho$ does not interfere with $|\cS|$.  This makes it possible to apply Definition~\ref{def:meta_sets} for both Markov chains with countable infinite state spaces and for interacting particle systems with state spaces $\cS = \{-1,+1\}^{\Lambda}$, $\Lambda \subset \Z^d$ for which the left-hand side of \cite[equation~(8.1.5)]{BdH15} might be larger than 1.  Typical examples are the random field Curie--Weiss model (cf.~Section~\ref{sec:RFCW}), where $\Lambda = \{1, \ldots, N\}$ with $N \to \infty$ and Ising models with Glauber dynamics at low temperature when $|\Lambda|$ diverges as the temperature tends to zero, cf.\ \cite[Section~19]{BdH15}.

  \item The main novelty of Definition~\ref{def:meta_sets} is the modification of the denominator compared to \cite[equation~(8.1.5)]{BdH15}. The main advantage of this particular form is the fact that estimates on various $\ell^p(\mu)$-norms of harmonic functions can be immediately derived.  This becomes apparent in Theorem~\ref{thm:mean:hitting:times} where sharp estimates on the mean exit time to ``deeper'' metastable sets are proven without using additional regularity and renewal estimates.
    
  \item Notice that if $|\cS| < \infty$ and $|M_i| = 1$ for all $i = 1, \ldots, K$, then the definition of metastability from the potential theoretic literature (see \cite[equation~(8.1.5)]{BdH15}) implies \eqref{intro:eq:cap:meta_sets}.  Since, in this setting, the numerator in both definitions coincides, it suffices to consider the denominator.  In view of \eqref{eq:detailedBalance},
    \begin{align*}
      \mu[A] \prob_{\mu_A}\!\big[\tau_{\!\cM} < \tau_A\big]
      &\;=\;
        \sum_{x \in A} \sum_{m \in \cM} 
        \mu[x] \prob_x\!\big[\tau_{\!\cM} < \tau_{A},\, X(\tau_{\!\cM}) = m\big]
      \\[1ex]
      &\;=\;
        \sum_{x \in A} \sum_{m \in \cM}
        \mu[m] \prob_m\!\big[\tau_{A} < \tau_{\!\cM},\, X(\tau_A) = x\big]
      \\[1ex]
      &\;=\;
        \mu[\cM] \prob_{\mu_{\cM}}\!\big[\tau_{A} < \tau_{\!\cM}\big]
    \end{align*}
    for all non-empty sets $A \subset \cS \setminus \cM$.  Hence,
    \begin{align*}
      \mspace{72mu}
      \prob_{\mu_A}\!\big[\tau_{\cM} < \tau_A\big]
      &\;\geq\; 
      \frac{1}{|A|} \sum_{a \in A} \frac{\mu[\cM]}{\mu[A]} 
      \prob_{\mu_{\cM}}\!\big[\tau_{a} < \tau_{\!\cM}\big]
      \\[1ex]
      &\;=\;
      \frac{1}{|A|} \sum_{a \in A} \frac{\mu[a]}{\mu[A]} 
      \prob_{a}\!\big[\tau_{\!\cM} < \tau_{a}\big]
      \;\geq\;
      \frac{1}{|\cS|} \min_{a \in \cS \setminus \cM} 
      \prob_{a}\!\big[\tau_{\!\cM} < \tau_{a}\big].
    \end{align*}
  \item All hitting probabilities appearing in Definition~\ref{def:meta_sets} can be equivalently expressed in terms of capacities, cf.\ Remark~\ref{rem:meta_sets_cap}.  The verifiability of Definition~\ref{def:meta_sets} relies crucially on the fact that upper and lower bounds on capacities can easily be deduced from their variational characterization.  In order to exemplify the usefulness of this approach, our key example will be the random field Curie--Weiss model with continuous distribution of the random field.
  \end{enumerate}
\end{rem}
\begin{assume}\label{ass:metastability:sets}
  Assume that for some $2 \leq K < \infty$ there exists non-empty, disjoint subsets $M_1, \dots, M_K \subset \cS$ and $\varrho > 0$ such that the Markov chain $(X(t) : t \geq 0)$ is $\varrho$-metastable with respect to $\cM = \set*{M_i : i \in 1, \ldots, K}$.
\end{assume}
The definition of metastable sets induces an almost canonical partition of the state space~$\cS$ into local valleys. 
\begin{defn}[Metastable partition]
  For any $M_i \in \cM$, the \emph{local valley}~$\valley_i$ around the metastable set $M_i$ is defined by
  \begin{align*}
    \valley_i
    \ldef
    M_i \cup \bigg\{
      x \in \cS \setminus {\textstyle \bigcup\limits_{j=1}^K} M_j
      \,:\,
      \prob_{x}\!\big[
        \tau_{M_i} < \tau_{{\scriptscriptstyle \bigcup_{j=1}^K} M_j \setminus M_i}
      \big]
      \geq
      \!\!\max_{M' \in \cM \setminus M_i}\!\!
      \prob_{x}\!\big[
        \tau_{M'} < \tau_{{\scriptscriptstyle \bigcup_{j=1}^K} M_j \setminus M'}
      \big]
    \bigg\}.
  \end{align*}
  A set of metastable sets $\cM = \set*{M_1, \ldots, M_K}$ gives rise to a \emph{metastable partition} $\set*{\cS_i : i = 1, \ldots, K}$ of the state space~$\cS$, that is
  \begin{align*}
    &\text{(i)}\quad M_i \subseteq \cS_i \;\subset\; \valley_i,&
    &\text{(ii)}\quad {\textstyle \bigcup_{i=1}^K} \cS_i \;=\; \cS,&
    &\text{(iii)}\quad \cS_i \cap \cS_j = \emptyset \ \ \text{ for }\ \  i \ne j \ . 
  \end{align*}
\end{defn}
\begin{rem}
  Notice that, by Lemma~\ref{lem:negligible:points}, any point $x \in \valley_i \cap \valley_j$ that lies in the intersection of two different local valleys has a negligible mass compared to the mass of the corresponding metastable sets $\mu[M_i]$ and $\mu[M_j]$.
\end{rem}
The potential theoretic approach to metastability relies on the translation of probabilistic objects to analytic ones, which we now introduce along the lines of~\cite{BEGK01,BEGK02,BEGK04,BBI09,BdH15}.  We simply write $\mu_i[\cdot] \ldef \mu[\,\cdot\, | \cS_i]$ to denote the corresponding conditional measure.  Let $\ell^2(\mu)$ be the weighted Hilbert space of all square summable functions \mbox{$f\!: \cS \to \R$} and denote by $\skp{\cdot,\cdot}_{\mu}$ the scalar product in $\ell^2(\mu)$.  Due to the detailed balance condition~\eqref{eq:detailedBalance} the generator, $L$, is symmetric in $\ell^2(\mu)$, that is $\skp{-L f,g}_{\mu} = \skp{f,-L g}_{\mu}$ for any $g, f \in \ell^2(\mu)$.  The associated \emph{Dirichlet form} is given for any $f \in \ell^2(\mu)$ by
\begin{align*}
  \cE(f)
  \;\ldef\;
  \skp{f,-Lf}_{\mu}
  \;=\;
  \frac{1}{2}\, \sum_{x, y \in \cS}\, \mu(x)\, p(x,y)\, \big(f(x) - f(y)\big)^2,
\end{align*}
which by the basic estimate $\cE(f) \leq \norm{f}_{\ell^2(\mu)}^2$ is well-defined.  Throughout the sequel, let $A, B \subset \cS$ be disjoint and non-empty.  The \emph{equilibrium potential}, $h_{A,B}$, of the pair $(A, B)$ is defined as the unique solution of the boundary value problem
\begin{align}\label{eq:equi:potential}
  \left\{
    \begin{array}{rcll}
      \big(L f\big)(x)
      &\mspace{-5mu}=\mspace{-5mu}& 0,
      &x \in \cS \setminus (A \cup B) \\[1ex]
      f(x)
      &\mspace{-5mu}=\mspace{-5mu}& \one_A(x), \quad 
      &x \in \phantom{(}A \cup B.
    \end{array}
  \right.
\end{align}
Note that the equilibrium potential has a natural interpretation in terms of \emph{hitting probabilities}, namely $h_{A,B}(x) = \prob_{x}[\tau_A < \tau_B]$ for all $x \in \cS \setminus (A \cup B)$.  A related quantity is the \emph{equilibrium measure}, $e_{A,B}$, on $A$ which is defined through
\begin{align}\label{eq:equi:measure}
  e_{A,B}(x)
  \;\ldef\;
  -\big(L h_{A,B} \big)(x)
  \;=\;
  \prob_{x}[\tau_B < \tau_A],
  \qquad \forall\, x \in A.
\end{align}
Clearly, the equilibrium measure is only non-vanishing on the (inner) boundary
of the set~$A$.  Further, the \emph{capacity} of the pair $(A, B)$ with potential one on $A$ and zero on $B$ is defined by
\begin{align}\label{eq:capacity:def}
  \capacity(A,B)
  \;\ldef\;
  \sum_{x \in A}\, \mu(x)\, e_{A,B}(x)
  \;=\;
  \sum_{x \in A}\, \mu(x)\, \prob_{x}[\tau_B < \tau_A]
  \;=\;
  \cE(h_{A,B}).
\end{align}
In particular, we have that
\begin{align}\label{eq:Capacity:HittingTimes}
  \prob_{\mu_A}[\tau_B < \tau_A]
  \;=\;
  \frac{\capacity(A, B)}{\mu[A]},
\end{align}
Moreover, $\capacity(A, B) = \capacity(B, A)$ and, as an immediate consequence of the probabilistic interpretation of capacities, cf.\ \eqref{eq:capacity:def}, we have that
\begin{align}\label{eq:cap:monotonicity}
  \capacity(A, B')
  \;\leq\;
  \capacity(A, B)
  \qquad \forall\, B' \subset B.
\end{align}
Let us emphasize that capacities have several variational characterizations (see for instance~\cite[Chaper 7.3]{BdH15}), which can be used to obtain upper and lower bounds.  One of them is the \emph{Dirichlet principle}
\begin{align}\label{eq:DirichletPrinciple}
  \capacity(A, B)
  \;=\;
  \inf\big\{\cE(f) \;:\; f|_A = 1,\; f|_B = 0,\; 0 \leq f \!\leq 1 \big\}
\end{align}
with $f = h_{A,B}$ as its unique minimizer.  Further, we denote by $\nu_{A,B}$ the \emph{last-exit biased distribution} that is defined by
\begin{align}\label{eq:def:LastExitBiasedDistri}
  \nu_{A,B}(x)
  \;\ldef\;
  \frac{\mu(x)\, \prob_{\!x}\!\big[\tau_B < \tau_A\big]}
  {\sum_{y \in A} \mu(y)\, \prob_{\!y}\!\big[\tau_B < \tau_A\big]}
  \;=\;
  \frac{\mu(x)\, e_{A,B}(x)}{\mathop{\mathrm{cap}}(A,B)},
  \qquad
  \forall\, x \in A.
\end{align}
Let us recall that $\mu_A(x) = \mu[x | A]$, which implies that $\nu_{A, B} \ll \mu_A$ for any non-empty, disjoint subsets $A, B \subset \cS$.
\begin{rem}\label{rem:meta_sets_cap}
  In view of \eqref{eq:Capacity:HittingTimes}, the condition \eqref{intro:eq:cap:meta_sets} can alternatively be written as
  \begin{align*}
    \forall A \subset \cS\setminus \bigcup_{i=1}^K M_i\ \forall M \in \cM :
    \quad
    \abs{\cM}\;\frac{
      \capacity\big(M, {\textstyle \bigcup_{i=1}^K M_i} \setminus M\big) / \mu[M]
    }
    {\capacity\big(A, {\textstyle \bigcup_{i=1}^K M_i} \big) / \mu[A]}
    \;\leq\;
    \varrho
    \;\ll\;
    1.
  \end{align*}
  Hence, the assumption of metastability is essentially a quantified comparison of capacities and measures.
\end{rem}
Finally, we write $\Mean_{\nu}[f]$ and $\var_{\nu}[f]$ to denote the expectation and the variance of a function $f\!: \cS \to \R$ with respect to a probability measure $\nu$.  Moreover, we define the relative entropy by
\begin{align*}
  \Ent_{\nu}[f^2]
  \;\ldef\;
  \Mean_{\nu}[f^2 \ln f^2] - \Mean_{\nu}[f^2]\, \ln \Mean_{\nu}[f^2],
\end{align*}
where we indicate the probability distribution $\nu$ explicitly as a subscript.

\subsection{Main result}
\label{subsec:main_results}
The first result concerns the mean hitting times of meta\-stable sets. We obtain an asymptotically expression in terms of capacities solely under Assumption~\ref{ass:metastability:sets}, if we, in addition, assume a bound on the asymmetry of the involved local minima.
\begin{thm}\label{thm:mean:hitting:times}
  Suppose that Assumption~\ref{ass:metastability:sets} holds with $K \geq 2$.  Fix $M_i \in \cM$ and define
  \begin{align*}
    J(i)
    \;\ldef\;
    \set[\big]{
      j \in \set*{1, \ldots, K} \setminus \{i\} : \mu[M_j] \geq \mu[M_i]
    }
    \qquad\text{and}\qquad
    B
    \;\ldef\;
    \bigcup\nolimits_{j \in J(i)} M_j. 
  \end{align*}
  If $B \ne \emptyset$ and if there exists $\delta \in [0, 1)$ such that $\mu[\cS_j] \leq \delta \mu[\cS_i]$ for all $j \not\in J(i) \cup \{i\}$ then
  \begin{align*}
    \mean_{\nu_{M_i, B}}[\tau_{B}]
    \;=\;
    \frac{\mu[\cS_i]}{\capacity(M_i, B)}\, 
    \Bigl(
      1 
      + 
      O\bigl(\delta + \varrho \ln (C_{\textup{ratio}} / \varrho)\bigr)
    \Bigr) \ , 
  \end{align*}
  with
  \begin{align*}
    C_{\textup{ratio}} 
    \;\ldef\; 
    \max_{j \in J(i)}\, \mu[\cS_j] / \mu[\cS_i] 
    \;<\; 
    \infty \ . 
  \end{align*}
\end{thm}
The main objects of interest in the present paper are the Poincar\'e and logarithmic Sobolev constant that are defined as follows.
\begin{defn}[Poincar\'e and logarithmic Sobolev constant]
  The Poincar\'e constant $C_{\PI} \equiv C_{\PI}(P, \mu)$ is defined by
  \begin{align}\label{eq:def:PI-constant}
    C_{\PI}
    \;\ldef\;
    \sup\big\{
      \var_{\mu}[f] \;:\; f \in \ell^2(\mu) \text{ such that } \cE(f) = 1
    \big\},
  \end{align}
  whereas the logarithmic Sobolev constant $C_{\LSI} \equiv C_{\LSI}(P, \mu)$ is given by
  \begin{align}\label{eq:def:LSI-constant}
    C_{\LSI}
    \;\ldef\;
    \sup\big\{
      \Ent_{\mu}[f^2] \;:\; f \in \ell^2(\mu) \text{ such that } \cE(f) = 1
    \big\}.
  \end{align}
\end{defn}
Let $\{F_i : i \in I\} = \cM \cup \{ \{x\} : x \in \cS \setminus \bigcup_{i=1}^K M_i\}$ be a partition of $\cS$ and denote by $\cF \ldef \sigma(F_i : i \in I)$ the corresponding $\sigma$-algebra, that is, $\cF$ is the $\sigma$-algebra lumping the metastable sets to single points.  Further, for any $f \in \ell^2(\mu)$ define the conditional expectation $\Mean_{\mu}[f \,|\, \cF]\!: \cS \to \R$ by
\begin{align}\label{def:F:cond:EX}
  \Mean_{\mu}[f \,|\, \cF](x) \;\ldef\; \Mean_{\mu}[f \,|\, F_i]
  \qquad \Longleftrightarrow \qquad
  F_i \ni x \ .
\end{align}
The starting point for proving sharp estimates of both the Poincar\'e and the logarithmic Sobolev constant in the context of metastable Markov chains is a splitting of the variance and the entropy into the contribution within and outside the metastable sets. The following two identities are the starting point of the identification of local relaxation within metastable valleys and rare transitions between metastable sets and hold for any $f \in \ell^2(\mu)$
\begin{align}
  \label{eq:var:split1}
  \var_{\mu}[f]
  &\;=\;
  \sum_{i=1}^K\, \mu[M_i]\, \var_{\mu_{M_i}}\![f]
  \,+\,
  \var_{\mu}\!\big[\Mean_{\mu}[f \,|\, \cF]\big]
  \\
  \label{eq:ent:split1}
  \Ent_{\mu}[f^2]
  &\;=\;
  \sum_{i=1}^K\, \mu[M_i]\, \Ent_{\mu_{M_i}}[f^2]
  \,+\,
  \Ent_{\mu}\!\big[\Mean_{\mu}[f^2 \,|\, \cF]\big] \ . 
\end{align}

Our main result relies on an assumption on the Poincar\'e and logarithmic Sobolev constants within the metastable sets and on a regularity condition for the  last exit biased distribution.
\begin{assume}\label{ass:PI+LSI:metastable:sets}
  Assume that for any $i \in \{1, \ldots, K\}$
  \begin{align}
    &\text{(i)}&
    \label{eq:ass:PI:metastable:sets}
    C_{\PI,i}
    &\;=\;
    \sup\big\{
      \mspace{6mu}\var_{\mu_{M_i}}[f]\mspace{6mu} 
      \,:\, 
      f \in \ell^2(\mu) \text{ such that } \cE(f) = 1
    \big\}
    \;<\;
    \infty,
    \\
    \label{eq:ass:LSI:metastable:sets}
    &\text{(ii)}&
    C_{\mathrm{LSI}, i}
    &\;=\;
    \sup\big\{
      \Ent_{\mu_{M_i}}[f^2]
      \,:\,
      f \in \ell^2(\mu) \text{ such that } \cE(f) = 1
    \big\}
    \;<\;
    \infty.
  \end{align}
\end{assume} 
In the error estimates the following derived constants occur
\begin{align*}
  C_{\PI,\cM}
  \;\ldef\; 
  \max\set[\bigg]{1, \sum_{i=1}^K\, \mu[M_i]\, C_{\PI,i}}
  \qquad\text{and}\qquad
  C_{\LSI, \cM} 
  \;\ldef\; 
  \max\set[\bigg]{1, \sum_{i=1}^K\, \mu[M_i]\, C_{\mathrm{LSI}, i}}.
\end{align*}
\begin{rem}
  The Assumption~\ref{ass:PI+LSI:metastable:sets} ensures that the process within each metastable set mixes quickly.  It can be interpreted as an additional smallness condition on the metastable sets $M\in \cM$ and for simple enough systems a simple bound on $C_{\PI,\cM}$ and $C_{\LSI,\cM}$ in terms of the maximal diameter of the sets $M\in \cM$ may be sufficient. For more complex systems, like the Curie--Weiss model, the constants~$C_{\PI,\cM}$ and~$C_{\LSI,\cM}$ may be comparable to known systems, which in this case is the Bernoulli--Laplace model.
\end{rem}
\begin{assume}[Regularity condition]\label{ass:regularity}
  Assume that there exists $\eta \in [0, 1)$ such that
  \begin{align}\label{eq:regularity}
    \var_{\mu_{M_i}}\!\bigg[\frac{\nu_{M_i,M_j}}{\mu_{M_i}}\bigg]
    \;\leq\;
    \frac{\eta\, \mu[M_i]}{\capacity(M_i, M_j)}
    \qquad
    \forall\, M_i, M_j \in \cM \quad \text{with } i \ne j.
  \end{align}
\end{assume}
\begin{rem}
  \begin{enumerate}[label=(\roman{enumi})]
  \item Note that the tentative definition of metastable sets as given in \cite[equation~(8.1.3)]{BdH15} would immediately imply that $\eta = o(1)$.  However, it is still an open problem how to relate the probabilities appearing in \cite[equation~(8.1.3)]{BdH15} to capacities.

  \item Since $e_{M_i, M_j}(x) \leq \capacity(M_i, M_j) / \mu(x)$ for all $x \in M_i$, the following trivial upper bound on $\eta$ holds
    \begin{align*}
      \eta
      \;\leq\;
      \min\set[\Big]{
        1, \varrho\, |M_i|\, \max_{x, y \in M_i}\set*{\mu(x) / \mu(y)}
      }.
    \end{align*}
    Hence, $\eta \ll 1$ provided that for each metastable set $M_i$ both its cardinality and the fluctuations of the invariant distribution $\mu$ on it are sufficiently small compared to $\varrho$.
    
  \item The above upper bound does not apply to particle systems like the Curie--Weiss model since $|M_i|$ is exponentially large in the system size.  Therefore, the verification of \eqref{eq:regularity} is based on coupling techniques.  For that purpose, the crucial observation is that the Curie--Weiss model is nearly lumpable in the sense that there exists a mesoscopic description, which up to small perturbations is Markovian. Under this condition Assumption~\ref{ass:regularity} is verifiable with $\eta$ of the same order as $\varrho$.  We expect that such strategy may apply to different mean field models that exhibit an effective mesoscopic description.
  \end{enumerate}
\end{rem}
\begin{rem}
  If each $M \in \cM$ consists of a single point, that is, $\forall M\in \cM: |M| = 1$, Assumptions~\ref{ass:regularity} and~\ref{ass:PI+LSI:metastable:sets} are satisfied for $\eta = 0$ and $C_{\PI,\cM} = C_{\LSI,\cM} = 1$.
\end{rem}
For the sake of presentation let us state the main result in the case of two metastable sets $K=2$. For the statement in the case of $K>2$, we refer to Theorem~\ref{thm:PI:metastable} and Theorem~\ref{thm:LSI:metastable}.
\begin{thm}\label{thm:main:PI}
  Suppose that the Assumptions~\ref{ass:metastability:sets} with $K=2$, \ref{ass:regularity} and \ref{ass:PI+LSI:metastable:sets}~i) hold such that $C_{\PI,\cM}(\varrho + \eta) \ll 1$. Then, it holds that
  \begin{align}\label{eq:main:PI}
    C_{\PI}
    \;=\;
    \frac{\mu[\cS_1]\, \mu[\cS_2]}{\capacity(M_1, M_2)}\;
    \bra[\Big]{1 + O\bra[\big]{\sqrt{C_{\PI,\cM} \bra{\varrho  + \eta}}}}.
  \end{align}
  Moreover, if in addition Assumption~\ref{ass:PI+LSI:metastable:sets}~ii) holds and
  \begin{align}\label{ass:Cmass}
    C_{\textup{mass}}
    \;\ldef\;
    \max_{i \in \{1, \ldots, K\}} \max_{x \in \cS_i}\,
    \ln\big(1 + \me^2 / \mu_i(x) \big)\
    \;<\;
    \infty
  \end{align}
  such that $C_{\textup{mass}} \, C_{\LSI,\cM}(\varrho + \eta) \ll 1$. Then, it holds that
  \begin{align}\label{eq:main:LSI}
    C_{\LSI}
    \;=\;
    \frac{1}{\Lambda\bra*{\mu[\cS_1], \mu[\cS_2]}}\; 
    \frac{\mu[\cS_1]\, \mu[\cS_2]}{\capacity(M_1, M_2)}\;
    \bra[\Big]{
      1 
      + 
      O\bra[\big]{\sqrt{C_{\textup{mass}}\, C_{\LSI,\cM} \bra{\varrho  + \eta}}}
    },
  \end{align}
  where $\Lambda(\alpha, \beta) \ldef \int_0^1 \alpha^s \beta^{1-s}\; \dx{s} = (\alpha-\beta)/\ln\tfrac{\alpha}{\beta}$ for $\alpha,\beta >0$ is the logarithmic mean.
\end{thm} 
\begin{rem}
  By using a standard linearization argument in~\eqref{eq:def:LSI-constant}, it follows that $C_{\LSI} \geq 2 C_{\PI}$; see \cite[Proposition 5]{BT06}. Notice that in the symmetric case when $\mu[\cS_1] \bra[\big]{1+o(1)} = \frac{1}{2} = \mu[\cS_2] \bra[\big]{1+o(1)}$ we have that $C_{\LSI} = 2 C_{\PI} \bra[\big]{ 1+ o(1)}$. Let us note that Assumption~\eqref{ass:Cmass} restricts the result on the logarithmic Sobolev constant to finite state spaces.
\end{rem}
\begin{cor}\label{cor:main:PI}
  Suppose that the assumptions of Theorem~\ref{thm:main:PI} hold.  Further, assume that $C_{\textup{ratio}}^{-1}  \ll 1$ and $\varrho \ln (C_{\textup{ratio}}) \ll 1$.  Then,
  \begin{align*}
    C_{\PI}
    \;=\;
    \mean_{\nu_{M_2, M_1}}\!\big[\tau_{M_1}\big]\;
    \bra[\Big]{
      1 
      + 
      O\bra[\big]{
        C_{\textup{ratio}}^{-1} + \varrho \ln(C_{\textup{ratio}})
        + \sqrt{C_{\PI,\cM} \bra{\varrho  + \eta}}
      }
    }.
  \end{align*}
\end{cor}

Let us comment on similar results in the literature.
\smallskip

The quantity on the right hand side of~\eqref{eq:main:PI} bears some similarity to the Cheeger constant \cite{Cheeger70} on weighted graphs \cite{LS88} defined by
\begin{align*}
  C_{\Cheeger}
  \;\ldef\;
  \sup_{A\subset \cS: \mu[A] \in (0,1)} 
  \frac{\mu[A]\, \mu[A^c]}{\capacity\bra*{ A, A^c}}
\end{align*}
with $A^c \ldef \cS \setminus A$.  Moreover, we note that $\capacity\bra*{ A, A^c} = -\skp{\one_{A}, L\one_{A^c}}_\mu$.  Then, the main result of~\cite[Theorem 2.1]{LS88} translated to the current setting reads
\begin{align*}
  C_{\Cheeger} \;\leq\; C_{\PI} \;\leq\; 8 C_{\Cheeger}^2 .
\end{align*}
Hence, the main result~\eqref{eq:main:PI} can be seen as an asymptotic sharp version of the Cheeger estimate in the metastable setting.

In the paper~\cite{BG16}, metastability has alternatively been characterized in terms of ratios $\eps$ between Dirichlet and Neumann spectral gaps of restricted generators. For this purpose the state space is decomposed into two sets $\cS = \cR \cup \cR^c$. Based on the assumption that $\eps \ll 1$, the result \cite[Theorem 2.9]{BG16} is an estimate on the mean-hitting time similar to Theorem~\ref{thm:mean:hitting:times}. Moreover, precise estimates of the relaxation rates toward the quasi-stationary distribution inside each element of the partition are obtained~\cite{BG16}. These estimates seem to be related to the local Poincar\'e inequality in Lemma~\ref{lemma:local:PI} below. Moreover, we expect, that there is a close connection between $\eps$ and $\varrho$ of Definition~\ref{def:meta_sets} in the setting $K=2$. 

In~\cite[Theorem 2.10]{BG16} a result bearing some similarity to~\eqref{eq:main:PI} is obtained. There the capacity $\capacity(M_1,M_2)$ needs to be replaced by so-called $(\kappa,\lambda)$-capacities between $\cR$ and $\cR^c$. These capacities are obtained by extending the state space by copies of $\cR$ and $\cR^c$ and equipping the connecting edges with conductivities $\kappa$ and $\lambda$. One notices that the error bound in this formulation depends on a careful choice of~$\kappa$ and~$\lambda$ in terms of $\eps$. The approach of this paper does not require such additional intermediate parameters and obtains similar results in Theorem~\ref{thm:main:PI} in a more straightforward manner.

\subsection{Random field Curie--Weiss model}
\label{subsec:main_results:RFCW}
One particular class of models we are interested in, are \emph{disordered mean field spin systems}.  As an example, we consider the Ising model on a complete graph, say on $N \in \N$ vertices, also known as Curie--Weiss model, in a random magnetic field.  The state space of this model is $\cS = \set*{-1,1}^N$.  The \emph{random Hamiltonian} is given by
\begin{align}\label{CW:Hamiltonian}
  H(\sigma)
  \;\ldef\;
  -\frac{1}{2N}\, \sum_{i,j = 1}^N \sigma_{i} \sigma_j
  \,-\, \sum_{i = 1}^N h_i \sigma_i ,
  \qquad \sigma \in \cS,
\end{align}
where $h \equiv (h_i : i \in \N)$ is assumed to be a family of i.i.d.~random variable on~$\R$ distributed according to $\prob^h$ with bounded support, that is
\begin{align}\label{CW:eq:ass:h}
  \exists\, h_{\infty} \in (0, \infty) : 
  \qquad  |h_i| \leq h_{\infty} 
  \qquad \prob^h\text{-almost surely.}
\end{align}
The \emph{random Gibbs measure} on $\cS$ is defined by
\begin{align*}
  \mu(\sigma)
  \;\ldef\;
  Z^{-1}\, \exp\bra{-\beta H(\sigma)}\, 2^{-N},
\end{align*}
where $\beta \geq 0$ is the inverse temperature and $Z$ is the normalization constant also called \emph{partition function}.  The additional factor $2^{-N}$ is for convenience and to be consistent with the definition in~\cite[(14.2.1)]{BdH15}.  The \emph{Glauber dynamics}, that we consider, is a Markov chain $(\sigma(t) : t \in \N_0)$ in discrete-time with \emph{random transition probabilities}
\begin{align}\label{CW:Glauber}
  p(\sigma, \sigma')
  \;\ldef\;
  \frac{1}{N}\, \exp\bra[\big]{- \beta \pra{ H(\sigma') - H(\sigma)}_+}\,
  \one_{|\sigma - \sigma'|_1=2},
\end{align}
where $[x]_+ \ldef \max\set{x,0}$ and $p(\sigma, \sigma) = 1 - \sum_{\sigma' \in \cS} p(\sigma, \sigma')$.  Notice that, for each realization of the magnetic field $h$, the Markov chain is ergodic and reversible with respect to the Gibbs measure $\mu$.

Various stationary and dynamic aspects of the random field Curie--Weiss model has been studied.  In particular, the metastable behavior of this model has been analyzed in great detail in \cite{BEGK01,BBI09,BBI12}, where the potential theoretic approach was used to compute precise metastable exit times and to prove the asymptotic exponential distribution of normalized metastable exit times.  For an excellent review we refer to the recent monograph \cite[Chapters~14 and 15]{BdH15}.  Estimates on the spectral gap have been derived in \cite{MP98} in the particular simple cases where the random field takes only two values $\pm \varepsilon$ and the parameters are chosen in such a way that only two minima are present.

A particular feature of this model is that it allows to introduce \emph{mesoscopic variables} by using a suitable coarse-graining procedure such that the induced dynamics are well approximated by a Markov process.  Let $I^h \ldef [-h_{\infty}, h_{\infty}]$ denote the support of $\prob^h$.  For any $n \in \N$ we find a partition of $I^h$ such that $|I_{\ell}^h| \leq 2h_{\infty} / n$ and $I^h = \bigcup_{\ell=1}^n I_{\ell}^h$.  Hence, each realization of $h$ induces a partition of the set $\set{1, \ldots, N}$ into mutually disjoint subsets
\begin{align*}
  \Lambda_{\ell} 
  \;\ldef\; 
  \set[\big]{i \in \set{1, \ldots, N} \,:\, h_i \in I_{\ell}^h }, 
  \qquad \ell \in \set{1, \ldots, n}.
\end{align*}
Based on this partition, consider the mesoscopic variable $\order\!: \cS \to \Gamma^n \subset [-1,1]^n$,
\begin{align*}
  \order(\sigma)
  \;=\;
  \bra[\big]{\order_1(\sigma), \ldots, \order_n(\sigma)}
  \qquad \text{with} \qquad
  \order_{\ell}(\sigma)
  \;\ldef\;
  \frac{1}{N}\, \sum_{i \in \Lambda_\ell} \sigma_i,
  \quad \ell \in \set{1, \ldots, n},
\end{align*}
that serves as an $n$-dimensional \emph{order parameter}.  A crucial feature of the mean field model is that the Hamiltonian \eqref{CW:Hamiltonian} can be rewritten as a function of the mesoscopic variable. In order to do so, for any $\ell \in \set{1, \ldots, n}$ the block-averaged field and its fluctuations are defined by
\begin{align*}
  \overbar{h}_{\ell}
  \;\ldef\;
  \frac{1}{\abs{\Lambda_\ell}}\, \sum_{i\in \Lambda_\ell} h_i
  \qquad\text{and}\qquad
  \tilde{h}_i \;\ldef\; h_i - \overbar{h}_{\ell},
  \quad \forall\, i \in \Lambda_{\ell}.
\end{align*}
Then,
\begin{align*}
  H(\sigma)
  \;=\;
  - N E\bra*{\order(\sigma)}
  -
  \sum_{\ell=1}^n \sum_{i \in \Lambda_\ell} \sigma_i \tilde{h}_i,
\end{align*}
where the function $E\!: [-1,1]^n \to \R$ is given by $E(\bfx) = \frac{1}{2} \bra*{\sum_{\ell=1}^n \bfx_{\ell}}^2 + \sum_{\ell=1}^n \overbar{h}_{\ell} \bfx_{\ell}$.  We define the distribution of $\order$ under the Gibbs measure as the \emph{induced measure}
\begin{align*}
  \magmu(\bfx)
  \;\ldef\;
  \mu \circ \order^{-1}(\bfx),
  \qquad \bfx \in \Gamma^n.
\end{align*}
Further, the \emph{mesoscopic free energy} $F\!:[-1,1]^n \to \R$ is defined by
\begin{align}\label{eq:CW:F}
  F(\bfx)
  \;\ldef\;
  E(\bfx) 
  + 
  \frac{1}{\beta}\, \sum_{\ell=1}^n\, \frac{|\Lambda_{\ell}|}{N}\,
  I_{\ell}\bra*{N \bfx_{\ell} / |\Lambda_{\ell}|},
\end{align}
where for any $\ell \in \set{1, \ldots, n}$ the \emph{entropy} $I_{\ell}$ is given as the Legendre-Fenchel dual of
\begin{align*}
  \R
  \;\ni\;
  t
  \;\mapsto\;
  \frac{1}{N}\,
  \sum_{i \in \Lambda_{\ell}} \ln \cosh\bra{t + \beta \tilde{h}_i}.
\end{align*}
Notice that the distribution $\magmu$ satisfies a sharp large deviation principle with scale~$N$ and rate function $F$.  The structure of the mesoscopic free energy landscape has been analyzed in great detail in \cite{BBI09}.  In particular, $\bfz \in [-1,1]^n$ is a critical point of $F$, if and only if, for all $\ell \in \set{1, \ldots, n}$
\begin{align}\label{CW:def:mesosCritPoint}
  \bfz_{\ell}
  \;=\;
  \frac{1}{N}\, \sum_{i \in \Lambda_{\ell}} \tanh\bra*{\beta( z(\bfz) + h_i)}
\end{align}
where $z(\bfz) = \sum_{\ell=1}^n \bfz_{\ell} \in \R$ solves in addition the equation $z = \frac{1}{N} \sum_{i=1}^N \tanh\bra*{\beta(z + h_i)}$.  It turns out that $\bfz$ is a critical point of index 1, if $\frac{\beta}{N} \sum_{i=1}^N 1 - \tanh^2\bra*{\beta(z(\bfz) + h_i)} > 1$.  Moreover, at any critical point $\bfz$ the value of the mesoscopic free energy can be computed explicitly and is given by
\begin{align}\label{CW:e:explicit:free_energy}
  F(\bfz)
  \;=\;
  \frac{1}{2} z(\bfz)^2
  \,-\,
  \frac{1}{\beta N} \sum_{i=1}^N \ln \cosh\bra*{\beta(z(\bfz) + h_i)} \ . 
\end{align}
Let us stress the fact that the topology of the mesoscopic energy landscape is independent of the artificial dimension parameter $n$.
\begin{rem}
  \begin{enumerate}[label=(\roman{enumi})]
  \item For a constant external field, that is, $h_i \equiv h$ for all $i$, the mesoscopic free energy, $F$, has two local minima if $\beta > 1$ and $h \in (-h_c(\beta), h_c(\beta))$, where $h_c(\beta) \ldef \sqrt{1-1/\beta} - \frac{1}{\beta} \ln (\sqrt{\beta} + \sqrt{\beta-1})$.
  \item By the strong law of large numbers, the set of solutions of the equation $z = \frac{1}{N} \sum_{i=1}^N \tanh\bra*{\beta(z + h_i)}$, determining the critical points of $F$, converges $\prob^h$-a.s.\ as $N \to \infty$ to the set of solutions of the deterministic equation
    \begin{align}\label{eq:CW:fixpoint:z}
      z
      \;=\;
      \EX^h\pra*{\tanh\bra[\big]{ \beta (z + h) }}.
    \end{align}
    Moreover, in view of \eqref{CW:def:mesosCritPoint}, the value of the mesoscopic free energy at critical points converges to a deterministic value for $\prob^h$-almost every realization of $h$ as $N$ tends to infinity.  
  \item If the distribution $\prob^h$ is symmetric, $z = 0$ is always a solution of \eqref{eq:CW:fixpoint:z}, and if $z > 0$ solves \eqref{eq:CW:fixpoint:z} than, by symmetry, $-z$ is as well a solution.  In general, the number of critical points depends on both the value of $\beta$ and the properties of the distribution $\prob^h$.   For discrete distributions $\prob^h$ the phase diagram has been studied in detail in \cite[Section 5]{APZ92} and \cite{SW85}.  
  \end{enumerate}
\end{rem}
In the sequel, we impose the following assumption on the law $\prob^h$.
\begin{assume}\label{CW:ass:law}
  Let $K \geq 2$ and assume that for $\prob^h$-almost every realization $h$, there exist $\beta > 0$ and $N_0(h) < \infty$ such that for all $N \geq N_0(h)$ and $n \geq 1$ the mesoscopic free energy $F\!: [-1,1]^n \to \R$ admits $K$ local minima.
\end{assume}
Denote by $\bfm_i \in \Gamma^n$, $i \in \set{1, \ldots, K}$, the best lattice approximation of the corresponding local minima.  We choose the label of $\bfm_i$ by the following procedure:  First, define for any non-empty, disjoint $\bfA, \bfB \subset \Gamma^n$ the \emph{communication height}, $\Phi(\bfA, \bfB)$, between $\bfA$ and $\bfB$ by
\begin{align}\label{CW:def:CommHeight}
  \Phi(\bfA, \bfB)
  \;=\;
  \min_{\boldsymbol{\gamma}}\max_{\bfx \in \boldsymbol{\gamma}} F(\bfx),
\end{align}
where the minimum is over all nearest-neighbour paths in $\Gamma^n$ that connect $\bfA$ and $\bfB$.  Then, the label is chosen in such a way that, with $\bfM_{k} \ldef \set{\bfm_1, \ldots, \bfm_k}$,
\begin{align}\label{CW:def:Delta}
  \Delta_{k-1}
  \;\ldef\;
  \Phi(\bfm_k, \bfM_{k-1}) - F(\bfm_k)
  \;\leq\;
  \min_{i < k}\set*{\Phi(\bfm_i, \bfM_k \setminus \bfm_i) - F(\bfm_i)}
\end{align}
for all $k = K, \ldots, 2$. Notice that, by construction, $\Delta_1 \geq \ldots \geq \Delta_{K-1} > 0 \rdef \Delta_K$.  Since $\Phi(\bfm_k, \bfM_{k-1})$ is given by the value of the mesoscopic free energy at the minimal saddle point between $\bfm_k$ and $\bfM_{k-1}$, \eqref{CW:e:explicit:free_energy} implies that the value of $\Delta_{k-1}$ is independent of $n$ for any $k = 2, \ldots, K$, and converges $\prob^h$-a.s.\ as $N \to \infty$.

In the sequel, we first impose conditions on the finiteness of the coarse-graining controlled by the parameter $n$.  Depending on the choice of $n$ the state space dimension $N$ has to be larger then a certain threshold.  In this sense, the results hold by first letting $N \to \infty$ and then $n \to \infty$.  With these definitions, we are able to formulate the statement that the random field Curie--Weiss model is $\varrho$-metastable in the sense of Definition~\ref{def:meta_sets}.
\begin{prop}[$\varrho$-metastability]\label{CW:prop:rho:metastable}
  Suppose that Assumption~\ref{CW:ass:law} holds.  Then, for $\prob^h$-almost every $h$ and any $c_1 \in (0,\Delta_{K-1})$ there exists $n_0 \equiv n_0(c_1)$ such that for all $n \geq n_0$ there exists $\overbar N < \infty$ such that for all $N \geq N_0(h) \vee \overbar N$ the random field Curie--Weiss model is $\varrho \ldef \me^{-c_1 \beta N}$-metastable in the sense of Definition~\ref{def:meta_sets} with respect to $\cM \ldef \{M_1,\dots, M_k \}$ with $M_k \ldef \order^{-1}(\bfm_k)$ for $k\in \set{1,\dots, K}$. 
\end{prop}
As an immediate consequence of Proposition~\ref{CW:prop:rho:metastable} and Theorem~\ref{thm:mean:hitting:times} we obtain the following result on the mean hitting between metastable sets with respect to the microscopic dynamics induced by the transition probabilities~\eqref{CW:Glauber}.
\begin{thm}
  Suppose that Assumption~\ref{CW:ass:law} holds.  For fixed $i \in \set{2,\dots,K}$ and $\delta > 0$ sufficiently small, suppose that, $\prob^h$-a.s., the sets
  \begin{align*}
    J(i)
    \;=\;
    \set{ j \in \set{1,\dots,i-1}: F(\bfm_j) +\delta \leq F(\bfm_i)}
    \qquad\text{and}\qquad B:= \bigcup_{j\in J(i)} M_i 
  \end{align*}
  are nonempty. Then, $\prob^h$-a.s.\ the following holds: For any $c \in (0,\min\set{\delta,\Delta_1, \dots, \Delta_{i-1}})$ there exists $n_0 = n_0(c)$ such that for all $n\geq n_0$ there exists $\overbar N$ such that, for all $N\geq N_0(h) \vee \overbar N$,
  \begin{align*}
    \mean_{\nu_{M_i, B}}[\tau_{B}]
    \;=\;
    \frac{\mu[\cS_i]}{\capacity(M_i, B)}\, 
    \bra*{ 1 + O(e^{-c \beta N})} \ , 
  \end{align*}
\end{thm}
To obtain matching upper and lower bounds in the application of Theorem~\ref{thm:main:PI} to the random field Curie--Weiss model in case $K \geq 3$, we impose the following non-degeneracy condition on the largest communication height. 
\begin{assume}[Non-degeneracy condition]\label{CW:ass:dominance}
  For $K \geq 3$, assume that $\prob^h$-a.s., there exist $\theta > 0$ and $N_1(h) < \infty$ such that
  \begin{align*}
    \Delta_1 - \Delta_2 \;\geq\; \theta,
    \qquad \forall\, N \geq N_1(h). 
  \end{align*}
\end{assume}
Under the non-degeneracy Assumption~\ref{CW:ass:dominance} it is possible to prove that the preimages of the first two local minima $\bfm_1$ and $\bfm_2$ are already metastable sets, which are relevant to capture the slowest time scale of the system.
\begin{prop}\label{CW:cor:rho:metastable}
    Suppose that Assumption~\ref{CW:ass:law} holds.  If $K = 2$ set $\theta = \Delta_1$ and $N_1(h) = 1$.  If $K \geq 3$ assume additionally that Assumption~\ref{CW:ass:dominance} is satisfied.  Then, for $\prob^h$-almost all $h$ and any $c_1 \in (0, \theta)$ there exists $n_0 \equiv n_0(c_1)$ such that for all $n \geq n_0$ there exists $\overbar N < \infty$ such that for all $N \geq N_0(h) \vee N_1(h) \vee \overbar N$ the random field Curie--Weiss model is $\varrho \ldef \me^{-c_1 \beta N}$-metastable with respect to $\cM \ldef \{M_1, M_2\}$, where $M_1 \ldef \order^{-1}(\bfm_1)$ and $M_2 \ldef \order^{-1}(\bfm_2)$.
\end{prop}
\begin{rem}
  Note that under the non-degeneracy condition from Assumption~\ref{CW:ass:dominance} the mesoscopic free energy landscape may still have more than one global minima. Moreover, we believe that the presented technique and especially Lemma~\ref{lemma:mean:diff} can be generalized to the case of several equal high energy barriers $\Delta_1= \Delta_2 = \dots= \Delta_l$ for some $l\geq 2$. This would allow us to drop the above Assumption~\ref{CW:ass:dominance}. In that case, the leading order capacity will be obtained by the effective capacity of the electrical network constructed from all the possibly degenerate leading order energy barriers in the system. For diffusion processes the construction is outlined in \cite[Section 4.5]{Sch12} and the according series and parallel laws for the total capacity are derived. 
\end{rem}
The second main result in this subsection is the application of Theorem~\ref{thm:main:PI} to the random field Curie--Weiss model defined by the random transition probabilities defined~\eqref{CW:Glauber}. 
\begin{thm}\label{CW:thm:EyringKramers}
  Suppose the assumptions of Proposition~\ref{CW:cor:rho:metastable} hold with $\varrho= \me^{-c_1\beta N}$.  Then, $\prob^h$-a.s., for any $c_2 \in (0,c_1/2)$ there exists $n_1 \equiv  n_1(c_1,c_2,\beta, h_{\infty}) < \infty$ such that for any $n \geq n_0 \vee n_1$ there exists $\overbar N < \infty$ such that for all $N \geq N_0(h) \vee N_1(h) \vee \overbar N$ the random field Curie--Weiss model satisfies a Poincar\'e inequality with constant
  \begin{align}\label{CW:eq:EyringKramers:Intro:PI}
    C_{\PI} 
    \;=\; 
    \frac{\mu\pra{ \cS_1 }\, \mu\pra{ \cS_2 } }{\capacity\bra*{M_1,M_2}}\;
    \bra[\Big]{ 1 + O\bra[\big]{ \me^{- c_2 \beta N}} }
  \end{align}
  as well as a logarithmic Sobolev inequality with constant
  \begin{align}\label{CW:eq:EyringKramers:Intro:LSI}
    C_{\LSI}
    \;=\; 
    \frac{C_{\PI} }{\Lambda\bra*{\mu\pra{ \cS_1 }, \mu\pra{ \cS_2 }} }\;
    \bra[\Big]{ 1 + O\bra[\big]{ \me^{- c_2 \beta N}} } .
  \end{align}
\end{thm}
Let us emphasis that this result is valid in the symmetric ($F(\bfm_1) = F(\bfm_2)$) as well as asymmetric case ($F(\bfm_1) \ne F(\bfm_2)$).  Moreover, the capacities between pairs of metastable sets are calculated asymptotically with explicit error bounds in~\cite{BEGK01, BBI09, BBI12, BdH15}.  Hence, the right-hand side of~\eqref{CW:eq:EyringKramers:Intro:PI} and~\eqref{CW:eq:EyringKramers:Intro:LSI} can be made asymptotically explicit in terms of the free energy~\eqref{eq:CW:F}.

In the asymmetric case $F(\bfm_1) \ne F(\bfm_2)$, we connect the mean hitting time with the Poincar\'e constant via Corollary~\ref{cor:main:PI}.
\begin{cor}
  Suppose that the assumptions of Theorem~\ref{CW:thm:EyringKramers} hold with $\varrho= \me^{-c_1 \beta N}$.  Then, $\prob^h$-a.s., for any $c_2 \in (0, \min\set{c_1/2, F(\bfm_2) - F(\bfm_1)})$
  \begin{align}\label{CW:eq:cor:PI}
    C_{\PI}
    \;=\;
    \mean_{\sigma}\!\big[\tau_{M_1}\big]\;
    \bra[\Big]{1 + O\bra[\big]{\me^{-c_2 \beta N}}},
    \qquad \forall\, \sigma \in M_2.
  \end{align}
\end{cor}
\begin{proof}
  In view of \cite[equation~(3.16)]{BBI09}, $C_{\textup{ratio}}^{-1} \ldef \mu[\cS_2] / \mu[\cS_1] = O\bra{\me^{-\beta \delta N}}$ for any $\delta \in (0, F(\bfm_2) - F(\bfm_1))$. Thus, \eqref{CW:eq:cor:PI} is an immediate consequence of Corollary~\ref{cor:main:PI} and \cite[Theorem~1.1]{BBI12}
\end{proof}
Finally, notice that sharp asymptotics of the mean hitting time including the precise prefactor has been establish in \cite{BBI09}, which by the above identification gives an asymptotic sharp formula for the Poincar\'e constant of the random field Curie Weiss model.

\section{Functional inequalities}
\label{sec:FI}
The results in this section consider functional inequalities which do not make any explicit reference to time. Therefore, the results hold in the more general setting of $L$ as defined in~\eqref{eq:generator} being the generator of a continuous time Markov chain on a countable state space $\cS$. This accounts to dropping the normalization condition $\sum_{y} p(x,y)=1$ and assuming $p(x,y)$ to be the elements of the infinitesimal generator satisfying
\begin{align*}
  \forall\, x \ne y : p(x,y) \geq 0 ,\quad \forall x:0 \leq - p(x,x) < \infty  \quad\text{and}\quad\sum_{y\in \cS} p(x,y)  = 0  .
\end{align*}
We refer to~\cite[Chapter 7.2.2]{BdH15} for the general relation of hitting times between discrete time and continuous time Markov chains.

\subsection{Capacitary inequality}
%
%
%
%
The capacitary inequality is a generalization of the co-area formula.  For Sobolev functions on $\R^d$ it has been first proven by Maz'ya in \cite{Ma72}.  For a comprehensive treatment of the continuous case with further applications we refer to \cite{BGL14,BCR06,BR03,Chen05,Chen06,Ma11}.
\begin{thm}[Capacitary inequality]\label{thm:CapInequ}
  For any $f \in \ell^2(\mu)$ and any $t \in [0, \infty)$, let $A_t$ be the super level-set of $f$, that is
  \begin{align}\label{eq:def:A_t}
    A_t
    \;\ldef\;
    \big\{
      x \in \cS \;:\; |f(x)| > t
    \big\}.
  \end{align}
  Let $B \subset \cS$ be non-empty, then for any function $f\!: \cS \to \R$ with $f|_B \equiv 0$ it holds that
  \begin{align}\label{eq:CapacitaryInequ:sharp}
    \int_{0}^{\infty} 2t\, \capacity(A_{t}, B)\, \dx{t}
    \;\leq\;
    4\, \cE(f).
  \end{align}
\end{thm}
\begin{proof}
  Due to the fact that $\cE(|f|) \leq \cE(f)$, let us assume without lost of generality that $f(x) \geq 0$ for all $x \in \cS$.  To lighten notation, for any $t \in [0, \infty)$, we denote by $h_t := h_{A_{t},B}$ the equilibrium potential as defined in \eqref{eq:equi:potential}.  Since $\supp L h_t \subset A_t \cup B$, $f|_B \equiv 0$ and $f|_{A_t} >t$, it follows that
  \begin{align*}
    t\, \capacity(A_t, B)
    &\;\leq\;
    \skp{-L h_t,f}_{\mu}
    \;=\;
    \frac{1}{2}\, \sum_{x, y \in \cS}\, \mu(x)\, p(x,y)\,
    \big( f(x) - f(y) \big) \big( h_{t}(x) - h_{t}(y) \big) .
  \end{align*}
  An application of the Cauchy--Schwarz inequality yields
  \begin{align}\label{eq:cap:ineq:est}
    &\int_{0}^{\infty} 2t\, \capacity(A_{t}, B)\, \dx{t}
    \nonumber\\[-1ex]
    &\mspace{36mu}\leq\;
    2\, \cE(f)^{1/2}\,
    \Bigg(
      \frac{1}{2}\, \sum_{x,y \in \cS} \mu(x)\, p(x,y)\,
      \bigg(
        \int_{0}^{\infty} \big( h_{t}(x) - h_{t}(y) \big)\, \dx{t}
      \bigg)^{\!2}
    \Bigg)^{\!1/2}.
  \end{align}
  Now, we use the following identity: for any function $g \in L^{\!1}([0,\infty))$ holds
  \begin{align*}
    \Bigg(
      \int_{0}^{\infty} g(t) \, \dx{t}
    \Bigg)^{\!2}
    = \int_0^\infty \int_0^\infty g(t)\; g(s) \; \dx{s}\, \dx{t}
    = 2 \int_0^\infty  \int_0^t g(t)\; g(s) \; \dx{s}\, \dx{t} .
  \end{align*}
  Thus, by rewriting the right-hand side of \eqref{eq:cap:ineq:est}, we find that
  \begin{align*}
    &\int_{0}^{\infty} 2t\, \capacity(A_{t}, B)\, \dx{t}
    \;\leq\;
    2\, \cE(f)^{1/2}\,
    \bigg(
      2\,\int_{0}^{\infty} \int_0^{t} \skp{-Lh_t,h_s}_{\mu}\, \dx{s}\, \dx{t}
    \bigg)^{\!\!1/2}.
  \end{align*}
  Finally, since $A_t \subset A_s$ for all $t \geq s$, we obtain that $\skp{-Lh_t,h_s}_{\mu} = \capacity(A_t, B)$.  Hence, the assertion of the theorem follows.
\end{proof}

\subsection{Orlicz--Birnbaum estimates}
Let us assume for a moment that for some constant $C_{\CI}>0$ a measure-capacity comparison of the form $C_{\CI} \capacity(A,B) \geq \mu[A]$ is valid for all $A\subset \cS \setminus B$.  Then, a combination of the capacitary inequality~\eqref{eq:CapacitaryInequ:sharp} with
\begin{align*}
  \Mean_{\mu}\pra*{ f^2 }
  \;=\;
  \int_0^\infty 2 t \, \mu[A_t] \; \dx{t},
\end{align*}
leads to $\Mean_{\mu}\pra{f^2} \leq 4 C_{\CI} \, \cE(f)$ for all $f$ with $f|_B \equiv 0$.  This observation, originally given in~\cite{Ma72}, provides estimates on the Dirichlet eigenvalue of the generator $L$.

This strategy can be generalized to the $\ell^p$ case and more generally to logarithmic Sobolev constants by introducing suitable Orlicz spaces.  In the sequel, we prove that Poincar\'e inequalities in Orlicz spaces are equivalent to certain measure-capacity inequalities.  Similar results for diffusion processes on $\R^d$ can be found in \cite[Chapter~8]{BGL14}.
\begin{defn}[{Orlicz space~\cite[Section~1.3]{RR02}}]
  A function $\Phi\!: [0,\infty) \to [0, \infty]$ is a \emph{Young function} if it is convex, $\Phi(0) = \lim_{r \to 0} \Phi(r) = 0$ and $\lim_{r \to \infty} \Phi(r) = \infty$.  The \emph{Legendre-Fenchel dual} $\Psi\!: [0, \infty) \to [0, \infty]$ of a Young function $\Phi$ defined by
  \begin{align*}
    \Psi(r) \;=\; \sup_{s\in [0,\infty]} \big\{ s r - \Phi(s) \big\}
  \end{align*}
  is again a Young function (cf.~Lemma~\ref{lem:YoungFunctions}), and the pair $(\Phi, \Psi)$ is called \emph{Legendre-Fenchel} pair.  For some $K > 0$ the Orlicz-norm of a function~$f \in \ell^1(\mu)$ is defined~by
  \begin{align}\label{eq:def:K-OrliczNorm}
    \norm{f}_{\Phi, \mu, K}
    \;\ldef\;
    \sup\big\{
    \Mean_{\mu}\!\big[ |f| g \big]
      \,:\, g \geq 0,\; \Mean_{\mu}\!\big[ \Psi(g) \big] \leq K
    \big\}.
  \end{align}
  We set $\norm{f}_{\Phi, \mu} \ldef \norm{f}_{\Phi, \mu, 1}$.  The space of \emph{Orlicz functions}, $\ell^{\Phi}(\mu, K)\subset \ell^1(\mu)$, is the set of summable functions $f$ on $\cS$ with finite Orlicz norm.
\end{defn}
\begin{lem}\label{lemma:est:IndOrlicz}
  For any $A \subset \cS$ holds
  \begin{align}\label{eq:est:IndOrlicz}
    \norm*{\dsOne_{A}}_{\Phi,\mu,K}
    \;=\;
    \mu[A]\, \Psi^{-1}\bra*{\frac{K}{\mu[A]}} ,
  \end{align}
  where $\Psi^{-1}(t) \ldef \inf \big\{ s\in [0,\infty] : \Psi(s) > t \big\}$.
\end{lem}
\begin{proof}
  Due to the variational definition of the Orlicz norm, by choosing $g(x) = \dsOne_A(x)\, \Psi^{-1}(K/\mu[A])$ we have that $\norm*{\dsOne_{A}}_{\Phi,\mu,K} \geq \mu[A]\, \Psi^{-1}\bra*{K/\mu[A]}$.  On the other hand, since $\Psi^{-1}$ is concave (cf.~Lemma~\ref{lem:YoungFunctions}), an application of Jensen's inequality yields
  \begin{align*}
    \Mean_{\mu}\big[\dsOne_A\, g\big]
    \;=\;
    \mu[A]\,
    \Mean_{\mu}\bigg[ \frac{\dsOne_A}{\mu[A]}\, (\Psi^{-1}\circ \Psi)(g)\bigg]
    \;\leq\;
    \mu[A]\,
    \Psi^{-1}\bigg(%
      \frac{1}{\mu[A]}\, \Mean_{\mu}\big[\dsOne_A\, \Psi(g)\big]
    \bigg).
  \end{align*}
  Taking finally the supremum over all $g$ with $\Mean_\mu[\Psi(g)] \leq K$ concludes the proof.
\end{proof}
\begin{ex}\label{ex:OrliczPairs}
  The following Legendre-Fenchel pairs are stated for later reference:
  \begin{itemize}
  \item[a)] For $p \in (1, \infty)$: $(\Phi_p(r), \Psi_p(r)) \ldef \big( \frac{1}{p} r^p , \frac{1}{p^*} r^{p^*} \big)$ with $1 / p + 1 / p^* = 1$ the resulting Orlicz norm is equivalent to the usual $\ell^p(\mu)$ spaces.  The limiting pair $p \to 1$ is given by $\Phi_1(r) = r$ and $\Psi_1\!: [0, \infty) \to [0, \infty]$ with
    \begin{align*}
      \Psi_1(r)
      \;=\;
      \begin{cases}
        0,          & r \leq 1 \\
        \infty,     & r > 1
      \end{cases}
      \quad\text{and hence}\quad
      \Psi_1^{-1}(r)
      \;=\;
      \begin{cases}
        0 , & r = 0 \\
        1 , & r > 0
      \end{cases}.
    \end{align*}
  \item[b)] $(\Phi_{\Ent}(r), \Psi_{\Ent}(r)) \ldef \big(\dsOne_{[1,\infty)}(r)\bra*{r \ln r - r + 1} , \me^r - 1\big)$ leads to a norm, which can be compared with the relative entropy
    \begin{align*}
      \forall f\!: \cS \to \R_+,
      \qquad
      \Ent_{\mu}[f] \;\leq\; \norm*{f}_{\Phi_{\Ent},\mu} .
    \end{align*}
    Indeed, by using the variational characterization of the entropy, we have
    \begin{align*}
      \Ent_{\mu}[f]
      &\;=\;
      \sup_g\big\{
        \Mean_{\mu}\!\big[f g\big]
        \,:\,
        \Mean_{\mu}\!\big[ \me^g \big] \leq 1
      \big\}
      \\
      &\;=\;
      \sup_h\big\{
        \Mean_{\mu}\!\big[f \ln \big( \me^h -1\big)\big]
        \,:\,
        h \geq 0,\; \Mean_{\mu}\!\big[\me^h - 1\big] \leq 1
      \big\}
      \;\leq\;
      \norm*{f}_{\Phi_{\Ent},\mu} ,
    \end{align*}
    where the last step follows from~\eqref{eq:def:K-OrliczNorm} by noting that $\ln(\me^h -1) \leq h$.
  \end{itemize}
\end{ex}
\begin{prop}[Orlicz--Birnbaum estimates]\label{prop:OrliczBirnbaum}
  Let $B \subset \cS$ and $\nu \in \cP(\cS)$. Then, for any Legendre-Fenchel pair $(\Phi,\Psi)$ there exist constants $C_\Phi, C_\Psi > 0$ satisfying
  \begin{align*}
    C_{\Psi} \;\leq\; C_{\Phi} \;\leq\; 4\, C_{\Psi} ,
  \end{align*}
  such that the following statements are equivalent:
  \begin{enumerate}[label=(\roman{enumi})]
  \item For all sets $A\subset \cS\setminus B$ the measure-capacity inequality holds
    \begin{align}\label{eq:def:MeasureCapInequ}
      \nu[A]\, \Psi^{-1}\bra*{\frac{K}{\nu[A]}}
      \;\leq\;
      C_{\Psi}\, \capacity(A, B) .
    \end{align}
  \item For all $f\!: \cS \to \R$ such that $f \in \ell^2(\mu)$ and $f|_{B} \equiv 0$, it holds that
    \begin{align}\label{eq:def:PhiPoincare}
      \norm*{ f^2 }_{\Phi,\nu,K}
      \;\leq\;
      C_{\Phi}\, \cE(f).
    \end{align}
  \end{enumerate}
\end{prop}
\begin{proof}
  (i) $\Rightarrow$ (ii):  Let $G_{\Psi, K} \ldef \set*{g : g \geq 0, \Mean_{\nu}[\Psi(g)] \leq K}$.  For $f \in \ell^2(\mu)$ with finite support let $A_t$ be the super-level set of $f$ as defined in \eqref{eq:def:A_t}.  Then,
  \begin{align*}
    \norm*{f^2}_{\Phi, \nu, K}
    \overset{\!\eqref{eq:def:K-OrliczNorm}\!}{\;=\;}
    \sup_{g \in G_{\Psi,K}} \Mean_{\nu}\!\big[f^2 g\big]
    \;\leq\;
    \int_0^\infty
      \mspace{-6mu}
      2t \sup_{g \in G_{\Psi,K}} \Mean_{\nu}\!\big[ g\, \dsOne_{A_t}\big]\,
    \dx{t}
    \;=\;
    \int_0^\infty \mspace{-6mu} 2\, \norm*{ \dsOne_{A_t}}_{\Phi,\nu,K}\, \dx{t}.
  \end{align*}
  Thus, an application of Lemma~\ref{lemma:est:IndOrlicz} and Theorem~\ref{thm:CapInequ} yields
  \begin{align*}
    \norm*{f^2}_{\Phi,\nu,K}
    &\overset{\!\eqref{eq:est:IndOrlicz}\!}{\;=\;}
    \int_0^{\infty} 
      \mspace{-6mu} 2\, \nu[A_t]\, \Psi^{-1}\bra*{\frac{K}{\nu[A_t]}}\,
    \dx{t}
    \overset{\!\eqref{eq:def:MeasureCapInequ}\!}{\;\leq\;}
    C_{\Psi}
    \int_0^\infty \mspace{-6mu} 2\, \capacity(A_t, B)\, \dx{t}
    \overset{\!\eqref{eq:CapacitaryInequ:sharp}\!}{\;\leq\;}
    4\, C_{\Psi}\, \cE(f).
  \end{align*}
  The case $f\in \ell^2(\mu)$ follows from dominated convergence, since $\cE(f) \leq \norm{f}_{\ell^2(\mu)}^2$.

  \noindent (ii) $\Rightarrow$ (i): Since $\cE(f)\leq \norm{f}_{\ell^2(\mu)}^2$, we get from~\eqref{eq:def:PhiPoincare}, that $f^2 \in \ell^{\Phi}(\nu,K)$. Hence, for any $f\in \ell^2(\mu)$ with $f|_A \equiv 1 $ and $f|_{B} \equiv 0$ it holds that
  \begin{align*}
    \nu[A]\, \Psi^{-1}\bra*{\frac{K}{\nu[A]}}
    \overset{\eqref{eq:est:IndOrlicz}}{\;=\;}
    \norm*{\dsOne_A}_{\Phi,\nu,K}
    \;\leq\;
    \norm*{f^2}_{\Phi,\nu,K}
    \overset{\eqref{eq:def:PhiPoincare}}{\;\leq\;}
    C_{\Phi}\, \cE(f),
  \end{align*}
  which, by the Dirichlet principle~\eqref{eq:DirichletPrinciple}, leads to~\eqref{eq:def:MeasureCapInequ}.
\end{proof}
\begin{rem}
  Let us note, that either estimate~\eqref{eq:def:MeasureCapInequ} or~\eqref{eq:def:PhiPoincare} of Proposition~\ref{prop:OrliczBirnbaum} implies $\nu \ll \mu$ on $\cS \setminus B$ with bounded density. Indeed, for $x \in \cS \setminus B$ choose the function $\cS \ni y \mapsto \one_{x}(y)$ as a test function in the Dirichlet principle~\eqref{eq:DirichletPrinciple} and apply~\eqref{eq:def:MeasureCapInequ}. The same estimate can be obtained from~\eqref{eq:def:PhiPoincare} by considering again $\cS \ni y \mapsto \one_{x}(y)$ and using the representation~\eqref{eq:est:IndOrlicz}. In both cases, we get that, for any $x\in \cS\setminus B$,
  \begin{align*}
    \nu(x)
    \;\leq\;
    \frac{C}{\Psi^{-1}\bra[\big]{K/\nu(x)}}
    \;\leq\;
    \frac{C}{\Psi^{-1}(K)}\; \mu(x),
  \end{align*}
  where we used the monotonicity of $[0, \infty) \ni r \mapsto \Psi^{-1}(r)$. Hereby, $C$ is either $C_\Psi$ or $C_\Phi$. Hence, $\nu \ll \mu$ and, therefore, $\ell^1(\mu) \subseteq \ell^1(\nu)$.
\end{rem}
\begin{rem}
  The result of Proposition~\ref{prop:OrliczBirnbaum} is a generalization of the Muckenhoupt criterion~\cite{Mu72} for weighted Hardy inequalities, which was translated to the discrete setting in~\cite{Mi99} for the particular case $\cS=\N_0$.  The statement is, that for any $\nu, \mu \in \cP(\N_0)$ and any $f\!: \set*{-1} \cup \N_0 \to \R$ with $f(0)=f(-1)=0$ the inequality
  \begin{align}\label{eq:MuckenhouptPoincare}
    \sum_{x \geq 0}\, \nu(x)\, f(x)^2
    \;\leq\;
    C_1 \sum_{x \geq 0}\, \mu(x)\, \big( f(x+1) - f(x) \big)^2
  \end{align}
  holds if and only if
  \begin{align*}
    C_2
    \;=\;
    \sup_{x \geq 1} \bra*{\sum_{y=0}^{x-1} \frac{1}{\mu[y]}} \sum_{y\geq x} \nu[y]
    \;<\;
    \infty .
  \end{align*}
  In this case the constants satisfy $C_2\leq C_1 \leq 4 C_2$.  This results can be deduced from Proposition~\ref{prop:OrliczBirnbaum} by using the Orlicz-pair $(\Phi_1,\Psi_1)$ from Example~\ref{ex:OrliczPairs}~a) and setting $B = \set*{0}$. Then~\eqref{eq:def:PhiPoincare} becomes~\eqref{eq:MuckenhouptPoincare} for the (continuous time) generator
  \begin{align*}
    (L f)(x)
    \;=\;
    \bra*{f(x+1)-f(x)} \,+\, \frac{\mu(x-1)}{\mu(x)}\, \bra*{f(x-1)-f(x)}
  \end{align*}
  and therefore $C_{\Phi_1}=C_1$.  Notice that the equilibrium potential and hence the capacity along a one-dimensional, cycle-free path can be calculated explicitly (see e.g.\ \cite[Section 7.1.4]{BdH15}).  In particular, for any $x \in \N$ the solution $h_{x,0} \equiv h_{\{x, \ldots, \infty\}, \{0\}}$ of the boundary value problem~\eqref{eq:equi:potential} on $\N_0$ is given by
  \begin{align*}
    h_{x,0}(y)
    \;=\;
    \left. 
      \sum_{z=y}^{x-1} \frac{1}{\mu(z)} 
    \right/ 
    \sum_{z=0}^{x-1} \frac{1}{\mu(z)} 
    \qquad\text{and}\qquad 
    \capacity\big(\{x, \ldots, \infty\},\{0\}\big) 
    \;=\; 
    \bra*{\sum_{z=0}^{x-1} \frac{1}{\mu(z)}}^{\!\!-1}\mspace{-15mu}.
  \end{align*}
  In view of~\eqref{eq:def:MeasureCapInequ}, this verifies that $C_{\Psi_1} = C_2$. The weighted Hardy inequality was then used to derive Poincar\'e and logarithmic Sobolev inequalities (cf.~\cite{BG99,Mi99,ABCF+00}), which we will do in a similar way in the following two corollaries.
\end{rem}

\subsection{Poincar\'e and Sobolev inequalities}
Note that Poincar\'e or logarithmic Sobolev inequalities do not follow directly from Proposition~\ref{prop:OrliczBirnbaum}.  The reason is that the Orlicz--Birnbaum estimate~\eqref{eq:def:PhiPoincare} is for Dirichlet test functions vanishing on a specific set, whereas the Poincar\'e and logarithmic Sobolev inequalities concern Neumann test functions, which have average zero.  Therefore, a splitting technique can be used to translate the Orlicz--Birnbaum estimate to the Neumann case. See also \cite[Chapter 4.4]{Chen06} for some background on this technique. The additional step is taken care in the following two corollaries.
\begin{cor}[Poincar\'e inequalities]\label{cor:PI_via_CapInequ}
  Let $\nu\in \cP(\cS)$ and $b\in \cS$.  Then, there exist $C_{\var}, C_{\PI} > 0$ satisfying
  \begin{align}\label{eq:PI_via_CapInequ}
    \nu(b)\, C_{\var} \;\leq\; C_{\PI} \;\leq\; 4 \, C_{\var}
  \end{align}
  such that the following statements are equivalent:
  \begin{enumerate}[label=(\roman{enumi})]
  \item For all $A \subset \cS \setminus \set*{b}$, the inequality holds
    \begin{align}\label{eq:var:MC}
      \nu[A] \;\leq\; C_{\var}\, \capacity(A, b) .
    \end{align}
  \item The mixed Poincar\'e inequality holds, that is
    \begin{align*}
      \var_{\nu}[f] \;\leq\; C_{\PI}\, \cE(f),
      \qquad \forall f \in \ell^2(\mu).
    \end{align*}
  \end{enumerate}
\end{cor}
\begin{proof}
  (i) $\Rightarrow$ (ii): Let $(\Phi_1, \Psi_1)$ as in Example~\ref{ex:OrliczPairs}~$a)$ and recall that $\Psi_1^{-1}|_{(0,\infty)} \equiv 1$. Then, the measure-capacity inequality~\eqref{eq:def:MeasureCapInequ} coincides exactly with \eqref{eq:var:MC}.  Hence,
  \begin{align*}
    \var_{\nu}\!\big[f\big]
    \;=\;
    \min_{a \in \R}\, \Mean_{\nu}\!\big[(f-a)^2\big]
    \;\leq\;
    \norm*{(f - f(b))^2}_{\Phi_1, \nu}
    \overset{\!\eqref{eq:def:PhiPoincare}\!}{\;\leq\;}
    4\,C_{\var}\, \cE(f) .
  \end{align*}
  (ii) $\Rightarrow$ (i): We start with deducing a lower bound for the variance. Let $0 \leq f\leq 1$ be given such that $f|_A \equiv 1$ and $f(b) = 0$, then
  \begin{align*}
    \var_{\nu}[f]
    \;=\;
    \frac{1}{2} \sum_{x,y\in \cS} \nu(x)\, \nu(y)\; \big(f(x)-f(y)\big)^2
    \;\geq\;
    \sum_{x\in A} \nu(x)\, \nu(b)
    \;=\;
    \nu[A]\, \nu(b) .
  \end{align*}
  The conclusion follows from the Dirichlet principle~\eqref{eq:DirichletPrinciple}.
\end{proof}
\begin{cor}[Logarithmic Sobolev inequalities]\label{cor:LSI_via_CapInequ}
  Let $\nu \in \cP(\cS)$ and $b \in \cS$. Then, there exist $C_{\Ent}, C_{\LSI}>0$ satisfying
  \begin{align}\label{eq:LSI_via_CapInequ}
    \frac{\nu(b)}{\ln(1+\me^2)}\, C_{\Ent}
    \;\leq\;
    C_{\LSI}
    \;\leq\;
    4\, C_{\Ent}
  \end{align}
  such that the following statements are equivalent:
  \begin{enumerate}[label=(\roman{enumi})]
  \item For all $A \subset \cS \setminus \set*{b}$ the inequality holds
    \begin{align}\label{eq:Ent:MC}
      \nu[A]\, \ln\bra*{1 + \frac{\me^2}{\nu[A]}}
      \;\leq\;
      C_{\Ent}\, \capacity(A, b) .
    \end{align}
  \item The mixed logarithmic Sobolev inequality holds, that is
    \begin{align}\label{eq:Ent}
      \Ent_{\nu}[f^2] \;\leq\; C_{\LSI}\, \cE(f), 
      \qquad \forall f \in \ell^2(\mu) .
    \end{align}
  \end{enumerate}
\end{cor}
\begin{proof}
  (i) $\Rightarrow$ (ii): Set $f_b(x) \ldef f(x) - f(b)$ for $x \in \cS$.  Then, by applying a useful observation due to Rothaus~\cite[Lemma~5.1.4]{BGL14},
  \begin{align}\label{eq:cor:LSI_via_CapInequ:p1}
    \Ent_{\nu}\!\big[f^2\big]
    &\;\leq\;
    \Ent_{\nu}\!\big[ f_{b}^2\big] \,+\, 2\, \Mean_{\nu}\!\big[f_{b}^2\big]
    \nonumber\\[1.5ex]
    &\;=\;
    \sup_g \big\{
      \Mean_{\nu}\!\big[ f_{b}^2\, (g+2)\big]
      \,:\, \Mean_{\nu}\!\big[ \me^g \big] \leq 1
    \big\}
    \nonumber\\[.5ex]
    &\;\leq\;
    \sup_h \set*{
      \Mean_{\nu}\!\big[ f_{b}^2\, h\big]
      \,:\, h \geq 0,\; \Mean_{\nu}\!\big[ \me^h-1 \big] \leq \me^2
    }
    \;=\;
    \norm*{ f_{b}^2}_{\Phi, \nu, \me^2}, 
  \end{align}
  where we used the Orlicz-Pair~\ref{ex:OrliczPairs}~$b)$ and the definition of the $K$-Orlicz norm with $K=\me^2$ in~\eqref{eq:def:K-OrliczNorm}. The first implication follows now by an application of~\eqref{eq:def:PhiPoincare}.
  \smallskip

  \noindent(ii) $\Rightarrow$ (i): In order to prove the opposite direction, let $A \subset \cS \setminus \{b\}$ with $\nu[A] \ne 0$, and consider a function $f\!: \cS \to [0,1]$ with the property that $f|_A \equiv 1$ and $f(b) = 0$.  By using $g = \ln\bra*{1 / \nu[A]}$ as test function in the variational representation of the entropy we deduce that
  \begin{align*}
    \Ent_{\nu}\!\big[ f^2 \big]
    &\;\geq\;
    \sup_g \set*{
      \Mean_{\nu}\!\big[ g \dsOne_A\big]
      \,:\,
      \Mean_\nu[\me^g \dsOne_A] \leq 1
    }
    \;\geq\;
    \nu[A]\, \ln\bra*{\frac{1}{\nu[A]}}.
  \end{align*}
  Since $\ln(1/x) / \ln(1 + \me^2/x) \geq (1-x)/\ln(1+\me^2)$ for all $x \in (0,1]$ and $\nu[A] \in (0, 1-\nu(b)]$, we obtain that
  \begin{align*}
    \nu[A]\, \ln\bra*{\frac{1}{\nu[A]}}
    \;\geq\;
    \nu[A]\,\ln\bra*{1 + \frac{\me^2}{\nu[A]}}\,\frac{\nu(b)}{\ln(1+\me^2)}.
  \end{align*}
  Thus, \eqref{eq:Ent:MC} follows from \eqref{eq:Ent} by the Dirichlet principle~\eqref{eq:DirichletPrinciple}.
\end{proof}
The results of Corollary~\ref{cor:PI_via_CapInequ} and Corollary~\ref{cor:LSI_via_CapInequ} can be strengthened to identify the optimal Poincar\'e and logarithmic Sobolev constant up to a universal numerical factor, that is, by replacing $\nu[b]$ in the lower bounds~\eqref{eq:PI_via_CapInequ} and~\eqref{eq:LSI_via_CapInequ} by a universal numerical constant. The price to pay is to enforce the assumptions in the inequalities~\eqref{eq:var:MC} and~\eqref{eq:Ent:MC}. Although, in the application to metastable Markov chains, these results cannot provide an asymptotic sharp constant, we include them here for completeness.
\begin{cor}\label{cor:PI_via_CapInequ:universal}
  Let $\nu\in \cP(\cS)$. Then, there exist $C_{\var}, C_{\PI} > 0$ satisfying
  \begin{align*}
    \tfrac{1}{2} C_{\var} \;\leq\; C_{\PI} \;\leq\; 4 \, C_{\var}
  \end{align*}
  such that the following statements are equivalent:
  \begin{enumerate}[label=(\roman{enumi})]
  \item For all disjoint subsets $A,B \subset \cS$ with
    \begin{align}\label{eq:var:MC:general}
      \nu[A] \;\leq\; \tfrac{1}{2}
      \quad \text{and} \quad
      \nu[B] \;\geq\; \tfrac{1}{2}
      \qquad\text{holds}\qquad
      \nu[A] \;\leq\; C_{\var}\, \capacity(A, B) .
    \end{align}
  \item The mixed Poincar\'e inequality holds, that is
    \begin{align*}
      \var_{\nu}[f] \;\leq\; C_{\PI}\, \cE(f) ,\qquad \forall f\in \ell^2(\mu) .
    \end{align*}
  \end{enumerate}
\end{cor}
\begin{proof}
  (i) $\Rightarrow$ (ii): Let us denote by $m \in \R$ the median of $f$ with respect to $\nu$, that is $\nu[ f < m ] \leq \frac{1}{2}$ and $\nu[ f > m ]\leq \frac{1}{2}$.  Note that the sets $A^-=\set{f<m}$ and $B^-=\set{f\geq m}$ and $A^+=\set{f>m}$ and $B^+=\set{f\leq m}$ satisfy the assumption \eqref{eq:var:MC:general}.  Moreover, by means of Proposition~\ref{prop:OrliczBirnbaum}, we get that
  \begin{align*}
    \var_{\nu}[f]
    \;\leq\;
    \Mean_{\nu}\big[(f-m)_-^2\big] + \Mean_{\nu}\big[(f-m)_+^2\big]
    \;\leq\;
    4 C_{\var} \Big( \cE\bra[\big]{(f-m)_-} + \cE\bra[\big]{(f-m)_+} \Big).
  \end{align*}
  Hence, the conclusion of the first implication follows once we have shown that $\cE\bra{(f-m)_+} \,+\, \cE\bra{(f-m)_-} \leq \cE(f)$.  However, such estimate is a consequence of the pointwise bound
  \begin{align*}
    \bra[\big]{(f(x)-m)_+ - (f(y)-m)_+}^2
    \,+\,
    \bra[\big]{(f(x)-m)_- - (f(y)-m)_-}^2
    \;\leq\;
    \bra[\big]{f(x)-f(y)}^2
  \end{align*}
  for any $x,y \in \cS$.  Indeed, the bound is obvious for the cases $x,y\in \set*{f>m}$ and $x,y\in \set*{f<m}$. Now, suppose that $x \in \set*{f >m}$ and $y\in \set*{f<m}$, then the inequality reduces to show
  \begin{align*}
    \bra*{f(x)-m}^2 \,+\, \bra*{f(y)-m}^2 \;\leq\; \bra*{f(x)-f(y)}^2 ,
  \end{align*}
  which follows from the elementary inequality $ m f(x) + m f(y) - m^2 \geq f(x) f(y)$ provided that $f(y)\leq m \leq f(x)$.
  \smallskip

  \noindent(ii) $\Rightarrow$ (i): For the converse statement let $f$ be a test function such that $0 \leq f \leq 1$, $f|_A \equiv 1$ and $f|_B \equiv 0$.  Then, 
  \begin{align*}
    \cE(f)\, C_\PI
    \;\geq\;
    \var_{\nu}[f]
    \;=\;
    \frac{1}{2} \sum_{x,y\in \cS} \nu(x)\, \nu(y)\; \big(f(x)-f(y)\big)^2
    \;\geq\;
    \nu[A]\, \nu[B]
    \;\geq\;
    \tfrac{1}{2}\, \nu[A] .
  \end{align*}
  The conclusion follows from the Dirichlet principle~\eqref{eq:DirichletPrinciple}.
\end{proof}
\begin{cor}
  Let $\nu \in \cP(\cS)$. Then, there exist $C_{\Ent}, C_{\LSI}$ satisfying
  \begin{align*}
    \frac{1}{2 \ln\bra*{1+\me^2}}\, C_{\Ent}
    \;\leq\;
    C_{\LSI}
    \;\leq\;
    4 \, C_{\Ent}
  \end{align*}
  such that the following statements are equivalent:
  \begin{enumerate}[label=(\roman{enumi})]
  \item For all $A,B \subset \cS$ disjoint with
    \begin{align*}
      \nu[A] \;\leq\; \tfrac{1}{2}
      \quad \text{and} \quad
      \nu[B] \;\geq\; \tfrac{1}{2}
      \qquad \text{holds} \qquad
      \nu[A]\, \ln\bra*{1 + \frac{\me^2}{\nu[A]}}
      \;\leq\; C_{\Ent}\, \capacity(A, B) .
    \end{align*}
  \item The mixed logarithmic Sobolev inequality holds, that is
    \begin{align*}
      \Ent_{\nu}[f] \;\leq\; C_{\LSI}\, \cE(f),
      \qquad \forall f\in\ell^2(\mu).
    \end{align*}
  \end{enumerate}
\end{cor}
\begin{proof}
  (i) $\Rightarrow$ (ii): We shift $f$ according to its median $m$ with respect to $\nu$ (cf.~proof of Corollary~\ref{cor:PI_via_CapInequ:universal}) by applying~\eqref{eq:cor:LSI_via_CapInequ:p1} to $f-m$ and get
  \begin{align*}
    \Ent_{\nu}[f^2]
    \;\leq\;
    \norm*{(f-m)^2}_{\Phi,\nu,\me^2}
    \;\leq\;
    \norm*{(f-m)_+^2}_{\Phi,\nu,\me^2} + \norm*{(f-m)_-^2}_{\Phi,\nu,\me^2} .
  \end{align*}
  The first implication follows by applying Proposition~\ref{prop:OrliczBirnbaum} and combining $\cE\bra{(f-m)_+}$ and $\cE\bra{(f-m)_-}$ as in the proof of Corollary~\ref{cor:PI_via_CapInequ:universal}.
  \smallskip
  
  \noindent(ii) $\Rightarrow$ (i): The converse statement follows exactly along the lines of the proof of Corollary~\ref{cor:LSI_via_CapInequ} with the additional assumption that $\nu[B] \geq \frac{1}{2}$.
\end{proof}

\section{Application to metastable Markov chains}
\label{sec:MetaMC}
In Section~\ref{s:apriori}, we derive estimates and other technical tools based on the capacitary inequality as well as the metastable assumption. Sections~\ref{s:PI} and~\ref{s:LSI} contain the main results on the asymptotically sharp estimates for the Poincar\'e and logarithmic Sobolev constants for metastable Markov chains, respectively.
\smallskip

Throughout this section, we suppose that Assumption~\ref{ass:metastability:sets} holds.

\subsection{A priori estimates}
\label{s:apriori}
To apply the definition of metastable sets, we first show that for any subset of the local valley $\valley_i$ the hitting probability of the union of all metastable sets can be replaced by the hitting probability of any single set $M \in \cM$.
\begin{lem}
  For any $M_i \in \cM$ and $A \subset \valley_i \setminus M_i$,
  \begin{align}\label{eq:CompareMetaHitting:sets}
    \prob_{\mu_A}\!\big[\tau_{M_i} < \tau_A\big]
    \;\geq\;
    \frac{1}{|\cM|}\,
    \prob_{\mu_A}\!\big[\tau_{{\scriptscriptstyle \bigcup_{j=1}^K} M_j} < \tau_A\big].
  \end{align}
  In particular,
  \begin{align}\label{eq:metaest:sets}
    \prob_{\mu_A}\!\big[\tau_{M_i} < \tau_A\big]
    \;\geq\;
    \frac{1}{\varrho}\,
    \max_{M \in \cM}
    \prob_{\mu_{M}}\!\big[
      \tau_{{\scriptscriptstyle \bigcup_{j=1}^K} M_j \setminus M} < \tau_{M}
    \big].
  \end{align}
\end{lem}
\begin{proof}
  Since \eqref{eq:metaest:sets} is an immediate consequence of \eqref{eq:CompareMetaHitting:sets} and Definition~\ref{def:meta_sets}, it suffices to prove \eqref{eq:CompareMetaHitting:sets}.  Since $\prob_{x}[\tau_{M} < \tau_{{\scriptscriptstyle \bigcup_{j=1}^K} M_j \setminus M}] = \prob_{x}[X(\tau_{{\scriptscriptstyle \bigcup_{j=1}^K} M_j}) = M]$ for any $M \in \cM$ and $x \in \cS$, we obtain
  \begin{align*}
    1
    \;=\;
    \sum_{M \in \cM}
    \prob_{x}[\tau_{M} < \tau_{{\scriptscriptstyle \bigcup_{j=1}^K} M_j \setminus M}]
    \;\leq\;
    |\cM|\,
    \prob_{x}[
      \tau_{M_i} < \tau_{{\scriptscriptstyle \bigcup_{j=1}^K} M_j \setminus M_i}
    ],
    \qquad \forall\, x \in \valley_i.
  \end{align*}
  Thus,
  \begin{align}\label{eq:valley:est1}
    \prob_{\nu_{A, B}}\!\big[
      \tau_{M_i} < \tau_{{\scriptscriptstyle \bigcup_{j=1}^K} M_j \setminus M_i}
    \big]
    \;\geq\; 
    \frac{1}{|\cM|},
    \qquad \forall\, A \subset \valley_i \setminus M_i,
  \end{align}
  where $\nu_{A, B}$ is the last-exit biased distribution as defined in \eqref{eq:def:LastExitBiasedDistri} with $B = \bigcup_{j=1}^K M_j$.  On the other hand, by using averaged renewal estimates that has been proven in \cite[Lemma~1.24]{Sl12}, we get that
  \begin{align}\label{eq:valley:est2}
    \prob_{\nu_{A, B}}\!\big[
      \tau_{M_i} < \tau_{{\scriptscriptstyle \bigcup_{j=1}^K} M_j \setminus M_i}
    \big]
    \;\leq\;
    \frac{\prob_{\mu_A}\!\big[\tau_{M_i} < \tau_{A}\big]}
    {
      \prob_{\mu_A}\!\big[
        \tau_{{\scriptscriptstyle \bigcup_{j=1}^K} M_j} < \tau_A 
      \big]
    }
  \end{align}
  By combining the estimates \eqref{eq:valley:est1} and \eqref{eq:valley:est2}, \eqref{eq:CompareMetaHitting:sets} follows.
\end{proof}
The following lemma shows that the intersection of different local valleys has a negligible mass under the invariant distribution.
\begin{lem}\label{lem:negligible:points}
  Suppose that $X \ldef \valley_k \cap \valley_l \setminus (M_k \cup M_l)$ is non-empty.  Then, it holds that
  \begin{align*}
    \mu[X]
    \;\leq\;
    \varrho\, |\cM|\, \min\big\{\mu[M_k], \mu[M_l]\big\}.
  \end{align*}
\end{lem}
\begin{proof}
  Without loss of generality assume that $\mu[M_k] \leq \mu[M_l]$.  Notice that by \eqref{eq:valley:est1}, $h_{M_l, M_k}(x) \geq \prob_x[\tau_{M_l} < \tau_{{\scriptscriptstyle \bigcup_{j=1}^K} M_j \setminus M_l}]\geq 1/|\cM|$ for any $x \in X \subset \valley_l \setminus M_l$.  Therefore,
  \begin{align*}
    \capacity(M_k, M_l)
    \;\geq\;
    \skp{-L h_{M_l, M_k}, h_{X, M_k}}_{\mu}
    \;=\;
    \skp{h_{M_l, M_k}, -L h_{X, M_k}}_{\mu}
    \;\geq\;
    \frac{1}{|\cM|}\, \capacity(X, M_k).
  \end{align*}
  Thus,
  \begin{align*}
    \mu[X]
    \;\leq\;
    |\cM|\;
    \frac{\capacity(M_k, M_l)}{\prob_{\mu_X}\!\big[\tau_{M_k} < \tau_X\big]}
    \overset{\eqref{eq:metaest:sets}}{\;\leq\;}
    \varrho |\cM|\,
    \frac{\capacity(M_k, M_l)}
    {
      \prob_{\mu_{M_k}}\!\big[
        \tau_{{\scriptscriptstyle \bigcup_{j=1}^K} M_j \setminus M_k} < \tau_{M_k}
      \big]
    }
    \;\leq\;
    \varrho |\cM|\, \mu[M_k],
  \end{align*}
  which concludes the proof.
\end{proof} 
The capacitary inequality combined with the definition of metastable sets yields that the harmonic functions, $h_{M_i, M_j}$, is almost constant on the valleys~$\cS_i$ and $\cS_j$.
\begin{lem}[$\ell^p$-norm estimate]\label{lem:est:L1-L2-h_ab}
  For any $M_i \in \cM$ and $f \in \ell^2(\mu)$ with $f(x) = 0$ for all $x \in M_i$,
  \begin{align}\label{eq:est:L2:sets}
    \Mean_{\mu_i}\!\big[f^2\big]
    \;\leq\;
     \frac{4\, \varrho}{\mu[\cS_i]}\;
     \bigg(
       \max_{M \in \cM}
       \prob_{\mu_{M}}\!\big[
         \tau_{{\scriptscriptstyle \bigcup_{j=1}^K} M_j \setminus M} < \tau_{M}
       \big]
     \bigg)^{\!\!-1}
    \cE(f).
  \end{align}
  In particular, for any $M_i, M_j \in \cM$ with $i \ne j$,
  \begin{align}
    \label{eq:est:L1/2:h_ab:sets}
    \Mean_{\mu_i}\pra[\big]{h_{M_j, M_i}^p}
    &\;\leq\;
    \varrho\, \frac{p}{p-1}\, 
    \min\bigg\{1, \frac{\mu[\cS_j]}{\mu[\cS_i]}\bigg\},
    \qquad \forall\, p > 1,
    \intertext{and}
    \label{eq:est:L1:sets}
    \Mean_{\mu_i}\pra[\big]{h_{M_j, M_i}}
    &\;\leq\;
    \varepsilon
    \,+\,
    \varrho\, \ln(1/\varepsilon)\, 
    \min\bigg\{1, \frac{\mu[\cS_j]}{\mu[\cS_i]}\bigg\},
    \qquad \forall\, \varepsilon \in (0, 1],
  \end{align}
  where $h_{M_j, M_i}$ denotes the equilibrium potential of the pair $(M_j, M_i)$.
\end{lem}
\begin{proof}
  First, notice that for any $A \subset \valley_i \setminus M_i$
  \begin{align*}
    \mu_i[A]
    \;=\;
    \frac{\capacity(A, M_i)}{\mu[\cS_i]\, \prob_{\mu_A}[\tau_{M_i} < \tau_A]}
    \overset{\eqref{eq:metaest:sets}}{\;\leq\;}
    \frac{\varrho}{\mu[\cS_i]}\;
    \bigg(
      \max_{M \in \cM} 
      \prob_{\mu_{M}}\!\big[
        \tau_{{\scriptscriptstyle \bigcup_{j=1}^K} M_j \setminus M} < \tau_{M}
      \big]
    \bigg)^{\!\!-1}\mspace{-14mu}
    \capacity(A, M_i).
  \end{align*}
  Thus, \eqref{eq:est:L2:sets} follows from Proposition~\ref{prop:OrliczBirnbaum} by choosing $(\Phi, \Psi) = (\Phi_1, \Psi_1)$ as in Example~\ref{ex:OrliczPairs}~a).  In the sequel, we aim at proving \eqref{eq:est:L1/2:h_ab:sets} and \eqref{eq:est:L1:sets}.  For any $t \in [0, 1]$ we write $A_t \ldef \set{x \in \cS : h_{M_j, M_i} > t}$ to denote the super level-sets of $h_{M_j, M_i}$, and set $h_t \ldef h_{A_t, M_i}$.  Then, 
  \begin{align}\label{eq:level:set:h}
    t \capacity(A_t, M_i) 
    \;\leq\; 
    \skp{-L h_t, h_{M_j, M_i}}_{\mu}
    \;=\;
    \skp{h_t, -L h_{M_j, M_i}}_{\mu}
    \;=\;
    \capacity(M_j, M_i).
  \end{align}
  Thus, for any $p > 1$, we obtain
  \begin{align*}
    \Mean_{\mu_i}\pra[\big]{h_{M_j, M_i}^p}
    &\;=\;
    \int_0^1 p t^{p-1}\, \mu_i[A_t]\, \dx{t}
    \\
    &\;\leq\;
    \frac{\varrho}{\mu[\cS_i]}\;
    \bigg(
      \max_{M \in \cM} 
      \prob_{\mu_{M}}\!\big[
        \tau_{{\scriptscriptstyle \bigcup_{j=1}^K} M_j \setminus M} < \tau_{M}
      \big]
    \bigg)^{\!\!-1}
    \int_0^1 p t^{p-1}\, \capacity(A_t, M_i)\, \dx{t}.
  \end{align*}
  Since,
  \begin{align*}
    \max_{M \in \cM}
    \prob_{\mu_M}\!\big[
      \tau_{{\scriptscriptstyle \bigcup_{j=1}^K} M_j \setminus M} < \tau_{M}
    \big]
    \;\geq\;
    \max\big\{
      \prob_{\mu_{M_i}}[\tau_{M_j} < \tau_{M_i}],\,
      \prob_{\mu_{M_j}}[\tau_{M_i} < \tau_{M_j}]
    \big\}
  \end{align*}
  we deduce that
  \begin{align*}
    \Mean_{\mu_i}\pra[\big]{h_{M_j, M_i}^p}
    &\overset{\eqref{eq:level:set:h}}{\;\leq\;}
    \varrho\, \min\bigg\{ 1, \frac{\mu[\cS_j]}{\mu[\cS_i]} \bigg\}\, 
    \int_0^1 p t^{p-2}\, \dx{t}
    \;=\;
    \varrho\, \frac{p}{p-1}\, 
    \min\bigg\{1, \frac{\mu[\cS_j]}{\mu[\cS_i]}\bigg\},
  \end{align*}
  which concludes the proof of \eqref{eq:est:L1/2:h_ab:sets}.  Likewise, we obtain for any $\varepsilon \in (0, 1]$ that
  \begin{align*}
    \Mean_{\mu_i}\pra[\big]{h_{M_j, M_i}}
    &\;=\;
    \varepsilon \,+\, \int_{\varepsilon}^1 \mu_i[A_t]\, \dx{t}
    \overset{\eqref{eq:level:set:h}}{\;\leq\;}
    \varepsilon \,+\, \varrho\,
    \min\bigg\{ 1, \frac{\mu[\cS_j]}{\mu[\cS_i]} \bigg\}
    \int_{\varepsilon}^1 t^{-1}\, \dx{t},
  \end{align*}
  and \eqref{eq:est:L1:sets} follows.
\end{proof}
The bound~\eqref{eq:est:L1:sets} of Lemma~\ref{lem:est:L1-L2-h_ab} provides the main ingredient for the proof of Theorem~\ref{thm:mean:hitting:times}.
\begin{proof}[Proof of Theorem~\ref{thm:mean:hitting:times}]
  Let $B$ and $J \equiv J(i)$ be defined as in Theorem~\ref{thm:mean:hitting:times}.  By \cite[Corollary~7.11]{BdH15} we have that
  \begin{align*}
    \mean_{\nu_{M_i,B}}\pra{\tau_{B}}
    \;=\;
    \frac{\Mean_{\mu}\pra{h_{M_i, B}}}{\capacity(M_i, B)}
    \;=\;
    \frac{\mu[\cS_i]}{\capacity(M_i, B)}\,
    \bra[\bigg]{
      \Mean_{\mu_i}\pra{h_{M_i, B}} 
      \,+\, 
      \sum_{j \ne i}\, \frac{\mu[\cS_j]}{\mu[\cS_i]}\, \Mean_{\mu_j}\pra{h_{M_i, B}}
    }.
  \end{align*}
  In order to prove a lower bound, we neglect the last term in the bracket above.  Since $\prob_{x}[\tau_{{\scriptscriptstyle \bigcup_{j \in J}} M_j} < \tau_{M_i}] \leq \sum_{j \in J} \prob_{x}[\tau_{M_j} < \tau_{M_i}]$, Lemma~\ref{lem:est:L1-L2-h_ab} implies with $\varepsilon = \varrho$
  \begin{align*}
    \Mean_{\mu_i}\pra{h_{M_i, B}}
    \;=\;
    1 - \Mean_{\mu_i}\pra{h_{B, M_i}}
    \;\geq\;
    1 - \sum_{j \in J}\, \Mean_{\mu_i}\pra{h_{M_j, M_i}}
    \overset{\eqref{eq:est:L1:sets}}{\;\geq\;}
    1 - |\cM|\, \varrho \big(1 + \ln 1/\varrho\big).
  \end{align*}
  Hence, we conclude that
  \begin{align*}
    \mean_{\nu_{M_i, B}}[\tau_B] 
    \;\geq\; 
    \frac{\mu[\cS_i]}{\capacity(M_i, B)}\, 
    \bra[\Big]{1 - |\cM| \big(\varrho +\varrho \ln 1/\varrho\big)}.
  \end{align*}
  Concerning the upper bound, recall that by assumption $\mu[\cS_j] / \mu[\cS_i] \leq \delta$ for all $j \not\in J \cup \{i\}$.  Thus, by Lemma~\ref{lem:est:L1-L2-h_ab} with $\varepsilon = \varrho / C_{\textup{ratio}}$, we get
  \begin{align*} 
    \sum_{j \ne i}\, \frac{\mu[\cS_j]}{\mu[\cS_i]}\, \Mean_{\mu_j}\pra{h_{M_i, B}}
    \;\leq\;
    |\cM|\, \delta 
    +
    \sum_{j \in J}\, \frac{\mu[\cS_j]}{\mu[\cS_i]}\, \Mean_{\mu_j}\pra{h_{M_i, M_j}}
    \overset{\eqref{eq:est:L1:sets}}{\;\leq\;}
    |\cM|\, 
    \bigg(
      \delta + \varrho + \varrho \ln \frac{C_{\textup{ratio}}}{\varrho} 
    \bigg).
  \end{align*}
  Since $\Mean_{\mu_i}[h_{M_i, B}] \leq 1$, the proof concludes with the estimate
  \begin{align*}
    \mean_{\nu_{M_i, B}}[\tau_B] 
    \;\leq\; 
    \frac{\mu[\cS_i]}{\capacity(M_i, B)}\, 
    \bra[\Big]
    {
      1 
      + 
      |\cM|\,\big(\delta + \varrho + \varrho \ln(C_{\textup{ratio}} / \varrho)\big)
    }.
  \end{align*}
\end{proof}
Let us define neighborhoods of the metastable sets in terms of level sets of harmonic functions. Therefore, we consider two non-empty, disjoint subsets $\cA, \cB \subset \cM$ of the set of metastable sets, and let $I_{\cA}, I_{\cB} \subset \{1, \ldots, K\}$ be such that $\cA = \{M_i : i \in I_{\!\cA}\}$ and $\cB = \{M_i : i \in I_{\!\cB}\}$.  Further, set $A = \bigcup_{M \in \cA} M$ and $B = \bigcup_{M \in \cB} M$.  For $\delta \in (0, 1)$ define the harmonic neighborhood of $A$ relative to $B$ by
\begin{align}\label{eq:def:neighborhood}
  \cU_{A}(\delta, B)
  \;\ldef\;
  \Big\{
    x \in {\textstyle \bigcup_{i \in I_{\!\cA}}} \cS_i 
    \,:\, 
    h_{A, B}(x) \geq 1 - \delta 
  \Big\}.
\end{align}
The following lemma shows that the capacity of $(\cU_A(\delta,B), \cU_B(\delta,A))$ is comparable to the capacity of $(A, B)$.
\begin{lem}\label{lem:harm:neighborhood}
  Let $M_i \in \cM$ and $\cB \subset \cM \setminus \{M_i\}$.  Then, for any $\delta \in (0, 1/2)\,$,
  \begin{align}\label{eq:harm:neighborhood:cap}
    1 - 2\delta
    \;\leq\;
    \frac{\capacity(M_i, B)}
    {\capacity(\cU_{M_i}(\delta, B), \cU_{B}(\delta, M_i))}
    \;\leq\;
    1.
  \end{align}
  Moreover, for $X \ldef \cS_i \setminus \cU_{M_i}(\delta, B)$,
  \begin{align}\label{eq:harm:neighborhood:mu}
    \mu[X]
    \;\leq\;
    \varrho\, \delta^{-1}\, \mu[M_i].
  \end{align}
\end{lem}
\begin{proof}
  Since $M_i \subset \cU_{M_i}(\delta, B)$ and $B \subset \cU_B(\delta, M_i)$ by definition, the upper bound in \eqref{eq:harm:neighborhood:cap} follows from the monotonicity of the capacity, see \eqref{eq:cap:monotonicity}.  In order to prove the lower bound in \eqref{eq:harm:neighborhood:cap}, notice that
  \begin{align*}
    h_{M_i, B}(x) \;\geq\; 1-\delta,
    \quad \forall\, x \in \cU_{M_i}(\delta, B)
    \qquad \text{and} \qquad
    h_{M_i, B}(x) \;\leq\; \delta,
    \quad  \forall\, x \in \cU_{B}(\delta, M_i).
  \end{align*}
  Thus, by using the symmetry of $-L$ in $\ell^2(\mu)$, we obtain
  \begin{align*}
    \capacity(M_i, B)
    &\;=\;
    \skp{-L h_{M_i, B}, h_{\cU_{M_i}(\delta), \cU_B(\delta)}}_{\mu}
    \\[.5ex]
    &\;=\;
    \skp{h_{M_i, B}, -Lh_{\cU_{M_i}(\delta), \cU_B(\delta)}}_{\mu}
    \;\geq\;
    \capacity(\cU_{M_i}(\delta), \cU_{B}(\delta))\,\big(1-2\delta\big).
  \end{align*}
  The proof of \eqref{eq:harm:neighborhood:mu} is similar to the one of Lemma~\ref{lem:negligible:points}.  Since $h_{M_i, B}(x) \leq 1-\delta$ for any $x \in X = \cS_i \setminus \cU_{M_i}(\delta, B)$, we get
  \begin{align*}
    \capacity(M_i, B)
    &\;\geq\;
    \skp{-L h_{B, M_i}, h_{X, M_i}}_{\mu}
    \;=\;
    \skp{h_{B, M_i}, -L h_{X, M_i}}_{\mu}\\
    &\;\geq\;
    \delta\, \capacity(X, M_i)
    \overset{\eqref{eq:metaest:sets}}{\;\geq\;}
    \frac{\delta\, \mu[X]}{\varrho}\,
    \prob_{\mu_{M_i}}\!\big[\tau_B < \tau_{M_i}\big].
  \end{align*}
  Thus, the assertion follows from \eqref{eq:Capacity:HittingTimes}.
\end{proof}

\subsection{Poincar\'e inequality}
\label{s:PI}
In this section we denote by $c$ a numerical finite constant, which may change from line to line.
\begin{thm}\label{thm:PI:metastable}
  Suppose that Assumption~\ref{ass:regularity} and \ref{ass:PI+LSI:metastable:sets}~i) hold.  Then,
  \begin{align}
    \label{eq:thm:PI:constant:sets:lb}
    C_{\PI}
    &\;\geq\;
    \max_{\substack{i,j \in \{1, \ldots, K\} \\ i \ne j}}\,
    \frac{\mu[\cS_{i}]\, \mu[\cS_{j}]}{\capacity(M_i, M_{j})}\,
    \Big( 1 - c \sqrt{\varrho} \Big),
    \\
    \label{eq:thm:PI:constant:sets:ub}
    C_{\PI}
    &\;\leq\;
    \frac{1}{2}\,
    \sum_{\substack{i,j=1\\i \ne j}}^K
    \frac{\mu[\cS_{i}]\, \mu[\cS_{j}]}{\capacity(M_i,M_j)}\,
    \bra*{ 1 + c \sqrt{C_{\PI,\cM}\bra{\varrho  + \eta}} }.
  \end{align}
\end{thm}
\begin{rem}
  It is possible to formulate a result with asymptotically matching upper and lower bounds for $C_{\PI}$ under suitable non-degeneracy assumption. These essentially demand that one of the term in the right-hand side of~\eqref{eq:thm:PI:constant:sets:ub} dominates the others.
\end{rem}
Let $\cG \ldef \sigma(\cS_i : i = 1, \ldots, K)$ be the $\sigma$-algebra generated by the sets of the metastable partition.  Since $M_i \in \cS_i$ for all $i=1,\ldots, K$, we have $\cG \subset \cF$ by construction of $\cF$, cf.~\eqref{def:F:cond:EX}.  We denote by $\Mean_{\mu}[f \,|\, \cG]$ the conditional expectation given $\cG$.  In order to prove Theorem~\ref{thm:PI:metastable} we use again the projection property of the conditional expectation to further split the variance $\var_{\mu}[\Mean_{\mu}[f \,|\, \cF]]$ into the local variances and the mean difference
\begin{align}\label{eq:var:split2}
  &\var_{\mu}\!\big[\Mean_{\mu}[f \,|\, \cF]\big]
  \nonumber\\[.5ex]
  &\mspace{36mu}=\;
  \sum_{i=1}^K\, \mu[\cS_i]\, \var_{\mu_i}\!\big[\Mean[f \,|\, \cF]\big]
  \,+\,
  \frac{1}{2}\, \sum_{i,j = 1}^K
  \mu[\cS_i]\, \mu[\cS_j]\,
  \bra*{\Mean_{\mu_i}[f] - \Mean_{\mu_j}[f]}^{\!2}.
\end{align}
Therewith, the proof of Theorem~\ref{thm:PI:metastable} consists in bounding both the local variances and the mean difference in terms of the Dirichlet form. Bounding the local variances is established by local Poincar\'e inequalities, which are a consequence of Lemma~\ref{lem:est:L1-L2-h_ab}.
\begin{lem}[Local Poincar\'e inequality]\label{lemma:local:PI}
  Suppose that Assumption~\ref{ass:PI+LSI:metastable:sets}~i) is satisfied.  Then, for any $f \in \ell^2(\mu)$ and $i \in \{1, \ldots, K\}$,
  \begin{align}\label{eq:local:PI:metastable}
    \var_{\mu_i}\pra*{\Mean_{\mu}[f \,|\, \cF]}
    &\;\leq\;
    \Mean_{\mu_i}\pra*{\bra[\big]{\Mean_{\mu}[f \,|\, \cF] - \Mean_{\mu_{M_i}}\![f]}^2}
    \nonumber\\[.5ex]
    &\;\leq\;
    c \, \frac{C_{\PI, \cM} \, \varrho }{\mu[\cS_i]}\;
    \bigg(\!
      \max_{M \in \cM}
      \prob_{\mu_{M}}\!\big[
        \tau_{{\scriptscriptstyle \bigcup_{j=1}^K} M_j \setminus M} < \tau_{M}
      \big]
    \bigg)^{\!\!-1}\!
    \cE(f).
  \end{align}
\end{lem}
\begin{proof}
  By noting that $\var_{\mu_i}[\Mean_{\mu}[f \,|\, \cF]] = \min_{a \in \mathbb{R}}\Mean_{\mu_i}[(\Mean_{\mu}[f \,|\, \cF] - a)^2]$, the first estimate in~\eqref{eq:local:PI:metastable} is immediate.  Moreover, the function $x \mapsto \Mean_{\mu}[f \,|\, \cF](x) - \Mean_{\mu_{M_i}}\![f]$ vanishes on $M_i$.  Hence, by \eqref{eq:est:L2:sets} we obtain
  \begin{align*}
    \Mean_{\mu_i}\pra*{(\Mean_{\mu}[f \,|\, \cF] - \Mean_{\mu_{M_i}}\![f])^2}
    \;\leq\;
    \frac{4 \varrho}{\mu[\cS_i]}\;
    \bigg(
       \max_{M \in \cM}
       \prob_{\mu_{M}}\!\big[
         \tau_{{\scriptscriptstyle \bigcup_{j=1}^K} M_j \setminus M} < \tau_{M}
       \big]
     \bigg)^{\!\!-1}
    \cE\big(\!\Mean_{\mu}[f \,|\, \cF] \big).
  \end{align*}
  Thus, we are left with bounding $\cE\big(\Mean_{\mu}[f \,|\, \cF])$ from above by $\cE(f)$. For any $\delta > 0$ by Young's inequality, that reads $|a b| \leq \delta a^2 + b^2 / (4\delta)$, we get for any $x, y \in \cS$
  \begin{align*}
    &\bra*{\Mean_{\mu}[f \,|\, \cF](x) \,-\, \Mean_{\mu}[f \,|\, \cF](y)}^{2}
    \\
    &\mspace{36mu}\leq\;
    \big(1 + 2\delta \big)\, \big(f(x) - f(y))^2
    \,+\,
    \Big(2 + \frac{1}{\delta}\Big)\,
    \sum_{z \in \{x,y\}}\! \big(f(z) - \Mean_{\mu}[f \,|\, \cF](z) \big)^{\!2}.
  \end{align*}
  Recall that $\Mean_{\mu}[f \,|\, \cF](x) = \Mean_{\mu_{M_i}}[f]$ for any $x \in M_i$.  Since $f(x) - \Mean_{\mu}[f \,|\, \cF](x) = 0$ for any $x \in \cS \setminus \bigcup_{i=1}^K M_i$, we obtain 
  \begin{align*}
    \cE\bra*{\Mean_{\mu}[f \,|\, \cF]}
    &\;\leq\;
    \big(1 + 2\delta \big)\, \cE(f)
    \,+\,
    \Big(2 + \frac{1}{\delta}\Big)\,
    \sum_{i=1}^K \mu[M_i]\, \var_{\mu_{M_i}}[f].
  \end{align*}
  Since $\var_{\mu_{M_i}}\![f] \leq C_{\PI,i}\, \cE(f)$ for any $i = 1, \ldots, K$, the assertion \eqref{eq:local:PI:metastable} follows by choosing $\delta = \sqrt{2 C_{\PI, \cM}}$.
\end{proof}
\begin{lem}[Mean difference estimate]\label{lemma:mean:diff}
  Let Assumptions~\ref{ass:regularity} and \ref{ass:PI+LSI:metastable:sets}~i) be satisfied.  Then, for any $f \in \ell^2(\mu)$ and $M_i, M_j \in \cM$ with $i \ne j$ it holds that
  \begin{align*}
    \bra*{\Mean_{\mu_i}[f] \,-\, \Mean_{\mu_j}[f]}^{\!2}
    \;\leq\;
    \frac{\cE(f)}{\capacity(M_i, M_j)}\;
    \bra*{1 + c \, \sqrt{ C_{\PI,\cM}\bra{\varrho + \eta}}}.
  \end{align*}
\end{lem}
\begin{proof}
  For $M_i, M_j \in \cM$ with $i \ne j$ let $\nu_{M_i, M_j}$ be the last-exit biased distribution as defined in \eqref{eq:def:LastExitBiasedDistri}, and denote by $g_{i,j} \ldef \nu_{M_i, M_j} / \mu_{\mu_{M_i}}$ the relative density of $\nu_{M_i,M_j}$ with respect to $\mu_{M_i}$.  Then, it holds that
  \begin{align*}
    \Mean_{\mu_i}[f]
    &\;=\;
    \Mean_{\mu_i}\pra*{\Mean_{\mu}[f \,|\, \cF]}
    \\[.5ex]
    &\;=\;
    \Mean_{\nu_{M_i, M_j}}\![f]
    \,+\,
    \Mean_{\mu_i}\pra*{ \Mean_{\mu}[f \,|\, \cF] - \Mean_{\mu_{M_i}}[f] }
    \,-\, \Mean_{\mu_{M_i}}\pra*{ \big(g_{i,j} - 1\big)\, f }.
  \end{align*}
  Thus, by applying Young's inequality, we obtain for any $\delta > 0$ and $f \in \ell^2(\mu)$,
  \begin{align*}
    \bra*{\Mean_{\mu_i}[f] - \Mean_{\mu_j}[f]}^{\!2}
    &\;\leq\;
    (1 + \delta)\,
    \bra*{\Mean_{\nu_{M_i, M_j}}\![f] \,-\, \Mean_{\nu_{M_j, M_i}}[f]}^{\!2}
    \\[.5ex]
    &\mspace{36mu}+\,
    2\Big(1 + \frac{1}{\delta}\Big)
    \sum_{k \in \{i,j\}}\mspace{-3mu}
    \Mean_{\mu_k}\pra*{ \Mean_{\mu}[f \,|\, \cF] - \Mean_{\mu_{M_k}}[f] }^2
    \\[.5ex]
    &\mspace{36mu}+\,
    2\Big(1 + \frac{1}{\delta}\Big)\,
    \Big(
      \Mean_{\mu_{M_i}}\pra*{\big(g_{i,j} - 1\big)f }^2
      \,+\, \Mean_{\mu_{M_j}}\pra*{\big(g_{j,i} - 1\big)f }^2
    \Big).
  \end{align*}
  Let $h_{M_i, M_j}$ be the equilibrium potential of the pair $(M_i, M_j)$.  Observe that a summation by parts together with an application of the Cauchy--Schwarz inequality yields
  \begin{align*}
    \bra*{ \Mean_{\nu_{M_i, M_j}}\![f] \,-\, \Mean_{\nu_{M_j,M_i}}\![f] }^{\!2}
    \;\leq\;
    \frac{\cE(f)}{\capacity(M_i, M_j)}.
  \end{align*}
  Recall that the function $x \mapsto \Mean_{\mu}[f \,|\, \cF](x) - \Mean_{\mu_{M_i}}[f]$ vanishes on $M_i$.  Thus, \eqref{eq:local:PI:metastable} implies that
  \begin{align*}
    \Mean_{\mu_i}\pra*{
      \big(\Mean_{\mu}[f \,|\, \cF] - \Mean_{\mu_{M_i}}[f]\big)^2
    }
    \;\leq\;
    c\, C_{\PI, \cM}\, \varrho\,
    \frac{\cE(f)}{\capacity(M_i, M_j)},
  \end{align*}
  where we used that $\max_{M \in \cM} \prob_{\mu_{M}}\!\big[ \tau_{{\scriptscriptstyle \bigcup_{j=1}^K} M_j \setminus M} < \tau_{M} \big] \geq \capacity(M_i, M_j) / \mu[M_i]$.  Further, the covariance between $g_{i,j}$ and $f$, thanks to Assumptions~\ref{ass:regularity} and~\ref{ass:PI+LSI:metastable:sets}~i), is bounded from above by
  \begin{align*}
    \Mean_{\mu_{M_i}}\pra*{\big(g_{i,j} - 1\big)f }^2
    \;\leq\;
    \var_{\mu_{M_i}}\![g_{i,j}]\,
    \var_{\mu_{M_i}}\![f]
    \;\leq\;
    \eta\, \mu[M_i]\, C_{\PI, i}\, \frac{\cE(f)}{\capacity(M_i, M_j)}.
  \end{align*}
  By combining the estimates above and choosing $\delta = \sqrt{C_{\PI, \cM}\bra{\varrho  + \eta}}$, we obtain the assertion.
\end{proof}
A combination of the splitting~\eqref{eq:var:split2} with the Lemmas~\ref{lemma:local:PI} and~\ref{lemma:mean:diff} gives the upper bound~\eqref{eq:thm:PI:constant:sets:ub} of Theorem~\ref{thm:PI:metastable}. The proof is complemented by a suitable test function yielding the lower bound~\eqref{eq:thm:PI:constant:sets:lb}.
\begin{proof}[Proof of Theorem~\ref{thm:PI:metastable}]
  The lower bound of $C_{\PI}$ is an immediate consequence of the variational definition of $C_{\PI}$; cf.\ \eqref{eq:def:PI-constant}.  Indeed, by choosing the equilibrium potential $h_{M_i, M_j}$ for any $M_i, M_j \in \cM$ with $i \ne j$ as a test function, we deduce from \eqref{eq:var:split1} and \eqref{eq:var:split2} that
  \begin{align*}
    \var_{\mu}\pra*{h_{M_i, M_j}}
    &\;\geq\;
    \mu[\cS_i]\, \mu[\cS_j]\;
    \bra*{ \Mean_{\mu_i}[h_{M_i, M_j}] - \Mean_{\mu_j}[h_{M_i, M_j}] }^{\!2}
    \\[.5ex]
    &\;=\;
    \mu[\cS_i]\, \mu[\cS_j]\;
    \bra*{ 1 - \Mean_{\mu_i}[h_{M_j, M_i}] - \Mean_{\mu_j}[h_{M_i, M_j}] }^{\!2}.
  \end{align*}
  Thus, in view of \eqref{eq:est:L1:sets}, we obtain that $\var_{\mu}[f] \geq \mu[\cS_i] \mu[\cS_j] \big(1 - 8 \varrho \big)^2$.  Since $\cE(h_{M_i, M_j}) = \capacity(M_i, M_j)$, \eqref{eq:thm:PI:constant:sets:lb} follows by optimizing over all $M_i \ne M_j \in \cM$.

  For the upper bound, observe that by using~\eqref{eq:Capacity:HittingTimes} and~\eqref{eq:cap:monotonicity},
  \begin{align}\label{eq:HitProbCap}
    &\bigg(\!
      \max_{M \in \cM}
      \prob_{\mu_{M}}\!\big[
        \tau_{{\scriptscriptstyle \bigcup_{j=1}^K} M_j \setminus M} < \tau_{M}
      \big]
    \bigg)^{\!\!-1}
    \nonumber\\[.5ex]
    &\mspace{36mu}\leq\;
    \frac{1}{|\cM|-1}\,
    \sum_{\substack{i,j=1 \\ i \ne j}}^K \, 
    \frac{\mu[\cS_i]}{\prob_{\mu_{M_j}}\pra[\big]{\tau_{M_i} < \tau_{M_j}}}
    \;\leq\;
    \frac{1}{|\cM|-1}\,
    \sum_{\substack{i,j=1 \\ i \ne j}}^K \,
    \frac{\mu[\cS_i]\, \mu[\cS_j]}{\capacity(M_i, M_j)}.
  \end{align}
  Hence, by an application of Lemma~\ref{lemma:local:PI}, it follows that
  \begin{align*}
    \sum_{i=1}^K
    \mu[\cS_i]\, \var_{\mu_i}\pra*{\Mean_{\mu}[f \,|\, \cF]}
    \overset{\eqref{eq:local:PI:metastable}}{\;\leq\;}
    c\,C_{\PI,\cM} \, \varrho \, \frac{1}{2}
    \sum_{\substack{i,j=1 \\ i \ne j}}^K
    \frac{\mu[\cS_i]\, \mu[\cS_j]}{\capacity(M_i, M_j)}\;
    \cE(f).
  \end{align*}
  Thus, a combination of \eqref{eq:var:split1} and \eqref{eq:var:split2} together with Lemma~\ref{lemma:mean:diff} yields \eqref{eq:thm:PI:constant:sets:ub} up to an additive factor $C_{\PI,\cM}$. To bound this additive error term, notice that
  \begin{align*}
    C_{\PI,\cM} 
    \;\leq\;  
    C_{\PI,\cM}\, \varrho\,
    \bigg(\!
      \max_{M \in \cM}
      \prob_{\mu_{M}}\!\big[
        \tau_{{\scriptscriptstyle \bigcup_{j=1}^K} M_j \setminus M} < \tau_{M}
      \big]
    \bigg)^{\!\!-1}
    \overset{\eqref{eq:HitProbCap}}{\!\leq\;}
    \frac{C_{\PI,\cM} \, \varrho}{|\cM| - 1}\, 
    \sum_{\substack{i,j=1 \\ i \ne j}}^K 
    \frac{\mu[\cS_i] \, \mu[\cS_j]}{\capacity(M_i, M_j)},
  \end{align*}
  which shows that $C_{\PI,\cM}$ can be absorbed into the right-hand side of~\eqref{eq:thm:PI:constant:sets:ub}.
\end{proof}

\subsection{Logarithmic Sobolev inequality}
\label{s:LSI}
In this subsection, we focus on sharp estimates of the logarithmic Sobolev constant in the context of metastable Markov chains. Again, we denote by $c$ a numerical finite constant, which may change from line to line.
\begin{thm}\label{thm:LSI:metastable}
  Suppose that the Assumptions~\ref{ass:regularity}, \ref{ass:PI+LSI:metastable:sets} and~\eqref{ass:Cmass} hold.  Then,
  \begin{align}
    \label{eq:thm:LSI:constants:lb}
    C_{\LSI}
    &\;\geq\;
    \max_{\substack{i,j \in \{1, \ldots, K\} \\ i \ne j}}\,
    \frac{\mu[\cS_{i}]\, \mu[\cS_{j}]}{\Lambda(\mu[\cS_i], \mu[\cS_j])}\,
    \frac{1}{\capacity(M_i, M_{j})}\,
    \Big( 1 - c \sqrt{\varrho} \Big)^2
    \\
    \label{eq:thm:LSI:constants:ub}
    C_{\LSI}
    &\;\leq\;
    \sum_{\substack{i,j=1\\i \ne j}}^K
    \frac{\mu[\cS_{i}]\, \mu[\cS_{j}]}{\Lambda(\mu[\cS_i], \mu[\cS_j])}\,
    \frac{1}{\capacity(M_i,M_j)}\,
    \Big( 1 + c\sqrt{C_{\textup{mass}} \, C_{\LSI,\cM}\bra{\varrho + \eta} } \Big) .
  \end{align}
\end{thm}
In order to proof Theorem~\ref{thm:LSI:metastable}, we decompose the entropy $\Ent_{\mu}\!\big[\Mean_{\mu}[f^2 \,|\, \cF]\big]$ in~\eqref{eq:ent:split1} into the local entropies within the sets $\cS_1, \ldots, \cS_K$ and the macroscopic entropy
\begin{align*}
  \Ent_{\mu}\pra*{\Mean_{\mu}[f^2 \,|\, \cF]}
  \;=\;
  \sum_{i=1}^K\, \mu[\cS_i]\, \Ent_{\mu_i}\pra*{\Mean_{\mu}[f^2 \,|\, \cF]}
  \,+\,
  \Ent_{\mu}\pra*{\Mean_{\mu}[f^2 \,|\, \cG]}.
\end{align*}
In the next lemma we derive an upper bound on the local entropies.
\begin{lem}[Local logarithmic Sobolev inequality]\label{lemma:local:LSI}
  Let Assumption~\ref{ass:PI+LSI:metastable:sets}~i) be satisfied, and assume that $C_{\textup{mass}} < \infty$. Then, for any $f \in \ell^2(\mu)$ and $i \in \{1, \ldots, K\}$,
  \begin{align}\label{eq:local:LSI:metastable}
    \Ent_{\mu_i}\pra*{\Mean_{\mu}[f^2 \,|\, \cF]}
    \;\leq\;
    c \frac{C_{\textup{mass}}\, C_{\PI, \cM} \, \varrho}{\mu[\cS_i]}\;
    \bigg(\!
      \max_{M \in \cM}
      \prob_{\mu_{M}}\!\big[
        \tau_{{\scriptscriptstyle \bigcup_{j=1}^K} M_j \setminus M} < \tau_{M}
      \big]
    \bigg)^{\!\!-1}\!
    \cE(f).
  \end{align}
\end{lem}
\begin{proof}
  First, notice that for any $A \subset \cS \setminus M_i$,
  \begin{align*}
    \mu_i[A]\, \ln \bigg(1 + \frac{\me^2}{\mu_i[A]}\bigg)
    &\;\leq\;
    \frac{\max_{x \in \cS_i} \ln(1 + \me^2 / \mu_i(x))}
    {\mu[\cS_i]\, \capacity(A \cap \cS_i, M_i)}\,
    \capacity(A, M_i)
    \\
    &\overset{\!\eqref{eq:metaest:sets}\!}{\;\leq\;}
    C_{\textup{mass}}\, \frac{\varrho}{\mu[\cS_i]}\;
    \bigg(
      \max_{M \in \cM} 
      \prob_{\mu_{M}}\!\big[
        \tau_{{\scriptscriptstyle \bigcup_{j=1}^K} M_j \setminus M} < \tau_{M}
      \big]
    \bigg)^{\!\!-1}\mspace{-10mu}
    \capacity(A, M_i).
  \end{align*}
  Since the function $x \mapsto \Mean_{\mu}[f^2 \,|\, \cF](x)$ is constant on $M_i$, Corollary~\ref{cor:LSI_via_CapInequ} implies that
  \begin{align*}
    \Ent_{\mu_i}\pra*{\Mean_{\mu}[f^2 \,|\, \cF] }
    \;\leq\;
    C_{\textup{mass}}\, \frac{4 \varrho}{\mu[\cS_i]}\;
    \bigg(
      \max_{M \in \cM} 
      \prob_{\mu_{M}}\!\big[
        \tau_{{\scriptscriptstyle \bigcup_{j=1}^K} M_j \setminus M} < \tau_{M}
      \big]
    \bigg)^{\!\!-1}
    \cE\Big(\!\sqrt{\Mean_{\mu}[f^2 \,|\, \cF]}\,\Big).
  \end{align*}
  Thus, we are left with bounding $\cE\big(\!\sqrt{\Mean_{\mu}[f^2 \,|\, \cF]}\big)$ from above with $\cE(f)$.  Applying Young's inequality, we get, for any $\delta > 0$ and $x, y \in \cS$,
  \begin{align*}
    &\Big(
      \sqrt{\Mean_{\mu}[f^2 \,|\, \cF](x)}
      \,-\, \sqrt{\Mean_{\mu}[f^2 \,|\, \cF](y)}\,
    \Big)^{\!2}
    \\
    &\mspace{36mu}\leq\;
    \big(1 + 2\delta \big)\, \big(|f(x)| - |f(y)|)^2
    \,+\,
    \Big(2 + \frac{1}{\delta}\Big)\,
    \sum_{z \in \{x,y\}}\!
    \Big( |f(z)| - \sqrt{\Mean_{\mu}[f^2 \,|\, \cF](z)}\, \Big)^{\!2}.
  \end{align*}
  Since $|f(z)| - \sqrt{\Mean_{\mu}[f^2 \,|\, \cF](z)} = 0$ for any $z \in \cS \setminus \bigcup_{i=1}^K M_i$ and $\Mean_{\mu}[f \,|\, \cF](x) = \Mean_{\mu_{M_i}}[f]$ for any $x \in M_i$, we obtain
  \begin{align*}
    \cE\Big(\!\sqrt{\Mean_{\mu}[f^2 \,|\, \cF]}\, \Big)
    &\;\leq\;
    \big(1 + 2\delta \big)\, \cE(|f|)
    \,+\,
    2\,\Big(2 + \frac{1}{\delta}\Big)\,
    \sum_{i=1}^K \mu[M_i]\, \var_{\mu_{M_i}}[f],
  \end{align*}
  where we additionally exploited the fact that, by Jensens' inequality, 
  \begin{align*}
    \Mean_{\mu_{M_i}}\pra*{\Big(|f| - \sqrt{\Mean_{\mu_{M_i}}[f^2]}\Big)^2}
    \;\leq\;
    2\, \var_{\mu_{M_i}}[f].
  \end{align*}
  Since $\var_{\mu_{M_i}}\![f] \leq C_{\PI,i}\, \cE(f)$ for any $i = 1, \ldots, K$ and $\cE(|f|) \leq \cE(f)$, \eqref{eq:local:LSI:metastable} follows by choosing $\delta = \sqrt{2 C_{\PI, \cM}}$.
\end{proof}
\begin{proof}[Proof of Theorem~\ref{thm:LSI:metastable}]
  In view of the variational definition of $C_{\LSI}$ (cf.\ \eqref{eq:def:LSI-constant}), \eqref{eq:thm:LSI:constants:lb} follows from the construction of a suitable test function.  For any $M_i, M_j \in \cM$ with $i \ne j$, $\delta \in [0, 1/2)$ and $g\!: \{i,j\} \to \mathbb{R}$ set
  \begin{align*}
    f(x)
    \;\ldef\;
    g(i)\, h_{\cU_{M_i}(\delta), \cU_{M_j}(\delta)}(x)
    \,+\,
    g(j)\, h_{\cU_{M_j}(\delta), \cU_{M_i}(\delta)}(x),
  \end{align*}
  $\cU_{M_i}(\delta) \equiv \cU_{M_i}(\delta, M_j)$ and $\cU_{M_j}(\delta) \equiv \cU_{M_j}(\delta, M_i)$ are the $\delta$-neighborhoods of $M_i$ and $M_j$ as defined in \eqref{eq:def:neighborhood}.  Then, by Lemma~\ref{lem:harm:neighborhood},
  \begin{align*}
    \cE(f)
    \;=\;
    \big(g(i) - g(j)\big)^2\, 
    \capacity(\cU_{M_i}(\delta), \cU_{M_j}(\delta))
    \overset{\eqref{eq:harm:neighborhood:cap}}{\;\leq\;}
    \big(g(i) - g(j)\big)^2\, 
    \frac{\capacity(M_i, M_j)}{1-2\delta}.
  \end{align*}
  Further, notice that $a \ln a - a\ln b - a + b \geq 0$ for all $a,b > 0$. Thus,
  \begin{align*}
    \Ent_{\mu}[f^2]
    &\;=\;
    \min_{c > 0}\, \Mean_{\mu}\pra*{f^2 \ln f^2 - f^2 \ln c - f^2 + c}
    \\[1ex]
    &\;\geq\;
    \min_{c > 0}\,
    \Mean_{\mu}\!\big[
      f^2 \ln f^2 - f^2 \ln c - f^2 + c)\, 
      \one_{\cU_{M_i}(\delta) \cup \cU_{M_j}(\delta)}
    \big]
    \\
    &\overset{\eqref{eq:harm:neighborhood:mu}}{\;\geq\;}
    \big(\mu[\cS_i] + \mu[\cS_j]\big)\, \big(1 - \varrho \delta^{-1}\big)\;
    \Ent_{\mathrm{Ber}(p)}[g^2],
  \end{align*}
  where $\mathrm{Ber}(p) \in \cP(\{i,j\})$ denotes the Bernoulli measure on the two-point space $\{i, j\}$ with success probability $p = 1 - q = \mu[\cS_i] / (\mu[\cS_i] + \mu[\cS_j])$.  This yields
  \begin{align*}
    C_{\LSI}
    \;\geq\;
    \frac{\Ent_{\mu}[f^2]}{\cE(f)}
    \;=\;
    \frac{\mu[\cS_{M_i}] + \mu[\cS_{M_j}]}{\capacity(M_i, M_j)}\,
    (1-2\delta)\, \big(1 - \varrho \delta^{-1}\big)\;
    \frac{\Ent_{\mathrm{Ber}(p)}[g^2]}{(g(i) - g(j))^2},
  \end{align*}
  for any $g\!:\{i,j\} \to \mathbb{R}$ with $g(i) \ne g(j)$.  Recall that the logarithmic Sobolev constant for Bernoulli measures is explicitly known and given by
  \begin{align*}
    \sup\bigg\{
      \frac{\Ent_{\mathrm{Ber}(p)}[g^2]}{(g(i) - g(j))^2}
      \;:\;
      g(i)\ne g(j)
    \bigg\}
    \;=\;
    \frac{pq}{\Lambda(p, q)}
    \;=\;
    \frac{\mu[\cS_i]\, \mu[\cS_j]}{\Lambda(\mu[\cS_i], \mu[\cS_j])}\,
    \big(\mu[\cS_i] + \mu[\cS_j]\big).
  \end{align*}
  This was found in \cite{HY95} and independently in \cite{DSC96}.  Thus, by choosing $\delta = \sqrt{\varrho}$, \eqref{eq:thm:LSI:constants:lb} follows.
  
  Let us now address the upper bound.  First, since $\Lambda(\mu[\cS_i], \mu[\cS_j]) \leq 1$ we deduce from Lemma~\ref{lemma:local:LSI} by following similar arguments as in the proof of Theorem~\ref{thm:PI:metastable} that
  \begin{align*}
    &\sum_{i=1}^K \mu[\cS_i]\, \Ent_{\mu_i}\pra*{\Mean[f^2 \,|\, \cF]}
    \overset{\eqref{eq:local:LSI:metastable}}{\;\leq\;}
    \! c C_{\textup{mass}}\, C_{\PI, \cM}\, \varrho
    \!\sum_{\substack{i,j=1 \\ i \ne j}}^K
    \frac{\mu[\cS_i]\, \mu[\cS_j]}{\Lambda(\mu[\cS_i], \mu[\cS_j])}\;
    \frac{\cE(f)}{\capacity(M_i, M_j)}.
  \end{align*}
  On the other hand, by \cite[Corollary~2.8]{MS14}, we have that
  \begin{align*}
    &\Ent_{\mu}\!\big[\!\Mean_{\mu}[f^2 \,|\, \cG]\big]
    \;\leq\;
    \frac{1}{2}\, \sum_{\substack{i,j=1\\i \ne j}}^K
    \frac{\mu[\cS_i]\, \mu[\cS_j]}{\Lambda(\mu[\cS_i], \mu[\cS_j])}\,
    \bigg(
      {\textstyle \sum\limits_{k \in \{i,j\}}}\!
      \var_{\mu_k}[f] 
      \,+\, \Big(\!\Mean_{\mu_i}[f] - \Mean_{\mu_j}[f]\Big)^{\!2}
    \bigg).
  \end{align*}
  In view of the projection property of the conditional expectation together with \eqref{eq:ass:PI:metastable:sets} and \eqref{eq:local:PI:metastable}
  \begin{align*}
    \sum_{k \in \{i,j\}}\!\var_{\mu_k}[f] 
    &\;=\!
    \sum_{k \in \{i,j\}}
    \bra*{
      \mu_k[M_k]\, \var_{\mu_{M_k}}[f] 
      + 
      \var_{\mu_k}\pra*{\Mean_{\mu}[f \,|\, \cF]} 
    }
    \;\leq\;
    c C_{\PI, \cM}\,
    \frac{\varrho\, \cE(f)}{\capacity(M_i, M_j)}.
  \end{align*}
  Thus, \eqref{eq:thm:LSI:constants:ub} follows up to the additive constant $C_{\LSI,\cM}$ by combining the estimates above and using Lemma~\ref{lemma:mean:diff}.  To bound the additive error term $C_{\LSI,\cM}$, notice that
  \begin{align*}
    C_{\LSI,\cM}
    \overset{\eqref{eq:HitProbCap}}{\;\leq\;}
    C_{\LSI,\cM}\, \varrho\,
    \sum_{\substack{i,j=1 \\ i \ne j}}^K 
    \frac{\mu[\cS_i] \, \mu[\cS_j]}{\Lambda(\mu[\cS_i], \mu[cS_j])}
    \frac{1}{\capacity(M_i, M_j)},
  \end{align*}
  where we used that $\Lambda(\mu[\cS_i], \mu[\cS_j]) \leq 1$.  This allows us to absorb the additive constant $C_{\LSI,\cM}$ into the right-hand side of \eqref{eq:thm:LSI:constants:ub}.
\end{proof}
\begin{proof}[Proof of Theorem~\ref{thm:main:PI}]
  For $K=2$ \eqref{eq:main:PI} and \eqref{eq:main:LSI} follow directly from Theorem~\ref{thm:PI:metastable} and Theorem~\ref{thm:LSI:metastable}.
\end{proof}

\section{Random field Curie--Weiss model}
\label{sec:RFCW}
The proof of Theorem~\ref{CW:thm:EyringKramers} follows from Theorems~\ref{thm:PI:metastable} and~\ref{thm:LSI:metastable} after having established Propositions~\ref{CW:prop:rho:metastable}, \ref{CW:prop:meta_sets:PI_LSI}, and \ref{CW:prop:varLEBD} in each of the three following sections. 
\subsection{Verification of \texorpdfstring{$\varrho$-}{}metastability}
%
%
%
%
In view of \eqref{eq:equi:measure}, estimates of hitting probabilities can be deduced from upper and lower bounds of the corresponding capacities.  Based on the  Dirichlet principle and a comparison argument for Dirichlet forms, our strategy is to compare the microscopic with suitable mesoscopic capacities via a coarse-graining. One direction of the comparison follows immediately from the Dirichlet principle. In this way, we can utilize the estimates on capacities contained in \cite{BBI09}.

For disjoint subsets $\bfA, \bfB \subset \Gamma^n$ set $A = \order^{-1}(\bfA)$ and $B = \order^{-1}(\bfB)$.  Then, the \emph{microscopic capacity} $\capacity(A, B)$ is bounded from above by the \emph{mesoscopic capacity} $\bcapacity(\bfA, \bfB)$
\begin{align}\label{eq:def:meso:cap}
  \capacity(A, B)
  &\overset{\eqref{eq:DirichletPrinciple}}{\;\leq\;}
  \inf_{g \in \mathcal{H}_{\bfA, \bfB}} \cE(g \circ \order) 
  \nonumber\\[.5ex]
  &\;=\;
  \inf_{g \in \mathcal{H}_{\bfA, \bfB}}
  \frac{1}{2}\, 
  \sum_{\bfx,\bfy \in \Gamma^n} \magmu(\bfx)\, \bfr(\bfx, \bfy)\, 
  \big(g(\bfx) - g(\bfy)\big)^2
  \nonumber\\[.5ex]
  &\;\rdef\;
  \bcapacity(\bfA, \bfB),
\end{align}
where
\begin{align}\label{eq:meso:rates}
  \bfr(\bfx, \bfy)
  \;\ldef\;
  \frac{1}{\magmu(\bfx)} 
  \sum_{\sigma \in \order^{-1}(\bfx)} \mu(\sigma)\, 
  \sum_{\sigma' \in \order^{-1}(\bfy)} p(\sigma, \sigma')
\end{align}
and $\mathcal{H}_{\bfA, \bfB} \ldef \set{g\!: \Gamma^n \to [0,1] : g|_{\bfA} = 1, g|_{\bfB} = 0}$.  Notice that the \emph{mesoscopic transition probabilities} $(\bfr(\bfx, \bfy) : \bfx, \bfy \in \Gamma^n)$ are reversible with respect to $\magmu$.  Recall that the metastable sets $M_1, M_2$ are defined as preimages under $\order$ of particular minima $\bfm_1, \bfm_2$ of $F$.  Hence, an upper bound on the numerator in \eqref{intro:eq:cap:meta_sets} follows from an upper bound on $\bcapacity(\bfm_1, \bfm_2)$.

In the following lemma we show that the denominator in \eqref{intro:eq:cap:meta_sets} can also be expressed in terms of mesoscopic capacities.
\begin{lem}\label{lem:micro:meso:lb}
  For $n \geq 1$ let $\bfB \subset \Gamma^n$ be non-empty and set $B = \order^{-1}(\bfB)$.  Further, define $\varepsilon(n) \ldef 2h_{\infty}/n$.  Then, for any $A \subset \cS \setminus B$ and $N \geq n$,
  \begin{align}\label{eq:micro:meso:lb}
    \prob_{\mu_A}\pra*{\tau_B < \tau_A}
    \;\geq\;
    |\Gamma^n|^{-1}\, \me^{-4 \beta \varepsilon(n) (2N + 1)}\,
    \min_{\bfx \in \Gamma^n \setminus \bfB} 
    \frac{\bcapacity(\bfx, \bfB)}{\magmu(\bfx)}.
  \end{align}
\end{lem}
\begin{proof}
  Notice that the image process $(\order(\sigma(t)) : t \geq 0)$ on $\Gamma^n$ is in general not Markovian.  For that reason, we introduce an additional Markov chain on $\cS$ with the property that its image under $\order$ is Markov and the corresponding Dirichlet form is comparable to the original one with a controllable error provided $n$ is chosen large enough.
 
  For fixed $n \geq 1$ let $(\overbar{\sigma}(t) : t \geq 0)$ be a Markov chain in discrete-time on $\cS$ with transition probabilities
  \begin{align*}
    \overbar{p}(\sigma,\sigma')
    \;\ldef\;
    \frac{1}{N}\, 
    \exp\bra[\big]{- \beta N \pra{ E(\order(\sigma')) - E(\order(\sigma))}_+}\,
    \one_{\abs{\sigma - \sigma'}_1=2}
  \end{align*}
  and $\overbar{p}(\sigma, \sigma) = 1 - \sum_{\sigma' \in \cS} \overbar{p}(\sigma, \sigma')$, which is reversible with respect to the random Gibbs measure
  \begin{align*}
    \overbar{\mu}(\sigma)
    \;\ldef\;
    \overbar{Z}^{-1}\, 
    \exp\bra[\big]{-\beta N E\bra[\big]{\order(\sigma)}}\, 2^{-N},
    \qquad \sigma \in \cS.
  \end{align*}
  Let us denote the law of this process by $\overbar{\prob}$, and we write $\overbar{\capacity}(A,B)$ for the corresponding capacities. Likewise, let $\overbar{\magmu} \ldef \overbar{\mu} \circ \order^{-1}$, and define $\overbar{\bfr}$ analog to \eqref{eq:meso:rates}.  Note that
  \begin{align}\label{eq:comparison:mu:p}
    \me^{-2\beta \varepsilon(n) N}
    \;\leq\;
    \frac{\overbar{\mu}(\sigma)}{\mu(\sigma)}
    \;\leq\;
    \me^{2\beta \varepsilon(n) N}
    \qquad \text{and} \qquad
    \me^{-2\beta \varepsilon(n)}
    \;\leq\;
    \frac{\overbar{p}(\sigma, \sigma')}{p(\sigma, \sigma')}
    \;\leq\;
    \me^{2\beta \varepsilon(n)}
  \end{align}
  for any $\sigma, \sigma' \in \cS$.  On the other hand, for any $\bfx, \bfy \in \Gamma^n$ it holds that $\overbar{p}(\sigma, \order^{-1}(\bfy)) = \overbar{p}(\sigma', \order^{-1}(\bfy))$ for every $\sigma, \sigma' \in \order^{-1}(\bfx)$. This ensures (see e.g.\ \cite{BR58}) that the Markov chain $(\overbar{\sigma}(t) : t \geq 0)$ is exactly lumpable, that is, $(\order(\overbar{\sigma}(t)) : t \geq 0)$ is a Markov process on $\Gamma^n$ with transition probabilities $\overbar\bfr$ and reversible measure $\overbar\magmu$.  As a corollary of \cite[Theorem 9.7]{BdH15} we obtain that, for $A = \order^{-1}(\bfa)$ and $B = \order^{-1}(\bfB)$ with $\set{\bfa}, \bfB \subset \Gamma^n$ disjoint,
  \begin{align}\label{eq:exact:lumpable}
    \overbar{\prob}_{\sigma}\pra*{\tau_B < \tau_A}
    \;=\;
    \overbar{\prob}_{\sigma'}\pra*{\tau_B < \tau_A}
    \qquad \forall\, \sigma, \sigma' \in A.
  \end{align}
  In particular, $\overbar{\capacity}(A, B) = \overbar{\bcapacity}(\bfa, \bfB)$. By using a comparison of Dirichlet forms, we deduce from \eqref{eq:comparison:mu:p} that, for any $A, B \subset \cS$,
  \begin{align}\label{eq:comparison:cap}
    \me^{-2 \beta \varepsilon(n) (N+1)}
    \;\leq\;
    \frac{\capacity(A, B)}{\overbar{\capacity}(A, B)}
    \qquad \text{and} \qquad
    \me^{-2 \beta \varepsilon(n) (N+1)}
    \;\leq\;
    \frac{\overbar{\bcapacity}(\bfa, \bfB)}{\bcapacity(\bfa, \bfB)}.
  \end{align}

  Let us now address the proof of \eqref{eq:micro:meso:lb}.  For a given $\emptyset \ne \bfB \subset \Gamma^n$ set $B = \order^{-1}(\bfB)$ and let $A \subset \cS \setminus B$ be arbitrary.  Then, we can find $\set{\bfx_k : k=1, \ldots, L} \subset \Gamma^n$ such that
  \begin{align*}
    A \cap \order^{-1}(\bfx_k) 
    \;\ne\; 
    \emptyset 
    \qquad\text{and}\qquad 
    A \;\subset\; \bigcup_{k=1}^L \order^{-1}(\bfx_k).
  \end{align*}
  We set $X_k \ldef \order^{-1}(\bfx_k)$ and $A_k \ldef A \cap X_k$ for $k \in \set{1, \ldots, L}$ to lighten notation.  Since
  \begin{align*}
    \overbar{\capacity}(A_k, B)
    \;\geq\;
    \sum_{\sigma \in A_k} \overbar{\mu}(\sigma)\, 
    \overbar{\prob}_{\sigma}\pra*{\tau_{B} < \tau_{X_k}}
    \overset{\eqref{eq:exact:lumpable}}{\;=\;}
    \overbar{\mu}[A_k]\;
    \frac{\overbar{\capacity}(X_k, B)}{\overbar{\mu}[X_k]}
    \;=\;
    \overbar{\mu}[A_k]\; 
    \frac{\overbar{\bcapacity}(\bfx_k, \bfB)}{\overbar{\magmu}(\bfx_k)}
  \end{align*}
  an application of \eqref{eq:comparison:cap} and \eqref{eq:comparison:mu:p} yields
  \begin{align}\label{eq:est:cap(Ak,B)}
    \capacity(A_k, B)
    \;\geq\;
    \me^{-2 \beta \varepsilon(n) (N+1)}\; \overbar{\mu}[A_k]\; 
    \frac{\overbar{\bcapacity}(\bfx_k, \bfB)}{\overbar{\magmu}(\bfx_k)}
    \;\geq\;
    \me^{-4 \beta \varepsilon(n) (2N+1)}\; \mu[A_k]\;
    \min_{\bfx \in \Gamma^n\setminus \bfB}\!
    \frac{\bcapacity(\bfx, \bfB)}{\magmu(\bfx)}.
  \end{align}
  Thus,
  \begin{align*}
    \capacity(A, B)
    \overset{\eqref{eq:cap:monotonicity}}{\;\geq\;}
    \frac{1}{L}\, \sum_{k=1}^L\, \capacity(A_k, B)
    &\overset{\eqref{eq:est:cap(Ak,B)}}{\;\geq\;}
    \frac{1}{L}\, \me^{-4 \beta \varepsilon(n) (2N+1)}\, \mu[A]\;
    \min_{\bfx \in \Gamma^n\setminus \bfB}\; 
    \frac{\bcapacity(\bfx, \bfB)}{\magmu(\bfx)}.
  \end{align*}
  Since $L \leq |\Gamma^n|$, the assertion \eqref{eq:micro:meso:lb} follows.
\end{proof}
\begin{proof}[Proof of Proposition~\ref{CW:prop:rho:metastable}]
  Let $n \geq 1$ and $M_k = \order^{-1}(\bfm_k)$ with $\bfm_k\in \Gamma^n$ for $k=1,\dots, K$ the local minima of $F$ as in the assumptions of Proposition~\ref{CW:prop:rho:metastable} with decreasing energy barriers $\set{\Delta_k : k \in \set{1, \ldots, K}}$ as defined in \eqref{CW:def:Delta}.  Then, by \cite[Proposition~4.5, Corollary~4.6 and Proposition~3.1]{BBI09} there exists $C < \infty$ such that, $\prob^h$-a.s., for any $N \geq N_0(h) \vee N_1(h)$ and all $1< k \leq K$ we have
  \begin{align*}
    \prob_{\mu_{M_k}}\pra*{\tau_{\bigcup_{i=1}^{k-1}M_i}  < \tau_{M_k}}
    \overset{\eqref{eq:def:meso:cap}}{\;\leq\;}
    \frac{\bcapacity(\bfM_{k-1}, \bfm_k)}{\magmu(\bfm_k)}
    \;\leq\;
    C\, N^n\, \me^{-\beta N \Delta_{k-1}}.
  \end{align*}
  On the hand, for any $A \subset \cS \setminus \bigcup_{i=1}^{K} M_i$ Lemma~\ref{lem:micro:meso:lb} implies that
  \begin{align*}
    \prob_{\mu_A}\pra*{\tau_{\bigcup_{i=1}^{K} M_i} < \tau_A}
    \;\geq\;
    |\Gamma^n|^{-1}\, \me^{-4 \beta \varepsilon(n) (2N + 1)}\,
    \min_{\bfx \in \Gamma^n \setminus \bfM} 
    \frac{\bcapacity(\bfx, \bfM)}{\magmu(\bfx)},
  \end{align*}
  where $\bfM := \bigcup_{i=1}^{K} \bfm_i$.  For any $\bfx \in \Gamma^n \setminus \bfM$ a lower bound on the mesoscopic capacity $\bcapacity(\bfx, \bfM)$ follows by standard comparison with the explicitly computable capacity $\bcapacity_{\boldsymbol{\gamma}}(\bfx, \bfM)$ of a one-dimensional path connecting $\bfx$ with $\bfM$.  For $\bfx \not\in \bfM$ there exists a cycle-free mesoscopic path $\boldsymbol{\gamma} = (\boldsymbol{\gamma}_0, \ldots, \boldsymbol{\gamma}_k)$ in $\Gamma^n$ such that $\boldsymbol{\gamma}_0 = \bfx$, $\boldsymbol{\gamma}_k \in \bfM$, $\bfr(\boldsymbol{\gamma}_i, \boldsymbol{\gamma}_{i+1}) > 0$ for all $i \in \set{0, \ldots, k-1}$ and $F(\boldsymbol{\gamma}_i)\leq F(\bfx) + O(1/N)$. This path can be obtained from the best-lattice approximation of the continuous gradient flow trajectory $\dot \bfx(t) = - \nabla F(\bfx(t))$ with $\bfx(0)= \bfx$. In particular, by \cite[Proposition~3.1]{BBI09}, there exists $C < \infty$ such that, $\prob^h$-a.s., for any $N \geq N_0(h) \vee N_1(h)$
  \begin{align}\label{eq:mag:path}
    \frac{\magmu(\bfx)}{\magmu(\boldsymbol{\gamma}_i)}
    \;\leq\;
    C N^n 
    \qquad \forall\, i \in \set{0, \ldots, k}.
  \end{align}
  Hence,
  \begin{align*}
    \frac{\bcapacity(\bfx, \bfM)}{\magmu(\bfx)}
    \;\geq\;
    \frac{\bcapacity_{\boldsymbol{\gamma}}(\bfx, \bfM)}{\magmu(\bfx)}
    \;=\;
    \bra*{
      \sum_{i=0}^{k-1} 
      \frac{\magmu(\bfx)}
      {
        \magmu(\boldsymbol{\gamma}_i)\, 
        \bfr(\boldsymbol{\gamma}_i,\boldsymbol{\gamma}_{i+1} )
      }
    }^{\!-1}
    \;\geq\;
    \frac{\me^{-\beta (2+h_{\infty})}}{k C N^{n+1}}\, , 
  \end{align*}
  where we used in the last step \eqref{eq:mag:path} and the fact that $\bfr(\bfz, \bfz') \geq N^{-1}\me^{-2\beta(2+h_{\infty})}$ for any $\bfz, \bfz' \in \Gamma^n$ with $\bfr(\bfz, \bfz') > 0$.  Since the path $\boldsymbol{\gamma}$ is assumed to be cycle-free, its length is bounded by $|\Gamma^n|$, which itself is bounded by $N^n$.  Thus, by combining the estimates above and using the fact that by Assumption~\ref{CW:ass:law} $\Delta_{K-1}>0$, we can absorb the subexponential prefactors. That is, $\prob^h$-a.s., for any $c_1 \in (0, \Delta_{K-1})$ there exists $n_0(c_1)$ such that for all $n \geq n_0(c_1)$ the following holds: there exists $\overbar N < \infty$ such that for every $N \geq N_0(h) \vee N_1(h) \vee \overbar N$,
  \begin{align*}
    K\; \frac{\max_{M \in \set{M_1,\dots,M_K}}
      \prob_{\mu_M}\!\big[
        \tau_{{\scriptscriptstyle \bigcup_{i=1}^K} M_i \setminus M} < \tau_M
      \big]}
    {
      \min_{A \,\subset \cS \setminus \bigcup_{i=1}^K M_i}
      \prob_{\mu_A}\!\big[\tau_{\bigcup_{i=1}^K M_i} < \tau_A\big]
    }
    \;\leq\;
    \me^{-\beta c_1 N}
    \;\rdef\;
    \varrho.
  \end{align*}
  This completes the proof.
\end{proof}
\begin{proof}[Proof of Proposition~\ref{CW:cor:rho:metastable}]
  The proof is very similar to the one presented above. However, one has to be more careful in the construction of the path for~\eqref{eq:mag:path}, which is replaced by the bound
  \begin{align*}
    \frac{\magmu(\bfx)}{\magmu(\boldsymbol{\gamma}_i)}
    \;\leq\;
    C N^n  \, \me^{\beta N \Delta_2},
    \qquad \forall\, i \in \set{0, \ldots, k}.
  \end{align*}
  The mesoscopic path $\boldsymbol{\gamma}$ is now constructed such that it passes through the communication height $\Phi(\bfx, \set{\bfm_1,\bfm_2}) = \max_{i \in \set{0, \ldots, k}} F(\boldsymbol{\gamma}_i)$, where $\Phi(\bfx, \bfM)$ is defined in \eqref{CW:def:CommHeight}.  The definition of $\Delta_2$ and the ordering of $\bfm_1$ and $\bfm_2$ ensures that $\Phi(\bfx, \set{\bfm_1,\bfm_2})\leq \Delta_2$ for all $\bfx\in \Gamma^n$. Finally, the non-degeneracy Assumption~\ref{CW:ass:dominance} implies that the subexponential factors can be absorbed. 
\end{proof}

\subsection{Regularity estimates via coupling arguments}
%
%
%
%
The main objective in this subsection is to show that Assumption~\ref{ass:regularity} is satisfied in the random field Curie--Weiss model.
\begin{prop}\label{CW:prop:varLEBD}
  Let the assumptions of Proposition~\ref{CW:cor:rho:metastable} be satisfied.  Then, $\prob^h$-a.s., for any $c_2 \in (0, c_1)$ there exists $n_1 \equiv n_1(c_1,c_2, \beta, h_{\infty})$ such that for any $n \geq n_0 \vee n_1$, for any $i \ne j \in \set{1,2}$ and $N \geq N_0(h) \vee N_1(h)$,
  \begin{align}\label{CW:eq:regularity}
    \var_{\mu_{M_i}}\pra*{\frac{\nu_{M_i,M_j}}{\mu_{M_i}}}
    \;\leq\;
    \frac{\eta\,\mu[M_i]}{\capacity(M_i, M_j)} \qquad\text{with}\qquad \eta = \me^{-c_2 \beta N}.
  \end{align}
  Moreover, if the external field $h$ takes only finite many discrete values then \eqref{CW:eq:regularity} holds with $\eta = 0$.
\end{prop}
Let us emphasize, that although the bound~\eqref{CW:eq:regularity} can in principle be deduced from~\cite[Proposition 6.12]{BBI09}, we include a proof of Proposition~\ref{CW:prop:varLEBD} that is based on a coupling construction. Coupling methods were first applied in the analysis of the classical Curie--Weiss model in \cite{LLP10}.  Later, this technique was adapted in~\cite[Section 3]{BBI12} to obtain pointwise estimates on the mean hitting time for a certain class of general spin models. This approach was simplified and generalized to Potts models in~\cite{Sl12}. Here, we give a streamlined presentation of~\cite{BBI12} thanks to the simplification of~\cite{Sl12} in the setting of the random field Curie--Weiss model.

We are going to construct a coupling $(\sigma(t), \varsigma(t) : t \in \N_0)$ such that $\sigma(t)$ and $\varsigma(t)$ are two versions of the Glauber dynamics of the random field Curie--Weiss model.  Hereby, we choose $\sigma(0) \in \order^{-1}(\bfx)$ and $\varsigma(0) \in \order^{-1}(\bfx)$, that is, the initial conditions have the same mesoscopic magnetization $\bfx \in \Gamma^n$.  We use that the Glauber dynamics of the Curie--Weiss model defined via~\eqref{CW:Glauber} can be implemented by first choosing a site $i \in \set{1, \dots, N}$ uniform at random and then flipping the spin at this site $i$ with probability given by the distribution $\nu_{i, \sigma}$ in the following way
\begin{align*}
  \nu_{i, \sigma}[-\sigma_i]
  \;\ldef\;
  N p(\sigma, \sigma^i)
  \qquad \text{and} \qquad
  \nu_{i, \sigma}[+1] + \nu_{i, \sigma}[-1] \;=\; 1,
\end{align*}
where $\sigma_j^i = \sigma_j$ for all $j \ne i$ and $\sigma_i^i = -\sigma_i$.  Note that for any $\sigma, \varsigma\in \order^{-1}(\bfx)$ and $i,j \in \set{1,\dots,N}$ such that $\order(\sigma^i) = \order(\varsigma^j)$, the estimate~\eqref{eq:comparison:mu:p} implies that 
\begin{align}\label{CW:e:comp:nu}
  \me^{-4\beta \varepsilon(n)} \nu_{i, \sigma}[-\sigma_i]
  \;\leq\;
  \nu_{j,\varsigma}[-\varsigma_j] .
\end{align}
The first objective is to couple the probability distributions $\nu_{i,\sigma}$ and $\nu_{j,\varsigma}$ for $\sigma, \varsigma\in \order^{-1}(\bfx)$ with $i,j$ chosen such that $\sigma_i = \varsigma_j$.  In view of~\eqref{CW:e:comp:nu}, the coupling can be constructed in such a way that we can decide in advance by tossing a coin whether both chains maintain the property of having the same mesoscopic value after the coupling step.

The actual construction of the coupling is a modification of the optimal coupling result on finite point spaces introduced in \cite[Proposition 4.7]{LPW06}. The constant $\me^{-4 \beta \eps(n)}$ from~\eqref{CW:e:comp:nu} will play the role of $\delta$, when we apply the following Lemma~\ref{lem:RFCW:OptCoupl}.
\begin{lem}[{Optimal coupling~\cite[Lemma 2.3]{Sl12}}]\label{lem:RFCW:OptCoupl}
  Let $\nu, \nu' \in \cP(\set{-1,1})$ and suppose that there exists $\delta\in (0,1)$ such that $\delta \nu(s) \leq \nu'(s)$ for $s \in \set{-1,1}$. Then, there exists an optimal coupling $(X,X')$ of $\nu$ and $\nu'$ with the additional property that for a Bernoulli-$\delta$-distributed random variable $V$ independent of $X$ it holds that
  \begin{align*}
    \prob\pra*{ X' = s' \mid V = 1, X = s } 
    \;=\; 
    \one_{s}(s')  
    \qquad\text{for } s, s' \in \set{-1,1}.
  \end{align*}
\end{lem}
Therewith, we are able to describe the coupling construction. Let $T>0$ and $M>0$ and choose a family $(V_i : i \in \set{1, \ldots, M})$ of i.i.d.~Bernoulli variables with
\begin{align*}
  \prob\pra{ V_i = 1} 
  \;=\; 
  1 - \prob\pra{ V_i = 0} 
  \;=\; 
  \me^{-4\beta \eps(n)} .
\end{align*}
The coupling is initialized with $\sigma(0) = \sigma$, $\varsigma(0) = \varsigma$, $M_0 = 0$ and $\xi = 0$.
\begin{algorithmic}
  \For{$t = 0, 1, \dots, T-1$}
    \If{$\xi = 0$ and $M_t < M$}
      \State{%
        Choose $i$ uniform at random in $\set{1, \dots, N}$ and
        set $I_t = i$.
      }
      \If{$\sigma_{i}(t) = \varsigma_{i}(t)$}
        \State{%
          Choose $s \in \set{-1, 1}$ at random according to $\nu_{i, \sigma}$ and
          set
          \begin{align*}
            \mspace{64mu}
            \sigma_j(t+1)
            \;=\;
            \begin{cases}
              \sigma_j(t), &j \ne i\\
              s, &j=i
            \end{cases}
            \quad \text{and} \quad
            \varsigma(t+1)
            \;=\;
            \begin{cases}
              \varsigma_j(t), &j\ne i\\
              s, &j=i
            \end{cases}.
          \end{align*}
        }
        \State{Set $M_{t+1} = M_t$.}
      \Else
        \State{Let $\ell$ be such that $i \in \Lambda_{\ell}$.}
        \State{%
          Choose $j$ uniform at random in $\set{j \in \Lambda_{\ell} : \varsigma_j \ne \sigma_j \text{ and } \varsigma_j = \sigma_i}.$
        }
        \State{%
          Apply Lemma~\ref{lem:RFCW:OptCoupl} to the distributions $\nu_{i, \sigma}$ and $\nu_{j, \varsigma}$, where $V_{M_t}$ decides
        }
        \State{if both chains maintain the same mesoscopic value.}
        \State{Set $M_{t+1} = M_t + 1$.}
        \If{$V_{M_t}=0$}
          \State{Set $\xi=1$.}
        \EndIf
      \EndIf
    \Else
      \State{
        Use the independent coupling to update $\sigma(t)$ and $\varsigma(t)$.
      }
    \EndIf
  \EndFor
\end{algorithmic}
\begin{lem}[Coupling property]
  The joint probability measure $\prob_{\sigma, \varsigma}$ of the processes $\bra*{\bra{\sigma(t)}, \bra{\varsigma(t)}, \bra{V_t} : t \in \set{1, \ldots, T}}$ obtained from the construction above is a coupling of two versions of the random field Curie--Weiss model started in $\sigma$ and $\varsigma$, respectively.
\end{lem}
\begin{proof}
  As soon as $\xi=1$ or $M_t \geq M$ for some $t < T$, both chains evolve independently.  Hence, the assertion is immediate.  For $\xi = 0$ and $M_t < M$, by construction, $i$ is chosen uniform at random among $\set{1, \ldots, N}$.  Then, in the case $\sigma_i = \varsigma_i$ it follows that $\nu_{i,\sigma} = \nu_{i,\varsigma}$, whereas in the other case Lemma~\ref{lem:RFCW:OptCoupl} ensures the coupling property.
\end{proof}
The coupling construction ensures that, once $\varsigma(t)$ and $\sigma(t)$ have merged, they evolve together until time $T$. Hence, we call the event $\set{\sigma(t) = \varsigma(t)}$ a successful coupling.  Since conditioning on this event may distort the statistical properties of the paths $\varsigma$, we will introduce two independent subevents which are sufficient to ensure a merging of the processes until time $T$.
\begin{lem}\label{CW:lem:AB_event}
  For any value $T$ and $M$, define the following two events:
  \begin{enumerate}[label=(\roman{enumi})]
  \item The event that all Bernoulli variables $V_i$ are equal to $1$, that is,
    \begin{align*}
      \cA \,\ldef\; \set*{ V_i = 1 \,:\, i \in \set{0, \ldots, M-1}} .
    \end{align*}
  \item The stopping time $\mathfrak{t}_i$ is the first time the $i$-th spin flips and $\mathfrak{t}$ is first time all coordinates of $\sigma$ have been flipped, that is,
    \begin{align*}
      \mathfrak{t}_i 
      \;=\; 
      \inf\set*{t \geq 0 \,:\, \sigma(t+1) = - \sigma(0)} 
      \qquad\text{and}\qquad 
      \mathfrak{t} \,\ldef\; \max_{i\in \set{1, \ldots, N}}  \mathfrak{t}_i.
    \end{align*}
    Therewith, the random variable
    \begin{align*}
      \mathcal{N}
      \;\ldef\;
      \sum_{i=1}^{N}\, \sum_{t=0}^{\mathfrak{t}_i} \one_{I_t = i}
    \end{align*}
    represents the total number of flipping attempts until time $\mathfrak{t}$. The event $\cB$, only depending on $\set{\sigma(t) : t \in \set{0, \ldots, T}}$, is defined for any $B \subset \cS$ by
    \begin{align*}
      \cB
      \;\ldef\;
      \set{\mathfrak{t} \leq \tau_B} \,\cap\, \set{\cN \leq M}.
    \end{align*}
  \end{enumerate}
  Then, it holds that
  \begin{align*}
    \cA \cap \cB 
    \;\subset\; 
    \set{ \sigma(\mathfrak{t}) = \varsigma(\mathfrak{t})} .
   \end{align*}
\end{lem}
\begin{proof}
  The event $\cB$ ensures that $\sigma(t)$ has not reached the set $B$ and all its spins have flipped once. By the event $\cA$, each flipping aligns one more spin with $\varsigma(t)$, and hence we have $\sigma(\mathfrak{t}) = \varsigma(\mathfrak{t})$.
\end{proof}
By construction, we have
\begin{align}\label{CW:coupling:probA}
  \prob_{\varsigma, \sigma}\pra{\cA} \;\leq\; \me^{-4\beta \eps(n) M}.
\end{align}
which is exponentially small in $M$.  Moreover, since 
\begin{align}\label{CW:nu:lb}
  \nu_{i,\sigma}[-\sigma_i] \;\geq\; \exp\bra[\big]{- 2 \beta \bra{1+h_\infty}}
\end{align}
for any $\sigma \in \cS$ and $i \in \set{1, \ldots, N}$, by standard large deviation estimates, we can bound the tail of the probability of the random variable $\cN$.
\begin{lem}\label{CW:lem:Nbound}
  Let $s > \alpha^{-1} \ldef \exp\bra{2\beta \bra{1+h_\infty}}$ and set $M = c_3 N$.  Then,
  \begin{align*}
    \prob_{\sigma}\pra{\cN > M} \;\leq\; \me^{-I_\alpha^{\nBer}(s-1) N} ,
  \end{align*}
  where $I_\alpha^{\nBer}$ is the rate function of the negative Bernoulli distribution with parameters~$N$ and $\alpha$, that is given by
  \begin{align}\label{CW:def:c2}
    (0, \infty) \ni s
    \;\longmapsto\;
    I_\alpha^{\nBer}(s)
    \;\ldef\;
    s \ln \frac{s}{(1+s)(1-\alpha)} - \ln \alpha - \ln (1+s)
    \;\geq\; 
    0.
  \end{align}
  In particular, $I_{\alpha}^{\nBer}$ is strictly convex on $(0, \infty)$ and $I_\alpha^{\nBer}(s-1) > 0$ for all $s > \alpha^{-1}$.
\end{lem}
\begin{proof}
  The bound \eqref{CW:nu:lb} implies that if a site is chosen uniformly at random among $\set{1, \ldots, N}$, it is flipped at least with probability $\alpha$.  Let $\bra{\omega(t) : t\in \set{0, \ldots, T}}$ be a family of independent $\mathop{\mathrm{Ber}}(\alpha)$-distributed random variables and define the negative binomial distributed random variable $\cR$ with parameters $N$ and $\alpha$ by
  \begin{align*}
    \cR
    \;\ldef\;
    \inf\set[\bigg]{ 
      s \geq 1 \,:\,  {\textstyle \sum\limits_{t=1}^s} \omega(t) = N
    } - N .
  \end{align*}
  Then, by using a straightforward coupling argument (see~\cite[Lemma 2.6]{Sl12}), we obtain that $\cN \leq \cR + N$.  Further, by using standard large deviation estimates, we find that
  \begin{align}\label{CW:est:probN}
    \prob_{\sigma}\pra*{ \cN > s N}
    \;\leq\;
    \prob\pra*{ \cR > (s -1)N}
    \;\leq\; 
    \me^{-N I_{\alpha}^{\nBer}(s-1)}.
  \end{align}
  Hereby, $I_{\alpha}^{\nBer}$ is given as the Legendre-Fenchel dual of the log-moment generating function of the negative Bernoulli distribution, that is $t \mapsto \log\bra*{\alpha / (1-(1-\alpha)e^t)}$.  The rate function $I_\alpha^{\nBer}$ is strictly convex, since $\partial_s^2 I_{\alpha}^{\nBer}(s-1) = \frac{1}{s(s+1)}$ and has its unique minimum in $s=\alpha^{-1} - 1$.  Hence, $I_{\alpha}^{\nBer}(s-1)$ is strictly positive for $s > \alpha^{-1}$.
\end{proof}
The above construction allows us to deduce the following bound on hitting probabilities of preimages of mesoscopic sets.
\begin{lem}\label{CW:lem:HitProb}
  For any $n \in \N$ and $\bfA, \bfB \subset \Gamma^n$ disjoint, set $A = \order^{-1}(\bfA)$ and $B = \order^{-1}(\bfB)$.  Further, let $\bfx \in \Gamma^n$ and choose $s > \alpha^{-1}$ according to Lemma~\ref{CW:lem:Nbound}.  Then
  \begin{align}\label{CW:est:HitProb}
    \prob_{\varsigma}\pra*{ \tau_B < \tau_A} 
    \;\geq\; 
    \me^{-4 \, \beta \, \eps(n) \, s \, N}\,
    \bra*{ \prob_{\sigma}\pra*{ \tau_B < \tau_A} \,-\, \me^{-I_{\alpha}^{\nBer}(s-1) N}},
    \qquad \forall\, \sigma, \varsigma \in \order^{-1}(\bfx)
  \end{align}
  where $I_\alpha^{\nBer}$ is given by~\eqref{CW:def:c2}.
\end{lem}
\begin{proof}
  We are going to use the above coupling construction with involved para\-meters $T = \infty$ and $M = s N$.  For that purpose, consider the following additional event:
  \begin{align*}
    \overbar{\cB}
    \;\ldef\;
    \set[\big]{\tau_B \leq \mathfrak{t}} \cap \set{\cN \leq M}.
  \end{align*}
  Notice that by Lemma~\ref{CW:lem:AB_event}, on the event $\cA \cap \cB$, we have $\sigma(\mathfrak{t}) = \varsigma(\mathfrak{t})$ and, in particular, $\tau^{\sigma}_B = \tau^{\varsigma}_B$.  Moreover, on the event $\cA \cap \cB \cap \set{\tau_B^{\varsigma} < \tau_A^{\varsigma}}$, it follows that $\tau_A^{\varsigma} = \tau_A^{\sigma}$.

  On the event $\cA \cap \overbar{\cB}$, the process $(\varsigma(t) : t \geq 0)$ reaches $B$ before time $\mathfrak{t}$.  However, by the coupling construction, we have that $\order(\sigma(t)) = \order(\varsigma(t))$ for all $t \leq \tau_B$. Since, by assumption, the sets $A, B$ are preimages of the mesoscopic sets $\bfA, \bfB$,  we conclude $\tau_B^{\sigma} = \tau_B^{\varsigma}$ and on the event $\set{\tau_B^{\varsigma} < \tau_A^{\varsigma}}$ the $\sigma$-chain can not reach $A$ before time $\tau_B^{\sigma}$.  Thus, we have
  \begin{align*}
    \prob_{\varsigma}\pra[\big]{ \tau_B < \tau_A}
    &\;\geq\; 
    \prob_{\varsigma,\sigma}\pra[\big]{
      \tau_B^{\varsigma}\, < \tau_A^{\varsigma}\,, \cA \cap \cB
    }
    \,+\, 
    \prob_{\varsigma,\sigma}\pra[\big]{
      \tau_B^{\varsigma}\, < \tau_A^{\varsigma}\,, \cA \cap \overbar{\cB}
    }
    \\[.5ex]
    &\;=\; 
    \prob_{\varsigma, \sigma}\pra[\big]{
      \tau_B^{\sigma} < \tau_A^{\sigma}, \cA \cap \cB
    } 
    \,+\, 
    \prob_{\varsigma, \sigma}\pra[\big]{ 
      \tau_B^{\sigma} < \tau_A^{\sigma}, \cA \cap \overbar{\cB}
    } 
    \\[.5ex]
    &\;\geq\; 
    \prob_{\varsigma, \sigma}[\cA]\, 
    \bra[\Big]{ 
      \prob_{\sigma}\pra[\big]{\tau_B < \tau_A} 
      \,-\,
      \prob_{\sigma}\pra[\big]{\cN > M}
    },
  \end{align*}
  which concludes the statement thanks to the estimates~\eqref{CW:coupling:probA} and~\eqref{CW:est:probN}.
\end{proof}
We are now in the position to apply the above lemma to the metastable situation of Proposition~\ref{CW:prop:varLEBD} and use the connection of hitting probabilities and the last exit biased distribution in~\eqref{eq:equi:measure}.
\begin{proof}[Proof of Proposition~\ref{CW:prop:varLEBD}]
  For an arbitrary $n \in N$ choose $\set{\bfa}, \bfB \subset \Gamma^n$ disjoint and set $A \ldef \order^{-1}(\bfa)$ and $B \ldef \order^{-1}(\bfB)$.  Then, Lemma~\ref{CW:lem:HitProb} implies that
  \begin{align*}
    e_{A,B}(\sigma)
    \overset{\eqref{eq:equi:measure}}{\;=\;\phantom{\leq}\mspace{-24mu}}
    \prob_{\sigma}\pra*{ \tau_B < \tau_A}
    \overset{\eqref{CW:est:HitProb}}{\;\leq\;}
    \me^{4\beta\, \eps(n) \, s \, N}\,
    \bra*{ \frac{\capacity(A,B)}{\mu[A]} \,+\, \me^{-I_{\alpha}^{\nBer}(s-1) N}},
    \qquad \forall\, \sigma \in A.
  \end{align*}
  Hence, in view of \eqref{eq:def:LastExitBiasedDistri} we obtain
  \begin{align*}
    \var_{\mu_A}\pra*{\frac{\nu_{A,B}}{\mu_A}}
    &\;=\;
    \frac{\mu[A]}{\capacity(A,B)}\, \Mean_{\nu_{A,B}}\pra*{e_{A,B}} - 1
    \nonumber\\[.5ex]
    &\;\leq\;
    \frac{\mu[A]}{\capacity(A,B)}\,
    \bra*{
      \bra*{\me^{4\beta \eps(n) s N} - 1}\, \frac{\capacity(A,B)}{\mu[A]}
      \,+\, 
      \me^{(4 \beta \eps(n) s -I_{\alpha}^{\nBer}(s-1)) N}
    }.
  \end{align*}
  In particular, under the assumptions of Proposition~\ref{CW:cor:rho:metastable}, we conclude from the estimate above that
  \begin{align*}
    &\bra*{\me^{4\beta \eps(n) s N} - 1}\, \frac{\capacity(M_1,M_2)}{\mu[M_2]}
    \,+\, 
    \me^{(4 \beta \eps(n) s -I_{\alpha}^{\nBer}(s-1)) N}
    \\[.5ex]
    &\mspace{144mu}\leq\;
    \me^{4\beta \eps(n) I_{\alpha}^{\nBer}(s-1) N}\, 
    \bra*{\me^{-c_1 \beta N} + \me^{-I_{\alpha}^{\nBer}(s-1) N}}.
  \end{align*}
  We have to show that the right-hand side is smaller $e^{-c_2 \beta N}$ as state in~\eqref{CW:eq:regularity}.  By exploiting the explicit definition of the rate function $I_{\alpha}^{\nBer}$ in~\eqref{CW:def:c2}, we can choose $s$ large enough such that $I_{\alpha}^{\nBer}(s-1) \geq \beta \, c_1$ with $c_1$ as in Proposition~\ref{CW:cor:rho:metastable}. Then, for any $c_2 \in (0,c_1)$, we find $n_1=n_1(c_1,c_2,h_\infty)$ such that for all $n>n_1$, it follows $4\eps(n) s = 8 h_\infty s /n < c_1-c_2$, and hence $\eta = e^{-c_2 \beta N}$ as stated in~\eqref{CW:eq:regularity}.
\end{proof}

\subsection{Local mixing estimates within metastable sets}
%
%
%
%
For the proof of Proposition~\ref{CW:prop:meta_sets:PI_LSI}, we follow \cite{MZ03} to compare the Poincar\'e constant $C_{\PI,i}$ in~\eqref{eq:ass:PI:metastable:sets} and logarithmic Sobolev constant $C_{\LSI,i}$ in~\eqref{eq:ass:LSI:metastable:sets} for any $M\in \cM$ with the ones of the Bernoulli--Laplace model.  First, we compare the variance and entropy.  Then we introduce the Bernoulli--Laplace model, for which we provide its Poincar\'e and logarithmic Sobolev constant from the literature.  Finally, by comparing the different Dirichlet forms we deduce a Poincar\'e and logarithmic Sobolev constant inside the metastable sets.
\smallskip

\noindent\textit{Step 1: Comparison of variance and entropy.} We compare the variance and entropy with respect to $\mu_{M}$ with the ones with respect to $\bar\mu_{M}$. Note that, by definition, $\bar{\mu}_{M}$ is the uniform measure on $M$. For the comparison of the variance, we use the two-sided comparison
\begin{align}\label{CW:H:Hbar}
  \overbar H(\sigma) - \eps N \leq H(\sigma) \leq \overbar H(\sigma) + \eps N
\end{align}
and obtain
\begin{align}\label{CW:metaPI:p1}
  \var_{\mu_{M}}[f] 
  \;=\;
  \inf_{a \in \R}\, \Mean_{\mu_M}\pra*{\bra*{f - a}^2}
  \;\leq\; 
  \me^{\beta \eps(n) N} \var_{\bar{\mu}_{M}}[f].
\end{align}
Similarly, for the entropy we use the fact that $b \log b - b \log a - b  + a \geq 0$ for any $a, b > 0$.  Then, by following essentially the same argument as given in~\cite{HS87},
\begin{align*}
  \Ent_{\mu_{M}}\pra*{f^2} 
  \;=\;
  \inf_{a>0}\, \Mean_{\mu_M}\pra*{ f^2 \log f^2 - f^2 \log a -f^2 +a }
  \;\leq\; 
  \me^{\beta \eps(n) N} \Ent_{\bar{\mu}_M}\pra*{f^2} .
\end{align*}
\smallskip

\noindent\textit{Step 2: Poincar\'e and logarithmic Sobolev constant of the Bernoulli--Laplace model.}
In the sequel, we introduce a dynamics on $M$.  For that purpose, denote by $\sigma^{j,k}$ the spin-exchange configuration, that is, $\sigma^{j,k}_i \ldef \sigma_i$ for $i\not\in\set{j,k}$ and $\sigma^{j,k}_j \ldef \sigma_k$ as well as $\sigma^{j,k}_k \ldef \sigma_j$.  Then, since $M=\order^{-1}(\bfm)$, we have for $\sigma \in M$, that $\sigma^{j,k}\in M$ if and only if $j,k\in \Lambda_\ell$ for some $\ell\in \set{1,\dots,n}$. Hence, for $\sigma,\sigma'\in M$ with $\abs{\sigma-\sigma'}_1 = 4$, we find $\ell\in \set{1,\dots,n}$ and $j,k\in \Lambda_{\ell}$ such that $\sigma'= \sigma^{j,k}$. Let us denote the according mesoscopic index by $\ell(\sigma,\sigma')$.  Therewith, we define the transition probabilities $\bra{\overbar{p}_{\BL}(\sigma, \sigma') : \sigma, \sigma' \in M}$ by
\begin{align*}
  \overbar{p}_{\BL}(\sigma, \sigma')
  \;\ldef\;
  \begin{cases}
    \frac{1}{\abs{\Lambda_{\ell(\sigma,\sigma')}}}, & \abs{\sigma-\sigma'}_1 = 4 \\
    0, &|\sigma-\sigma'|_1 > 4
    \\
    1-\sum\nolimits_{\eta \in M \setminus \sigma} \overbar{p}_{BL}(\sigma,\eta),
    &\sigma = \sigma'
  \end{cases} .
\end{align*}
Note that $\overbar{p}_{\BL}$ is reversible with respect to the uniform distribution $\bar{\mu}_M$.  Since $\abs{M} = \prod_{\ell=1}^n \binom{\abs{\Lambda_\ell}}{k_\ell}$ with $k_\ell = \bfm_\ell N$ and $\order^{-1}(\bfm) = M$, $\bar{\mu}_M$ is a product measure.  Moreover, the transition probabilities, $\overbar{p}_{\BL}$, are compatible with the tensorization, since any jump only occurs among two coordinates in $\Lambda_\ell$ for some $\ell \in \set{1, \dots, n}$.  Hence, if we regard the coordinates of $\sigma$ such that $\sigma_i=+1$ as particle position, then the Markov chain induced by $\overbar{p}_{\BL}$ is an exclusion process of particles in $n$ boxes of size $\set{\Lambda_\ell}_{\ell=1}^n$ such that in each box $\ell \in \set{1,\dots,n}$ the particle number is $k_\ell$.  This is the product of $n$ Bernoulli--Laplace models.  Both the spectral gap and logarithmic Sobolev constant are well known; cf.\ \cite{DS87} and \cite[Theorem 5]{LY98}.  Let us denote by $\overbar\cE_{\BL}$ the Dirichlet form corresponding to $(\bar{\mu}_M, \overbar{p}_{\BL})$. Then, by the tensorization property of the Poincar\'e and logarithmic Sobolev constant (see \cite[Lemma 3.2]{DSC96}), we obtain
\begin{align*}
  \var_{\bar{\mu}_M}[f] 
  &\;\leq\; 
  \max_{\ell \in \set{1, \dots, n}} 
  \set[\Big]{ C_{\PI,\BL(\abs{\Lambda_\ell}, k_\ell)} }\, 
  \overbar{\cE}_{\BL}(f) 
  \\
  \Ent_{\bar{\mu}_M}\pra*{f^2} 
  &\;\leq\; 
  \max_{\ell \in \set{1, \dots, n}} 
  \set[\Big]{ C_{\LSI,\BL(\abs{\Lambda_\ell},k_\ell)} }\,
  \overbar\cE_{\BL}(f),
\end{align*}
where for some universal constant $c_{\BL}>0$,
\begin{align}\label{CW:metaPI:p2}
  C_{\PI,\BL(\abs{\Lambda_\ell},k_\ell)} 
  &\;\ldef\; 
  \bra*{\frac{\abs{\Lambda_\ell}}{k_{\ell} \bra*{\abs{\Lambda_\ell} - k_\ell}}}^{-1}
  \;\leq\; 
  \frac{N}{4}
  \\
  C_{\LSI,\BL(\abs{\Lambda_\ell},k_\ell)} 
  &\;\ldef\; 
  C_{\PI,\BL(\abs{\Lambda_\ell},k_\ell)}\,
  \bra*{
    c_{\BL} \log 
    \frac{\abs{\Lambda_\ell}^2}{ k_\ell \bra*{\abs{\Lambda_\ell}- k_\ell}}
  }^{-1}
  \;\leq\; 
  \frac{N}{8 \log 2 \; c_{\BL}}.
  \nonumber
\end{align}

\noindent\textit{Step 3: Comparison of Dirichlet forms.}  Note that, for $\sigma, \sigma' \in M$, the transition probabilities $\overbar{p}_{\BL}(\sigma,\sigma')$ are not absolutely continuous with respect to $p(\sigma,\sigma')$.  For this reason, consider the auxiliary Dirichlet form $\cE_2$ associated to the two-step transition probabilities $\bra{p_2(\sigma, \sigma'): \sigma, \sigma' \in \cS}$ given by
\begin{align*}
  p_2(\sigma,\sigma')
  \;\ldef\;
  \sum_{\sigma''} p(\sigma, \sigma'')\, p(\sigma'', \sigma') .
\end{align*}
Since $\mu$ is reversible with respect to $p$, $\mu$ is also reversible with respect to $p_2$.  In addition, $p$ and $p_2$ have the same eigenvectors $\varphi_j$ and if we denote by $-1 \leq \lambda_j \leq 1$ the according eigenvalue of $p$, then the $j$th eigenvalue of $p_2$ is $\lambda_j^2$. Hence,
\begin{align}\label{CW:metaPI:p4}
  \cE_{2}(f)
  \;=\; 
  \sum_{j=1}^{\abs{\cS_n}}\, \bra{1 - \lambda_{j}^2}
  \abs[\big]{\skp{f, \varphi_j}_{\mu}}^2 
  \;\leq\; 
  2 \sum_{j=1}^{\abs{\cS_n}}\, \bra{1 - \lambda_{j}}
  \abs[\big]{\skp{f,\varphi_j}_{\mu}}^2 
  \;=\; 
  2 \cE(f) . 
\end{align}
Thus, it suffices to compare the Dirichlet forms $\cE_{\BL}$ and $\cE_2$.  For that purpose, we are left with establishing a bound on the ratio of the rates $\bar{\mu}_M(\sigma) \overbar{p}_{\BL}(\sigma,\sigma')$ and $\mu(\sigma) p_2(\sigma,\sigma')$ for $\sigma, \sigma' \in M$. For $\sigma,\sigma'\in M$ with $\abs{\sigma- \sigma'}_1 = 4$, we find $\sigma''$ such that $\abs{\sigma-\sigma''}_1=2$ and $\abs{\sigma'-\sigma''}_1=2$, which allows us to obtain a lower bound using the explicit representation of the Hamiltonian~\eqref{CW:Hamiltonian} as well as the boundedness of the external field~\eqref{CW:eq:ass:h}
\begin{align*}
  p_2(\sigma,\sigma') 
  \;\geq\; 
  p(\sigma,\sigma'')\, p(\sigma'',\sigma') 
  \;\geq\; 
  \frac{1}{N^2} \exp\bra[\big]{- 4 \beta \bra{1+h_\infty}} .
\end{align*}
Hence, the bound~\eqref{CW:H:Hbar} and the trivial estimate $\abs{\Lambda_\ell}\geq 1$ leads to
\begin{align}\label{CW:metaPI:p3}
  \frac{\bar{\mu}_M(\sigma) \overbar p_{\BL}(\sigma,\sigma')}
  {\mu(\sigma) p_2(\sigma,\sigma')} 
  &\;\leq\; 
  N^2 \exp\bra[\big]{\beta\bra{\eps(n) N + 4 + 4 h_\infty}} ,
\end{align}
which results in a comparison of the Dirichlet form $\overbar \cE_{\BL}$ and $\cE_2$ with the same constant.
\begin{prop}\label{CW:prop:meta_sets:PI_LSI}
  Assumption~\ref{ass:PI+LSI:metastable:sets} holds with constants $C_{\PI,\cM}$ and $C_{\LSI,\cM}$ satisfying
  \begin{align*}
    \max\set[\Big]{C_{\PI,\cM}, 2\log 2 c_{\BL} C_{\LSI,\cM}} 
    \;\leq\; 
    \frac{N^3}{2}\, \exp\bra[\big]{2\beta \bra*{\eps(n) N + 2 + 2h_\infty}}
  \end{align*}
  for some universal $c_{\BL}>0$.
\end{prop}
\begin{proof}[Proof Proposition~\ref{CW:prop:meta_sets:PI_LSI}]
  The conclusion follows by combining the chain of estimates for the variance
  \begin{align*}
    \var_{\mu_{M}}[f] 
    &\overset{\eqref{CW:metaPI:p1}}{\;\leq\;} 
    \me^{\beta \eps(n) N}\, \var_{\bar{\mu}_M}[f] 
    \overset{\eqref{CW:metaPI:p2}}{\;\leq\;} 
    \frac{N}{4}\, \me^{\beta \eps(n) N} \overbar\cE_{\BL}(f)
    \\
    &\overset{\eqref{CW:metaPI:p3}}{\;\leq\;} 
    \frac{N^3}{4}\, \me^{2\beta \bra*{\eps(n) N + 2 + 2h_\infty}}\, \cE_{2}(f) 
    \overset{\eqref{CW:metaPI:p4}}{\;\leq\;} 
    \frac{N^3}{2}\, \me^{2\beta \bra*{\eps(n) N + 2 + 2h_\infty}}\, \cE(f)
  \end{align*}
  and likewise for the entropy.  Notice that the final estimate on $C_{\PI,M}$ and $C_{\LSI,M}$ is independent of $M$.  Since the constants $C_{\PI,\cM}$ and $C_{\LSI,\cM}$ are convex combinations of $C_{\PI,M}$ and $C_{\LSI,M}$ for $M\in \cM$ and, the assertion follows.
\end{proof}

\appendix
\section{Young functions}
%
%
%
%
\begin{lem}[Properties of Young functions]\label{lem:YoungFunctions}
  A function $\Phi\!: [0,\infty) \to [0, \infty]$ is a Young function if it is convex, $\Phi(0)=\lim_{r\to 0} \Phi(r) =0$ and $\lim_{r \to \infty} \Phi(r)=\infty$. Then it holds that
  \begin{enumerate}[label=(\roman{enumi})]
  \item $\Phi$ is non-decreasing;
  \item its Legendre-Fenchel dual $\Psi\!: [0,\infty) \to [0,\infty]$ defined by
    \begin{align*}
      \Psi(r) \;\ldef\; \sup_{s\in [0,\infty]} \set*{ s r - \Phi(s)}
    \end{align*}
    is again a Young function;
  \item the (pseudo)-inverse of $\Phi$, defined by $\Phi^{-1}(t) \ldef \inf \big\{ s\in [0,\infty] : \Phi(s) > t \big\}$ is concave and non-decreasing.
  \end{enumerate}
\end{lem}
\begin{proof}
  (i). By convexity, it holds for any $\alpha \in (0,1)$ that
  \begin{align*}
    \Phi(\alpha t) 
    \;=\ 
    \Phi\bra*{\alpha t + (1-\alpha) \cdot 0} 
    \;\leq\; 
    \alpha\, \Phi(t) + (1-\alpha\, \Phi(0)
    \;=\;
    \alpha\, \Phi(t).
  \end{align*}
  Hence, by using additionally the non-negativity of $\Phi$ it follows for any $\alpha \in (0,1)$
  \begin{align*}
    \Phi(t) 
    \;\geq\; 
    \frac{1}{\alpha}\, \Phi(\alpha t) 
    \;\geq\; 
    \Phi(\alpha t) .
  \end{align*}

  (ii).  The convexity of $\Psi$ follows by convex duality for Legendre-Fenchel transform, since $\Phi$ is a convex function.  Since $\Phi(s)\geq 0$ for all $s$ and at least equality for $s=0$, it first follows $\Psi(r)\geq 0$ for all $r$ and in particular
  \begin{align*}
    \Psi(0) \;=\; \sup_{s\in [0,\infty]} \set*{- \Phi(s)} \;=\; 0 .
  \end{align*}
  Now, from $\lim_{r\to \infty} \Phi(r)=\infty$ and the convexity of $\Phi$ follows that, there exists $\kappa> 0$ such that $\Phi(r)\geq \kappa r$ for $r\geq R$. Hence, we get
  \begin{align*}
    \lim_{r\to 0 } \sup_{s\in [0,\infty]} \set*{ sr - \Phi(s)} 
    \;\leq\; 
    \lim_{r\to 0 } \max\set*{ \sup_{ s\in [0,R]} sr, \sup_{s \geq R } \set*{s( r- \kappa)}}
    \;=\; 
    0
  \end{align*}
  Similarly, since $\lim_{r \to \infty} \Phi(r) = 0$, it follows that $\Phi(r) \leq \eps < \infty$ for all $r\in [0,\delta]$, and hence
  \begin{align*}
    \lim_{r\to \infty} \sup_{s \in [0,\infty]}\set*{ sr -\Phi(s)} 
    \;\geq\; 
    \lim_{r\to \infty} \bra*{ \delta r - \eps} 
    \;=\; 
    \infty.
  \end{align*}

  (iii). The fact, that $\Phi^{-1}$ is non-decreasing follows immediately from its definition and that $\Phi$ is non-decreasing. Now, let $u,v \in \set*{ \Phi(s) : s \in \R, \Phi(s) < \infty}$. Then, by convexity of $\Phi$ follows for $\alpha \in (0,1)$ and $\beta = 1-\alpha$,
  \begin{align*}
    \Phi\bra*{\alpha \Phi^{-1}(u) + \beta \Phi^{-1}(v)} 
    \;\leq\; 
    \alpha \Phi\bra*{\Phi^{-1}(u)} + \beta \Phi\bra*{\Phi^{-1}(v)} 
    \;=\; 
    \alpha u + \beta v ,
  \end{align*}
  where we used that $\Phi$ is continuous on its finite support, since it is convex.  Since $\Phi^{-1}$ is non-decreasing, the inequality is preserved after applying it
  \begin{align*}
    \Phi^{-1}\bra*{\Phi\bra*{\alpha \Phi^{-1}(u) + \beta \Phi^{-1}(v)}} 
    \;\leq\; 
    \Phi^{-1} \bra*{\alpha u + \beta v} .
  \end{align*}
  Now by noting
  \begin{align*}
    \Phi^{-1}\bra*{\Phi(x)} \;=\; \inf\set*{ s: \Phi(s) > \Phi(x)} \geq x ,
  \end{align*}
  if follows that $\Phi^{-1}$ is concave on the finite range of $\Phi$. If this range is finite, then $\Phi^{-1}$ gets extended continuously as a constant and hence still concave.
\end{proof}

\subsubsection*{Acknowledgement}
The authors wish to thank Georg Menz and Anton Bovier for many fruitful discussions about metastability and related topics. A.S.\ and M.S.\ thank the FOR~718 \emph{Analysis and Stochastics in Complex Physical Systems} for fostering this collaboration.  A.S.\ acknowledges support through the CRC~1060 \emph{The Mathematics of Emergent Effects} at the University of Bonn. The authors thank the anonymous referees for the much valuable suggestions for improving the presentation of this article. 

\bibliographystyle{abbrv}
\bibliography{literature}
\end{document}